\documentclass[reqno]{amsart}

\usepackage{amsmath}
\usepackage{amsthm}
\usepackage{amsfonts}
\usepackage{amssymb}
\usepackage{enumitem}
\usepackage{hyperref}
\usepackage{bbm}
\usepackage[nocompress]{cite}

\newcommand{\indicator}[1]{\ensuremath{\mathbf{1}_{\{#1\}}}}
\newcommand{\oindicator}[1]{\ensuremath{\mathbf{1}_{{#1}}}}

\DeclareMathOperator{\var}{Var}
\DeclareMathOperator{\tr}{tr}

\DeclareMathOperator{\rank}{rank}

\newcommand{\Prob}{\mathbb{P}}
\newcommand{\E}{\mathbb{E}}
\newcommand{\C}{\mathbb{C}}
\renewcommand{\P}{\mathbb{P}}
\newcommand{\N}{\mathbb{N}}
\renewcommand\Re{\operatorname{Re}}
\renewcommand\Im{\operatorname{Im}}
\newcommand{\eps}{\varepsilon}

\def\R{\mathbb{R}}

\renewcommand{\d}{\, d }

\theoremstyle{plain}
\newtheorem{theorem}{Theorem}[section]

\newtheorem{lemma}[theorem]{Lemma}
\newtheorem{corollary}[theorem]{Corollary}

\newtheorem{proposition}[theorem]{Proposition}

\theoremstyle{definition}
\newtheorem{definition}[theorem]{Definition}

\theoremstyle{remark}
\newtheorem{remark}[theorem]{Remark}

\makeatletter \let\@wraptoccontribs\wraptoccontribs \makeatother

\begin{document}
	
	\title[Extreme eigenvalues of Laplacian random matrices]{Extreme eigenvalues of Laplacian random matrices with Gaussian entries}

	\author[A. Campbell]{Andrew Campbell}
	\address{Institute of Science and Technology Austria, Am Campus 1, 3400 Klosterneuburg, Austria}
	\email{andrew.campbell@ist.ac.at}
	\thanks{A. Campbell was partially supported by the European Research Council Grant No. 101020331. }
	
	\author[K. Luh]{Kyle Luh}
	\address{Department of Mathematics\\ University of Colorado\\ Campus Box 395\\ Boulder, CO 80309-0395\\USA}
	\email{kyle.luh@colorado.edu}
	\thanks{K. Luh was supported in part by the Ralph E. Powe Junior Faculty Enhancement Award.}
	
	\author[S. O'Rourke]{Sean O'Rourke}
	\address{Department of Mathematics\\ University of Colorado\\ Campus Box 395\\ Boulder, CO 80309-0395\\USA}
	\email{sean.d.orourke@colorado.edu}
	\thanks{S. O'Rourke has been supported in part by NSF CAREER grant DMS-2143142. }
	
	\contrib[with an appendix by]{Santiago Arenas-Velilla and Victor Perez-Abreu}%{Santiago Arenas-Velilla and Victor P\'{e}rez-Abreu}
	
	%\author[S. Arenas-Valilla]{Santiago Arenas-Velilla (Appendix)}
	%\address{CIMAT, Guanajuato, Mexico}
	%\email{santiago.arenas@cimat.mx}
	
	%\author[V. P\'erez-Abreu]{Victor P\'erez-Abreu (Appendix)}
	%\address{Instituto Polit\'ecnico Nacional, Mexico}
	%\email{vpereza@ipn.mx}
	
	\begin{abstract}
		A Laplacian matrix is a real symmetric matrix whose row and column sums are zero.  
		We investigate the limiting distribution of the largest eigenvalues of a Laplacian random matrix with Gaussian entries.  Unlike many classical matrix ensembles, this random matrix model contains dependent entries.  Our main results show that the extreme eigenvalues of this model exhibit Poisson statistics.  In particular, after properly shifting and scaling, we show that the largest eigenvalue converges to the Gumbel distribution as the dimension of the matrix tends to infinity.  
		While the largest diagonal entry is also shown to have Gumbel fluctuations, there is a rather surprising difference between its deterministic centering term and the centering term required for the largest eigenvalues.  
	\end{abstract}

	\maketitle 
	
	\section{Introduction}
	
	The famous Tracy--Widom distribution appears as the limiting distribution for the largest eigenvalue of many classical random matrix ensembles \cite{MR3800840,MR4019881,MR4091114,MR2485010,MR2545551,MR2398763,MR3582818,MR2988403,MR4260217,MR2308592,MR4474568,MR1727234,MR3161313,MR1257246,MR1385083}.  However, for many other matrix ensembles---such as those containing certain structural properties or significantly less independence---other distributions can appear as the limiting law for the largest eigenvalue \cite{MR2797947,MR2797949,MR2288065}.  
	
	The present paper focuses on the limiting distributions for the largest eigenvalues of random Laplacian matrices with Gaussian entries.  
	
	\begin{definition} \label{def:Laplacian}
		For an $n \times n$ real symmetric matrix $A$, we define the \emph{Laplacian matrix} $\mathcal{L}_A$ of $A$ as 
		\[ \mathcal{L}_A := D_A - A, \]
		where $D_A = (D_{ij})$ is the diagonal matrix containing the row sums of $A$:
		\[ D_{ii} = \sum_{j=1}^n A_{ij}. \]
		We refer to $\mathcal{L}_A$ as a \emph{Laplacian matrix}.  
	\end{definition}
	
	Every real symmetric matrix that maps the all-ones vector to zero can be represented as the Laplacian matrix $\mathcal{L}_A$ of some real symmetric matrix $A$.  In the literature, Laplacian matrices are also known as Markov matrices.  
	
	When $A$ is a random real symmetric matrix, we say $\mathcal{L}_A$ is a random Laplacian matrix.  
	Random Laplacian matrices play an important role in many applications involving complex graphs \cite{MR2248695}, analysis of algorithms for the $\mathbb{Z}_2$ synchronization problem \cite{MR3777782},  community detection in the stochastic block model \cite{MR3349181}, and other optimization problems in semidefinite programming \cite{MR3777782}.  In the theoretical physics literature, random Laplacian matrices have also been used to study random impedance networks \cite{MR1747178, PhysRevE.67.047101}.  
	
	This paper focuses on the fluctuations of the largest eigenvalues of $\mathcal{L}_A$ when $A$ is drawn from the Gaussian Orthogonal Ensemble.  
	Recall that an $n \times n$ real symmetric matrix $A$ is drawn from the Gaussian Orthogonal Ensemble (GOE) if the upper-triangular entries $A_{ij}$, $1 \leq i \leq j \leq n$ are independent Guassian random variables, where $A_{ij}$ has mean zero and variance $\frac{1 + \delta_{ij}}{n}$ and $\delta_{ij}$ is the Kronecker delta.  
	
	When $A$ is drawn from the GOE, the limiting empirical spectral distribution for $\mathcal{L}_A$ is given by the free convolution of the semicircle law and the Gaussian distribution \cite{MR2759729,MR2206341,MR4440252} (see also the earlier derivations given in \cite{MR1747178, PhysRevE.67.047101}).   More generally, this is also the limiting spectral distribution of  $\mathcal{L}_A$ when $A$ is a properly normalized Wigner matrix \cite{MR2759729,MR2206341,MR4440252}; extensions of this result are also known for generalized Wigner matrices \cite{MR4440252}, inhomogeneous and dilute Erd\H{os}--R\'{e}nyi random graphs \cite{MR4193182,MR2967963}, and block Laplacian matrices \cite{MR3321497}.  Moreover, the asymptotic location of the largest eigenvalue of $\mathcal{L}_A$, for a large class of Wigner matrices $A$, has been established by Ding and Jiang \cite{MR2759729}.  When $A$ is an $n \times n$ matrix drawn from the GOE, these results show that the largest eigenvalue of $\mathcal{L}_A$ is asymptotically close to $\sqrt{2 \log n}$ as the dimension $n$ tends to infinity.  The generalized Wigner case was studied in \cite{MR4440252}, while the smallest eigenvalues were investigated in \cite{MR2948689}.  Spectral norm bounds are also known \cite{MR2206341,MR4440252}. 
	
	The goal of this paper is to study the largest eigenvalues of $\mathcal{L}_A$ when $A$ is drawn from the GOE.   This particular random matrix appears in the $\mathbb{Z}_2$ synchronization problem of recovering binary labels with Gaussian noise \cite{MR3777782}.  While we focus on the edge of the spectrum, the eigenvalues in the bulk (and corresponding eigenvectors) were studied by Huang and Landon \cite{MR4058984} for the case when $A$ is a Wigner matrix or the adjacency matrix of a sparse Erd\H{os}--R\'{e}nyi random graph.  Similar to \cite{MR4058984}, we also use resolvent techniques to prove our main results.  
	
	For any $n \times n$ real symmetric matrix $M$, we let $\lambda_n(M) \leq \cdots \leq \lambda_1(M)$ be the ordered eigenvalues of $M$.  
	Recall that the standard Gumbel distribution has cumulative distribution function 
	\begin{equation} \label{def:GumbelCDF}
		F(x) := \exp \left(-e^{-x} \right), \qquad x \in \mathbb{R}. 
	\end{equation} 
	Define
	\begin{equation} \label{def:anbn}
		a_n := \sqrt{2 \log n} \qquad \text{and} \qquad b_n := \sqrt{2 \log n} - \frac{ \log \log n + \log(4 \pi) - 2 }{2 \sqrt{2 \log n}}. 
	\end{equation}
	
	While our main results describe the joint behavior of several eigenvalues, for simplicity, we start with the largest eigenvalue of $\mathcal{L}_A$, which we show has Gumbel fluctuations.  
	
	\begin{theorem}[Largest eigenvalue] \label{thm:largest}
		Let $A$ be an $n \times n$ matrix drawn from the GOE.  Then the centered and rescaled largest eigenvalue of $\mathcal{L}_A$ 
		\[ a_n \left( \lambda_1(\mathcal{L}_A) - b_n \right) \]
		converges in distribution as $n \to \infty$ to the standard Gumbel distribution, where $a_n$ and $b_n$ are defined in \eqref{def:anbn}.  
	\end{theorem}

	\begin{remark} \label{rem:symmetry}
		By symmetry, an analogous version of Theorem \ref{thm:largest} holds for the smallest eigenvalue of $\mathcal{L}_A$ as well.  
	\end{remark}
	
	The value of the centering term $b_n$ is surprising as it does not coincide with the location of the largest diagonal entry of $\mathcal{L}_A$.  It was widely suspected that the behavior of the largest eigenvalue of $\mathcal{L}_A$ was dictated by the largest diagonal entry.\footnote{In fact, this was the result of \cite{arenas2021extremal}, which unfortunately contains an error in its proof \cite{SV}.}  The largest diagonal entry of $\mathcal{L}_A$ does indeed have a Gumbel distribution but with a different deterministic shift.  Define
	\begin{equation} \label{eq:def:bn'}
		b_n' := \sqrt{2 \log n} - \frac{\log \log n + \log(4 \pi)}{2 \sqrt{2 \log n}}.
	\end{equation}
	It is shown in Appendix \ref{sec:appendix} that, when $A$ is drawn from the GOE, the centered and rescaled largest diagonal entry of $\mathcal{L}_A$ 
	\[ a_n \left(\max_{1 \leq i \leq n} (\mathcal{L}_A)_{ii} - b_n' \right) \]
	converges in distribution as $n \to \infty$ to the standard Gumbel distribution, where $a_n$ is specified in \eqref{def:anbn}.  In other words, while the scaling factors are the same, the centering terms for the largest eigenvalue and the largest diagonal entry are different.  The terms $a_n$ and $b_n'$ are the correct scaling and centering terms, respectively, when considering the limiting fluctuations for the maximum of $n$ independent and identically distributed (iid) standard Gaussian random variables (see, for example, \cite[Theorem 1.5.3]{MR691492}), as described below.  We refer the reader to \cite{MR900810,MR691492} for more details concerning the extreme values of sequences of iid random variables. 
	
	It is useful to compare our results with those of classical extreme value theory.  To that end, for each $n \geq 1$, let $\xi_1, \ldots, \xi_n$ be a sample of $n$ iid standard Gaussian random variables, and let $\xi^{(n)}_1 > \xi^{(n)} > \cdots > \xi^{(n)}_n$ be their order statistics.  Classical results (see, for instance, \cite[Theorem 1.5.3]{MR691492}) imply that the scaled and centered largest order statistic 
	\[ a_n ( \xi^{(n)}_1 - b_n') \]
	converges in distribution to the standard Gumbel distribution as $n \to \infty$, where $a_n$ is defined in \eqref{def:anbn} and $b_n'$ is defined in \eqref{eq:def:bn'}.  
	More generally, for any fixed integer $k \geq 1$, there exists (see, for example, \cite{MR691492} Theorem 2.3.1 and the remarks thereafter) a non-trivial joint cumulative distribution $F_k: \mathbb{R}^k \to [0, 1]$ so that  
	\begin{equation} \label{eq:normal_limit}
		\lim_{n \to \infty} \Prob \left( a_n \left( \xi^{(n)}_1 - b'_n \right) \leq x_1, \ldots, a_n \left( \xi^{(n)}_k - b'_n \right) \leq x_k \right) = F_k(x_1, \ldots, x_k) 
	\end{equation} 
	for any fixed $x_1, \ldots, x_k \in \mathbb{R}$. (In particular, $F_1$ is the cumulative distribution function of the standard Gumbel distribution.)   
	By using the correct deterministic centering term, we show the same behavior for the $k$ largest eigenvalues of $\mathcal{L}_A$. 
	\begin{theorem}[$k$ largest eigenvalues] \label{thm:main}
	Let $A$ be an $n \times n$ matrix drawn from the GOE, and fix an integer $k \geq 1$.  Then, for any fixed $x_1, \ldots, x_k \in \mathbb{R}$, 
	\[ \lim_{n \to \infty} \Prob\left( a_n \left( \lambda_1(\mathcal{L}_A) - b_n \right) \leq x_1, \ldots, a_n \left( \lambda_k(\mathcal{L}_A) - b_n \right) \leq x_k \right) = F_k(x_1, \ldots, x_k), \]
	where $a_n$ and $b_n$ are defined in \eqref{def:anbn} and $F_k$ is defined in \eqref{eq:normal_limit}.  
	\end{theorem}
	
	Theorem \ref{thm:largest} follows immediately from Theorem \ref{thm:main} by taking $k = 1$.  In addition, Theorem \ref{thm:main} allows us to study the gaps between the eigenvalues.  While one can consider the joint distribution of gaps between several consecutive eigenvalues, for simplicity we only present a result for the limiting gap distribution of the largest two eigenvalues.  
	\begin{corollary}[Gap distribution] \label{cor:gapdist}
	Let $A$ be an $n \times n$ matrix drawn from the GOE.  Then, for any $x > 0$, 
	\[ \lim_{n \to \infty} \Prob \left( a_n \left( \lambda_1(\mathcal{L}_A) - \lambda_2(\mathcal{L}_A) \right) > x \right) = e^{-x}, \]
	where $a_n$ is defined in \eqref{def:anbn}. 
	\end{corollary}
	\begin{proof}
	In view of Theorem \ref{thm:main}, the function $F_2$ defined in \eqref{eq:normal_limit} describes the limiting joint distribution of $a_n \left( \lambda_1(\mathcal{L}_A) - b_n \right)$ and $a_n \left( \lambda_2(\mathcal{L}_A) - b_n \right)$.  $F_2$ also describes the limiting joint distribution of the largest order statistics $a_n \left( \xi^{(n)}_1 - b'_n \right)$ and $a_n \left( \xi^{(n)}_2 - b'_n \right)$, where the limiting gap distribution can be deduced from standard techniques in extreme value theory (see, for instance, \cite{MR1994955,MR900810}). 
	\end{proof}
	
	In addition to Corollary \ref{cor:gapdist}, Theorem \ref{thm:main} also allows us to show that the extreme eigenvalues exhibit Poisson statistics.  To this end, let $\mathcal{P}_n$ be the random point process constructed from the rescaled eigenvalues of $\mathcal{L}_A$:
	\begin{equation} \label{eq:def:Pn}
		\mathcal{P}_n = \sum_{j=1}^n \delta_{ a_n ( \lambda_j(\mathcal{L}_A) - b_n) }, 
	\end{equation}
	where $a_n$ and $b_n$ are defined in \eqref{def:anbn} and $\delta_x$ is a point mass at $x$.  
	\begin{corollary}[Poisson point process limit of extreme eigenvalues] \label{cor:poisson}
	Let $A$ be an $n \times n$ matrix drawn from the GOE.  Then the random point process $\mathcal{P}_n$, defined in \eqref{eq:def:Pn} and constructed from the eigenvalues of $\mathcal{L}_A$, converges in distribution in the vague topology as $n \to \infty$ to the Poisson point process $\mathcal{P}$ with intensity measure $\mu$ with density $d\mu = e^{-x}dx$. 
	\end{corollary}
	
	\begin{remark}
	By symmetry, analogous versions of Theorem \ref{thm:main}, Corollary \ref{cor:gapdist}, and Corollary \ref{cor:poisson} hold for the smallest eigenvalue of $\mathcal{L}_A$ as well. 
	\end{remark}

	In contrast to many classical models of random matrices with Gaussian entries, where the limiting behavior of the eigenvalues at the edge is describe by a determinantal point process \cite{MR1257246,MR1385083}, Corollary \ref{cor:poisson} shows Poisson statistics for the extreme eigenvalues.  In random matrix theory, Poisson statistics also describe the extreme eigenvalues for matrices with heavy-tailed or sparse entries (see, for example, \cite{MR2548495,MR2234922,MR2081462,MR2125574,MR4515695} and references therein).
	
	We note that the Poisson point process $\mathcal{P}$ also arises when one considers the largest order statistics from a sample $\xi_1, \ldots, \xi_n$ of $n$ iid standard normal random variables.  That is, if $\xi^{(n)}_1 > \xi^{(n)} > \cdots > \xi^{(n)}_n$ are the order statistics of $\xi_1, \ldots, \xi_n$, then it follows from standard results (see, for example, \cite{MR900810} Proposition 3.21) that the point process 
	\[ \mathcal{Q}_n = \sum_{j=1}^n \delta_{ a_n( \xi^{(n)}_j - b_n')} \]
	converges in distribution in the vague topology as $n \to \infty$ to the Poisson point process $\mathcal{P}$.  Recall that $a_n$ is defined in \eqref{def:anbn} and $b_n'$ is defined in \eqref{eq:def:bn'}.  
	
	In many matrix models with Gaussian entries (such as the GOE), there are explicit formulas for the density of the eigenvalues.  The authors are not aware of any formulas for the eigenvalues of $\mathcal{L}_A$ when $A$ is drawn from the GOE.  In particular, $\mathcal{L}_A$ is not orthogonally invariant in this case.  The proof of Theorem \ref{thm:main}---which is outlined in Section \ref{sec:overview} below---instead relies on comparing the largest eigenvalues of $\mathcal{L}_A$ to the largest eigenvalues of a matrix model with independent entries and then applying resolvent techniques to analyze this new model.  The fact that the entries are Gaussian is crucial to our method, but we conjecture that Theorem \ref{thm:main} should still hold when $A$ is a Wigner matrix with some appropriate moment assumptions on the entries.  
	
	To conclude this section, we present the proof of Corollary \ref{cor:poisson} using Theorem \ref{thm:main}.  
	
	\begin{proof}[Proof of Corollary \ref{cor:poisson}]
	We begin with some preliminaries.  
	Let $a \in \mathbb{R}$, and fix an integer $k \geq 0$.  It follows that 
	\[ \Prob( \mathcal{P}_n([a, \infty)) \geq k) = \Prob( a_n (\lambda_k(\mathcal{L}_A) - b_n) \geq a). \]
	Thus, from Theorem \ref{thm:main}, we conclude that $\mathcal{P}_n([a, \infty))$ converges in distribution to $\mathcal{P}([a, \infty))$ for any $a \in \mathbb{R}$.  In particular, this implies that $\{\mathcal{P}_n([a, \infty)) \}_{n \geq 1}$ is tight.  Hence, for any $a \in \mathbb{R}$ and any $\eps > 0$, there exists an integer $N > 0$ so that
	\begin{equation} \label{eq:Pntight}
		\Prob( \mathcal{P}_n([a, \infty)) > N) < \eps 
	\end{equation}
	for all $n$ and 
	\begin{equation} \label{eq:Pbnd}
		\Prob( \mathcal{P}([a, \infty)) > N) < \eps. 
	\end{equation}
	
	We will now use \eqref{eq:Pntight} and \eqref{eq:Pbnd} to prove Corollary \ref{cor:poisson}.  
	In view of Theorem 16.16 in \cite{MR1876169} or alternatively the results in Chapter 4.2 of \cite{MR3642325}, it suffices to show that $\mathcal{P}_n(U)$ converges in distribution to $\mathcal{P}(U)$ for any set $U \subset \mathbb{R}$ which is a union of finitely many disjoint bounded intervals.  In addition, since $\Prob( \mathcal{P}(\{c\}) > 0 ) = 0$ for all $c \in \mathbb{R}$, it suffices to assume that $U = [a_1, b_1) \cup \cdots \cup [a_l, b_l)$ for some integer $l \geq 1$ and real numbers $a_1 < b_1 < \cdots < a_l < b_l$.  Thus, by writing each interval $[a_i, b_i)$ as the difference of $[a_i, \infty)$ and $[b_i, \infty)$,  we find
	\begin{align*}
		\Prob &\left( \mathcal{P}_n(U) = k \right) = \sum_{k_1, j_1, \ldots, k_l, j_l} p^{(n)}_{k_1, j_1, \ldots, k_l, j_l}
	\end{align*}
	where the sum on the right-hand side is over all non-negative integers $k_1, j_1, \ldots, k_l, j_l$ so that $k_1 - j_1 + k_2 - j_2 + \cdots + k_l - j_l = k$ and
	\begin{align*}
		&p^{(n)}_{k_1, j_1, \ldots, k_l, j_l} \\
		&\qquad=  \Prob \left( \mathcal{P}_n([a_1, \infty)) = k_1,  \mathcal{P}_n([b_1, \infty)) = j_1, \ldots,  \mathcal{P}_n([a_l, \infty)) = k_l,  \mathcal{P}_n([b_l, \infty)) = j_l \right). 
	\end{align*}
	Since 
	\begin{align*}
		p^{(n)}_{k_1, j_1, \ldots, k_l, j_l} &= \Prob (  a_n (\lambda_{k_1+1}(\mathcal{L}_A) - b_n) < a_1 \leq a_n (\lambda_{k_1}(\mathcal{L}_A) - b_n), \ldots,  \\
		&\qquad\qquad\qquad a_n (\lambda_{j_l+1}(\mathcal{L}_A) - b_n) < b_l \leq a_n (\lambda_{j_l}(\mathcal{L}_A) - b_n) ), 
	\end{align*}
	by Theorem \ref{thm:main} each $p^{(n)}_{k_1, j_1, \ldots, k_l, j_l}$ converges to $p_{k_1, j_1, \ldots, k_l, j_l}$, where (see Lemma \ref{lemma:order and poisson equivalence}) 
	\begin{align*}
		&p_{k_1, j_1, \ldots, k_l, j_l} \\
		&\qquad=  \Prob \left( \mathcal{P}([a_1, \infty)) = k_1,  \mathcal{P}([b_1, \infty)) = j_1, \ldots,  \mathcal{P}([a_l, \infty)) = k_l,  \mathcal{P}([b_l, \infty)) = j_l \right). 
	\end{align*}	
	
	Let $\eps > 0$.  
	By \eqref{eq:Pntight} and \eqref{eq:Pbnd} (and the fact that $\{p^{(n)}_{k_1, j_1, \ldots, k_l, j_l}\}_{k_1, j_1, \ldots, k_l, j_l \geq 0}$ is a collection of probabilities of disjoint events) there exists $N > 0$ so that
	\[ \left| \sum_{k_1, j_1, \ldots, k_l, j_l}^N p^{(n)}_{k_1, j_1, \ldots, k_l, j_l} - \sum_{k_1, j_1, \ldots, k_l, j_l} p^{(n)}_{k_1, j_1, \ldots, k_l, j_l} \right| < \eps \]
	and 
	\[ \left| \sum_{k_1, j_1, \ldots, k_l, j_l}^N p_{k_1, j_1, \ldots, k_l, j_l} - \sum_{k_1, j_1, \ldots, k_l, j_l} p_{k_1, j_1, \ldots, k_l, j_l} \right| < \eps, \]
	where the notation $\sum_{k_1, j_1, \ldots, k_l, j_l}^N$ denotes the sum $\sum_{k_1, j_1, \ldots, k_l, j_l}$ with the additional restriction that none of the indices $k_1, j_1, \ldots, k_l, j_l$ exceed  $N$.  In particular, $\sum_{k_1, j_1, \ldots, k_l, j_l}^N$ denotes a finite sum, and hence 
	\[ \lim_{n \to \infty} \sum_{k_1, j_1, \ldots, k_l, j_l}^N p^{(n)}_{k_1, j_1, \ldots, k_l, j_l} = \sum_{k_1, j_1, \ldots, k_l, j_l}^N p_{k_1, j_1, \ldots, k_l, j_l}. \]
	Therefore, we conclude that
	\begin{align*}
		\sum_{k_1, j_1, \ldots, k_l, j_l} p_{k_1, j_1, \ldots, k_l, j_l} - \eps &\leq \sum_{k_1, j_1, \ldots, k_l, j_l}^N p_{k_1, j_1, \ldots, k_l, j_l}  \\
		&\leq \liminf_{n \to \infty} \sum_{k_1, j_1, \ldots, k_l, j_l}^N p^{(n)}_{k_1, j_1, \ldots, k_l, j_l} \\
		&\leq \liminf_{n \to \infty} \sum_{k_1, j_1, \ldots, k_l, j_l} p^{(n)}_{k_1, j_1, \ldots, k_l, j_l} \\
		&\leq \limsup_{n \to \infty} \sum_{k_1, j_1, \ldots, k_l, j_l} p^{(n)}_{k_1, j_1, \ldots, k_l, j_l} \\
		&\leq \limsup_{n \to \infty} \sum_{k_1, j_1, \ldots, k_l, j_l}^N p^{(n)}_{k_1, j_1, \ldots, k_l, j_l} + \eps \\
		&\leq \sum_{k_1, j_1, \ldots, k_l, j_l}^N p_{k_1, j_1, \ldots, k_l, j_l} + \eps \\
		&\leq \sum_{k_1, j_1, \ldots, k_l, j_l} p_{k_1, j_1, \ldots, k_l, j_l} + \eps.
	\end{align*}
	Since $\eps > 0$ was arbitrary, the proof is complete. 
	\end{proof}

	\section*{Acknowledgements}
	
	The authors thank Santiago Arenas-Velilla and Victor P\'erez-Abreu for comments on an earlier draft of this manuscript and for contributing Appendix \ref{sec:appendix}.  The authors also thank Yan Fyodorov for providing useful references.  
	
	\section{Overview and main reduction} \label{sec:overview}
	
	Let $A$ be an $n \times n$ matrix drawn from the GOE.  Define
	\begin{equation} \label{def:L}
		L := D - A,
	\end{equation}
	where $D$ is an $n \times n$ diagonal matrix with iid standard normal entries, independent of $A$.  We will show that the largest eigenvalues of $L$ have the same asymptotic behavior as that described in Theorem \ref{thm:main}.  
	
	\begin{theorem} \label{thm:main_L}
		Let $L$ be the $n \times n$ random matrix defined above, and fix an integer $k \geq 1$.  Then, for any fixed $x_1, \ldots, x_k \in \mathbb{R}$, 
	\[ \lim_{n \to \infty} \Prob\left( a_n \left( \lambda_1(L) - b_n \right) \leq x_1, \ldots, a_n \left( \lambda_k(L) - b_n \right) \leq x_k \right) = F_k(x_1, \ldots, x_k), \]
	where $a_n$ and $b_n$ are defined in \eqref{def:anbn} and $F_k$ is defined in \eqref{eq:normal_limit}. 
	\end{theorem}
	\begin{remark}
	Since $L$ has the same distribution as $D + A$ and $A - D$, Theorem \ref{thm:main_L} also applies to the largest eigenvalues of these matrices.  In fact, in the forthcoming proofs, it will sometimes be convenient to work with these other, equivalent expressions for $L$.  
	\end{remark}
	Theorem \ref{thm:main_L} was motivated by many results in the literature for models of deformed Wigner matrices, including the results in \cite{MR3449389,MR3502606,MR3208886,MR3500269,MR3800833,MR2288065} and references therein.  
	In this section, we prove Theorem \ref{thm:main} using Theorem \ref{thm:main_L}.  We begin by introducing the notation used here and throughout the paper.  
	
	\subsection{Notation}
	For a complex number $z$, we let $\Re(z)$ be the real part and $\Im(z)$ be the imaginary part of $z$.  We will use $i$ for both the imaginary unit and as an index in summations; the reader can tell the difference based on context.  We denote the upper-half plane as $\mathbb{C}_+ := \{z \in \mathbb{C} : \Im(z) > 0\}$.  
	
	For a matrix $A$, we let $A_{ij}$ be the $(i,j)$-entry of $A$.  The transpose of $A$ is denoted $A^\mathrm{T}$, and $A^\ast$ is the conjugate transpose of $A$.  We use $\tr A$ to denote the trace of $A$, and $\det A$ is the determinant of $A$.  If $A$ is an $n \times n$ real symmetric matrix, we let $\lambda_n(A) \leq \cdots \leq \lambda_1(A)$ be its eigenvalues.  For any matrix $M$, let $\|M\|$ be its spectral norm (also known as the operator norm). We let $I$ be the identity matrix.  If $M$ is a square matrix and $z \in \mathbb{C}$, we will sometimes write $M + z$ for the matrix $M + zI$.  
	For a vector $v$, we use $\|v\|_2$ to mean the standard Euclidean norm.

	For an event $E$, $\Prob(E)$ is its probability.
	We let $\oindicator{E}$ be the indicator function of the event $E$.  
	For a random variable $\xi$, $\E \xi$ is its expectation.  $\E_A \xi$ is the expectation of $\xi$ with respect to the GOE matrix $A$ and $\E_D \xi$ is its expectation with respect to the random diagonal matrix $D$.

	For a natural number $n$, we let $[n] = \{1, \ldots, n\}$ be the discrete interval.  For a finite set $S$, $|S|$ will denote the cardinality of $S$.  The function $\log(\cdot)$ will always denote the natural logarithm.  
	
	Asymptotic notation is used under the assumption that $n$ tends to infinity, unless otherwise noted.  We use $X = O(Y)$, $Y=\Omega(X)$, $X \ll Y$, or $Y \gg X$ to denote the estimate $|X| \leq C Y$ for some constant $C > 0$, independent of $n$, and all $n \geq C$.  If $C$ depends on other parameters, e.g. $C = C_{k_1, k_2, \ldots, k_p}$, we indicate this with subscripts, e.g. $X = O_{k_1, k_2, \ldots, k_p}(Y)$.  The notation $X = o(Y)$ denotes the estimate $|X| \leq c_n Y$ for some sequence $(c_n)$ that converges to zero as $n \to \infty$, and, following a similar convention, $X = \omega(Y)$ means $|X| \geq c_n Y$ for some sequence $(c_n)$ that converges to infinity as $n \to \infty$.  Finally, we write $X = \Theta(Y)$ if $X \ll Y \ll X$.  
	
	\subsection{Proof of Theorem \ref{thm:main}}

	Using Theorem \ref{thm:main_L}, we now complete the proof of Theorem \ref{thm:main}.

	We begin with an observation from \cite{MR4058984}.  Let $A$ be an $n \times n$ matrix drawn from the GOE and $\mathcal{L}_A$ its corresponding Laplacian matrix as defined in Definition \ref{def:Laplacian}.  Let $\tilde{R}$ be any fixed, $n \times n$ orthogonal matrix with last column
	\begin{equation} \label{eq:def:e}
	\mathbf{e} := (1/\sqrt{n}, \ldots, 1/\sqrt{n})^\mathrm{T}
	\end{equation}
	 and we use the notation $\tilde{R} = (R|\mathbf{e})$, where $R$ contains the first $n-1$ columns of $\tilde{R}$.  The eigenvalues of $\mathcal{L}_A$ coincide with those of $\tilde{R}^\mathrm{T} \mathcal{L}_A \tilde{R}$.  The set of eigenvalues of $\tilde{R}^\mathrm{T} \mathcal{L}_A \tilde{R}$ are simply the $n-1$ eigenvalues of $R^\mathrm{T}
	\mathcal{L}_A R$ along with a zero eigenvalue. In particular, with probability $1 - o(1)$, the behavior of the largest eigenvalue of $\mathcal{L}_A$ will be that of the largest eigenvalue of $R^\mathrm{T} \mathcal{L}_A R$ since these largest eigenvalues are positive with probability tending to one (see Theorem 1 in \cite{MR2759729} or, alternatively, see the proof of Proposition \ref{kprop:finalreduction} below).    
	\begin{lemma} [Proposition 2.10, \cite{MR4058984}] \label{klem:firstreduction}
		The random matrix $R^\mathrm{T} \mathcal{L}_A R$ is equal in distribution to the matrix $A' + R^\mathrm{T} \tilde DR + gI$, where $A'$ is an $(n-1) \times (n-1)$ GOE matrix, $\tilde{D}$ is an $n \times n$ diagonal matrix with iid centered Gaussian random variables with variance $n/(n-1)$ along the diagonal, $g$ is a centered Gaussian random variable with variance $1/(n-1)$, and $A'$, $\tilde D$ and $g$ are jointly independent. 
	\end{lemma}
	
	In the next lemma we make some minor reductions to remove some nuisances such as the slight dimension discrepancy and the awkward variances.  
	\begin{lemma} \label{klem:scalingandg}
		In the notation of Lemma \ref{klem:firstreduction}, if we let $W = A' + \sqrt{\frac{n-1}{n}} R^\mathrm{T} \tilde{D} R$ and $W'  = A' + R^\mathrm{T} \tilde{D}R + gI$ then for any $k \in [n]$, we have
		\[
		\P(| \lambda_k (W) - \lambda_k(W')| \geq 2 \sqrt{\log n/n}) = O(1/n) 
		\]
	\end{lemma}
	\begin{proof}
		We use a simple coupling argument by placing $W$ and $W'$ on the same probability space.  
		Since $g$ is a Gaussian with variance $1/(n-1)$, by Weyl's inequality (also known as Weyl's perturbation theorem, see Corollary III.2.6 in \cite{MR1477662}) we have that 
		\[
		|\lambda_k(W') - a_n \lambda_k(A' + R^\mathrm{T} \tilde{D}R)| < \sqrt{2 \log n/ n}. 
		\]   
		with probability at least $1 - 1/n$.  
		Similarly, with probability $1-O(1/n)$
		\[
		|\lambda_k(W) - \lambda_k(A' + R^\mathrm{T} \tilde{D}R)| \leq \frac{1}{n} \|R^\mathrm{T} \tilde{D}R\| \leq \frac{1}{n} \| \tilde D \| = 3 \sqrt{\log n}/n, 
		\]
		where $\| \tilde D \|$ was controlled using standard bounds on the maximum of iid normal random variables (see, for example, \cite[Theorem 3]{MR3389997}).  
		Therefore, the result follows from the triangle inequality and a union bound.
	\end{proof}

	In the final reduction, we compare $R^\mathrm{T} \tilde{D} R  + A'$ to a slightly augmented matrix.

	\begin{proposition} \label{kprop:finalreduction}
		We recall the notation of Lemma \ref{klem:firstreduction} and define
		\begin{equation} \label{keq:D}
		D := \sqrt{\frac{n-1}{n}} \tilde{D}
		\end{equation}
		\begin{equation} \label{keq:A}
		A := \sqrt{\frac{n-1}{n}} \left(\begin{array}{cc} A' & Y \\
			Y^\mathrm{T} & g' \end{array} \right)
		\end{equation}
		where $Y$ is a random vector in $\mathbb{R}^{n-1}$ with entries that are independent, centered Gaussians of variance $\frac{1}{n-1}$, $g'$ is a centered gaussian random variable with variance $\frac{2}{n-1}$. We consider a probability space in which $\tilde{D}$, $A'$, $Y$ and $g'$ are all defined and are jointly independent of each other.   
		Then,  for any fixed $k$ independent of $n$, 
		\[
		 \P(|\lambda_k(R^\mathrm{T} D R + A') -\lambda_k(\tilde{R}^\mathrm{T} D \tilde{R} + A) | \geq  \log \log n/\log n) = o(1). 
		\]	
	\end{proposition}
	\begin{remark}
		The scaling in \eqref{keq:D} and \eqref{keq:A} are designed so that $D$ has iid standard Gaussian random variables along its diagonal and $A$ is drawn from a GOE so it is compatible with our previous notation.  
	\end{remark}
	
		\begin{proof}[Proof of Proposition \ref{kprop:finalreduction}]
		Again, we use a coupling argument by embedding $R^\mathrm{T} D R + A'$ in a slightly larger matrix.  Note that $A$ is drawn from an $n \times n$ GOE, so we define
		\[
		Z := \tilde{R}^\mathrm{T} D \tilde{R} + \sqrt{\frac{n}{n-1}} A = \left( \begin{array}{cc} R^\mathrm{T} D R + A' & Y + R^\mathrm{T} D \mathbf{e} \\
			\mathbf{e}^\mathrm{T} DR + Y^\mathrm{T} & \mathbf{e}^\mathrm{T} D \mathbf{e} + g' \end{array} \right). 
		\]
		We first establish that the largest eigenvalue of $Z$ is sufficiently close to that of $R^\mathrm{T} D R +A'$.  
		As $R^\mathrm{T} D R +A'$ is a submatrix of $Z$, we immediately have that 
		\begin{equation} \label{keq:upperbound}
			\lambda_1(R^\mathrm{T} D R + A') \leq \lambda_1(Z).
		\end{equation}
		To find a corresponding lower bound for $\lambda_1(R^\mathrm{T} D R + A')$, we consider the eigenvalue-eigenvector equation for $Z$:
		\begin{equation}\label{keq:eigenvalue}
			Z v = \left( \begin{array}{cc} R^\mathrm{T} D R + A' & Y + R^\mathrm{T} D \mathbf{e} \\
				\mathbf{e}^\mathrm{T} DR + Y^\mathrm{T} & \mathbf{e}^\mathrm{T} D \mathbf{e} + g' \end{array} \right) \left(
			\begin{array}{c} w \\
				t \end{array} \right) = \lambda_1(Z) \left(
			\begin{array}{c} w \\
				t \end{array} \right). 
		\end{equation}
		Here, $v = \begin{pmatrix} w \\ t \end{pmatrix}$ is a unit eigenvector of $\lambda_1(Z)$, $w \in \R^{n-1}$ and $t \in \R$.  We observe that if $\|w\|_2 > 0$,
		\[
		\frac{w^\mathrm{T}}{\|w\|_2} (R^\mathrm{T} D R + A') \frac{w}{\|w\|_2} \leq \lambda_1(R^\mathrm{T} D R + A')
		\]
		by the Courant minimax principle.
		Considering the top $n-1$ coordinates of equation \eqref{keq:eigenvalue} yields
		\begin{equation} \label{keq:topn}
			(R^\mathrm{T} D R + A') w + t (Y + R^\mathrm{T} D \mathbf{e}) = \lambda_1(Z) w. 
		\end{equation}
		Multiplying this equation by $\frac{w^\mathrm{T}}{\|w\|_2^2}$ from the left and rearranging, we conclude that
		\begin{equation} \label{keq:lowerbound}
			\lambda_1(R^\mathrm{T} D R + A')  \geq \frac{w^\mathrm{T}}{\|w\|_2} (R^\mathrm{T} D R + A') \frac{w}{\|w\|_2}  = \lambda_1(Z) - \frac{t}{\|w\|_2^2} w^\mathrm{T} (Y + R^\mathrm{T} D \mathbf{e}).
		\end{equation}   
		Our next goal is to control the size of $\frac{t}{\|w\|_2^2} w^\mathrm{T} (A' + R^\mathrm{T} D \mathbf{e})$.
		We define the following events, which we later show hold with probability $1-o(1)$:
		\[
		\mathcal{E}_1 = \left\{\|A'\| \leq 10, \|\mathbf{e}^\mathrm{T} D R\|_2 \leq 10, \|Y\|_2 \leq 10, |\mathbf{e}^\mathrm{T} D \mathbf{e} + g'| \leq \frac{\log \log n}{\sqrt{n}} \right\}, 
		\]
		\[
		\mathcal{E}_2 = \left\{ \|Y^\mathrm{T} R^\mathrm{T} D\|_2 \leq \sqrt{\log \log n}, \|\mathbf{e}^\mathrm{T} D R R^\mathrm{T} D\|_2 \leq \sqrt{\log \log n} \right\}, 
		\]
		\[
		\mathcal{E}_3 = \left\{\sqrt{\log n} \leq \lambda_1(Z) \leq  2 \sqrt{\log n } \right\}. 
		\]
		We extract the final coordinate of equation \eqref{keq:eigenvalue} to obtain
		\begin{equation}\label{keq:lastcoordinate}
			(\mathbf{e}^\mathrm{T} D R + Y^\mathrm{T}) w + t(\mathbf{e}^\mathrm{T} D \mathbf{e} + g')  = \lambda_1(Z) t.
		\end{equation}
		On the event $\mathcal{E} = \mathcal{E}_1 \cap \mathcal{E}_2 \cap \mathcal{E}_3$, we have
		\begin{equation} \label{keq:tbound}
			|t| \leq 2 \frac{(\|\mathbf{e}^\mathrm{T} D R\|_2 + \|Y\|_2) \|w\|_2 }{\lambda_1(Z)} \leq \frac{40}{\sqrt{\log n}},
		\end{equation}
		which also assures us that $\|w\|_2 \geq 1/2> 0$ for all $n$ sufficiently large.  
		Left multiplying equation \eqref{keq:topn} by $Y^\mathrm{T}$ yields 
		\[
		Y^\mathrm{T} R^\mathrm{T} D R w + Y^{\mathrm{T}} A' w+  t Y^\mathrm{T} (Y + R^\mathrm{T} D \mathbf{e}) = \lambda_1(Z) Y^\mathrm{T} w.
		\]
		This implies that on the event $\mathcal{E}$, 
		\begin{align}\label{keq:firstproduct}
			|Y^\mathrm{T} w| &\leq \frac{\|Y^\mathrm{T} R^\mathrm{T} D\|_2 \|R w\|_2 + \|Y^{\mathrm{T}}\|_2 \|A'\| \|w\|_2 + t \|Y\|_2^2 + t \|Y^\mathrm{T} R^\mathrm{T} D \|_2 \|\mathbf{e}\|_2 }{\lambda_1(Z)}  \nonumber \\
			&\leq \frac{\sqrt{\log \log n} + 100 + 4000/\sqrt{\log n} + 100 \sqrt{\log \log n}/\sqrt{\log n}}{\sqrt{\log n}}  \nonumber \\
			&\ll \sqrt{\frac{\log \log n}{\log n}}.
		\end{align}
		Left multiplying equation \eqref{keq:topn} by $\mathbf{e}^\mathrm{T} D R$ gives
		\[
		\mathbf{e}^\mathrm{T} D R R^\mathrm{T} D R w + \mathbf{e}^\mathrm{T} D R A' w + t \mathbf{e}^\mathrm{T} D R Y + t \mathbf{e}^\mathrm{T} D R R^\mathrm{T} D \mathbf{e} = \lambda_1(Z) \mathbf{e}^\mathrm{T} D R w
		\]
		On the event $\mathcal{E}$, we deduce that
		\begin{align} \label{keq:secondproduct}
			|\mathbf{e}^\mathrm{T} D R w| &\leq \frac{\|\mathbf{e}^\mathrm{T} D RR^\mathrm{T} D\|_2 \| R w\|_2 + \|\mathbf{e}^\mathrm{T} D R\|_2 \|A'\| \|w\|_2  }{\lambda_1(Z)}  \nonumber \\
			&\qquad \qquad + \frac{t \|\mathbf{e}^\mathrm{T} D R\|_2 \|Y\|_2 + t \|\mathbf{e}^\mathrm{T} D RR^\mathrm{T}D\|_2 \|\mathbf{e}\|_2 }{\lambda_1(Z)} \nonumber \\
			&\leq \frac{100 \sqrt{\log \log n} + 4000/\sqrt{\log n} + 40 \sqrt{\log \log n}/\sqrt{\log n}}{\sqrt{\log n}} \nonumber \\
			&\ll \sqrt{\frac{\log \log n}{\log n}}.
		\end{align} 
		Combining, equations \eqref{keq:upperbound}, \eqref{keq:lowerbound}, \eqref{keq:tbound}, \eqref{keq:firstproduct} and \eqref{keq:secondproduct} we conclude that 
		\[
		\lambda_1(Z) - o\left(\frac{\log \log n}{\log n}\right) \leq \lambda_1(Z) - \left|\frac{t}{\|w\|_2^2} w^\mathrm{T} (Y + R^\mathrm{T} D \mathbf{e})\right|\leq \lambda_1(R^\mathrm{T} D R + A') \leq \lambda_1(Z).
		\] 
		Therefore, 
		\begin{align*}
		 \P(|\lambda_1(&R^\mathrm{T} D R + A') - \lambda_1(Z) | \geq \varepsilon) \\ 
			&\leq \P(|\lambda_1(R^\mathrm{T} D R + A') - \lambda_1(Z) | \geq \varepsilon | \mathcal{E})  + \P(\mathcal{E}^c) \\
			&= o(1).
		\end{align*}
		We can then remove the factor $\sqrt{n/n-1}$ using Weyl's inequality (as in the proof of Lemma \ref{klem:scalingandg}) to conclude that 
		\[
		\P(|\lambda_1(R^\mathrm{T} D R + A') - \lambda_1(\tilde{R}^\mathrm{T} D \tilde{R} + A) | \geq \log \log n/\log n) = o(1). 
		\]

		It remains to show that $\P(\mathcal{E}) = 1 - o(1)$.  The bounds 
		\[
		\|A\| \leq 10, \|\mathbf{e}^\mathrm{T} D R\|_2 \leq 10, \|Y\|_2 \leq 10, |\mathbf{e}^\mathrm{T} D \mathbf{e} + g'| \leq \frac{\log \log n}{\sqrt{n}}
		\]
		are straightforward (and sub-optimal) gaussian concentration or random matrix theory results, so the proofs are omitted (see \cite[Section 5.4]{MR3185193}, \cite[Chapter 5]{MR2567175}).  For $\mathcal{E}_2$, we first observe that $\|Y^\mathrm{T} R^\mathrm{T}\|_2^2 \leq  \|Y\|_2^2$. 
		By Markov's inequality, 
		\[
		\P(|\|Y^\mathrm{T} R^\mathrm{T} D\|_2^2  \geq \log \log n) \leq \frac{\E \|Y^\mathrm{T} R^\mathrm{T} D\|_2^2}{\log \log n}  \leq \frac{\E \|Y\|_2^2}{\log \log n} = \frac{\frac{n}{n-1}}{\log \log n} = o(1). 
		\]
		Thus, with probability $1-o(1)$, 
		\[
		\|Y^\mathrm{T} R^\mathrm{T} D\|_2 \leq \sqrt{\log \log n}.
		\]
		
		We now bound $\|\mathbf{e}^\mathrm{T} D R R^\mathrm{T} D\|_2$.  As $\tilde{R}$ is an orthogonal matrix,
		we have that 
		\[
		(R R^\mathrm{T})_{ii} = 1 - \frac{1}{n} \text{ and } (R R^\mathrm{T})_{ij} = -\frac{1}{n}
		\]
		for $i \neq j$.  Therefore,
		\begin{align*}
			\|\mathbf{e}^\mathrm{T} D R R^\mathrm{T} D\|_2^2 &= \sum_{i,j, k=1}^{n} \frac{1}{n} D_{ii} (R R^\mathrm{T})_{ij} D_{jj}^2 (R R^\mathrm{T})_{jk} D_{kk} \\
			&= \frac{1}{n} \sum_{i=1}^n D_{ii}^4 \left(1 - \frac{1}{n}\right)^2 - \frac{2}{n^2} \sum_{i \neq j} D_{ii}^3 D_{jj} \left(1 - \frac{1}{n}\right) \\
			&\qquad + \frac{1}{n^3} \sum_{i \neq j} D_{ii}^2 D_{jj}^2 + \frac{1}{n^3} \sum_{i \neq j, j \neq k, k \neq i} D_{ii} D_{jj}^2 D_{kk}.  
		\end{align*}
		Each sum on the right-hand side can be easily controlled via Markov's inequality.
		For the first sum, which is the dominant one, 
		\[
		\P\left(\frac{1}{n} \sum_{i=1}^{n} D_{ii}^4 \geq \frac{ \log \log n}{2}\right) \leq \frac{2 \sum_{i=1}^n \E D_{ii}^4}{n \log \log n} = \frac{6n}{n \log \log n} = o(1).
		\]
		For the second sum,
		\begin{align*}
			\P\left( \left|\frac{2}{n^2} \sum_{i \neq j} D_{ii}^3 D_{jj} \right| \geq 1\right) &\leq \frac{4 \E\left[ \left( \sum_{i \neq j} D_{ii}^3 D_{jj} \right)^2 \right]}{n^4 } \\
			&= \frac{4 \sum_{i \neq j} (\E D_{ii}^6 \E D_{jj}^2 + \E D_{ii}^4 \E D_{jj}^4)}{n^4} \\
			&= o(1), 
		\end{align*}
		where in the first equality we make use of the fact that odd moments of a centered Gaussian random variable vanish.
		Similarly, for the third sum,
		\[
		\P\left(\frac{1}{n^3} \sum_{i \neq j} D_{ii}^2 D_{jj}^2  \geq 1 \right) = o(1).
		\]
		For the final sum, we again make use of the fact that the odd moments of a centered Gaussian random variable vanish to obtain 
		\begin{align*}
			\P\left( \left|\frac{1}{n^3} \sum_{i \neq j, j \neq k, k \neq i} D_{ii} D_{jj}^2 D_{kk} \right| \geq 1 \right) &\leq \frac{\E  \left| \sum_{i \neq j, j \neq k, k \neq i} D_{ii} D_{jj}^2 D_{kk} \right|^2}{n^6} \\
			&=O \left( \frac{1}{n^2} \right). 
		\end{align*}
		Thus, asymptotically almost surely,
		\[
		\|\mathbf{e}^\mathrm{T} D R R^\mathrm{T} D\|_2 \leq \sqrt{\log \log n},
		\]
		which completes the proof for $\mathcal{E}_2$.  
		
		Now, we show that $\P(\mathcal{E}_3) = 1- o(1)$.  The eigenvalues of $\tilde{R}^\mathrm{T} D \tilde{R}$ are the diagonal entries of $D$, which are independent Gaussian random variables, so $\lambda_1(\tilde{R}^\mathrm{T} D \tilde{R})$ is the maximum of $n+1$ iid Gaussian random variables. 
		The result easily follows from standard results (see Proposition \ref{kprop:maxgauss}) and Weyl's inequality as  $\|A\| \leq 10$  with probability $1 - o(1)$. Thus, we have shown that with high probability,
		\begin{equation} \label{keq:firstspacing}
		|\lambda_1(R^\mathrm{T} D R + A') -  \lambda_1(\tilde{R}^\mathrm{T} D \tilde{R} + A)| = O(\log \log n/\log n). 
		\end{equation}

		We now can establish the result for a fixed $k$, independent of $n$.  
		By Cauchy's interlacing theorem, we have that 
		\begin{equation} \label{keq:upperbound2} 
			\lambda_{k+1}(Z) \leq \lambda_k(R^\mathrm{T} D R + A') \leq \lambda_k(Z).
		\end{equation}
		Again, we consider the eigenvalue-eigenvector equation for $Z$:
		\begin{equation}\label{keq:keigenvalue}
			Z v = \left( \begin{array}{cc} R^\mathrm{T} D R + A' & Y + R^\mathrm{T} D \mathbf{e} \\
				\mathbf{e}^\mathrm{T} DR + Y^\mathrm{T} & \mathbf{e}^\mathrm{T} D \mathbf{e} + g' \end{array} \right) \left(
			\begin{array}{c} w \\
				t \end{array} \right) = \lambda_k(Z) \left(
			\begin{array}{c} w \\
				t \end{array} \right). 
		\end{equation}
		Here, $v = \begin{pmatrix} w \\ t \end{pmatrix}$ is a unit eigenvector of $\lambda_k(Z)$, $w \in \R^{n-1}$ and $t \in \R$.  We observe that 
		\[
		(R^\mathrm{T} D R + A' - \lambda_k(Z) ) w = - t (Y + R^\mathrm{T} D \mathbf{e}).
		\]
		This implies that the least singular value of $R^\mathrm{T} D R + A' - \lambda_k(Z)$
		is upper bounded by $\varepsilon:= |t| \frac{\|w^{\mathrm{T}} Y\|_2 + \|w^{\mathrm{T}} R^\mathrm{T} D \mathbf{e}\|_2}{\|w\|_2}$.  We let $\mathcal{E}'_k$ be the event that $\sqrt{\log n} \leq \lambda_k(Z) \leq 2 \sqrt{\log n}$. Note that $\P(\mathcal{E}'_k) = 1 -o(1)$ follows from Proposition \ref{prop:count}. As before, on the event $\mathcal{E}\cap \mathcal{E}_k$, $\varepsilon = O(\log \log n/\log n)$.   As our matrices are symmetric, this implies that $R^TDR + A' - \lambda_k(Z)$ has an eigenvalue with absolute value at most $\varepsilon$.   In other words, there exists a $k'$ such that
		\[
		|\lambda_{k'}(R^TDR + A') - \lambda_{k}(Z)| \leq \varepsilon.
		\]
		By Cauchy interlacing, $k'$ can be chosen to be $k-1$ or $k$.  We wish to exclude the possibility that $k' = k-1$.  Here, we use the spacing of the eigenvalues of $Z$.  By  
		Proposition \ref{aprop:Spacing for L+D matrix} and a simple union bound,
		\[
		|\lambda_{\ell}(Z) - \lambda_{\ell+1}(Z) | = \Omega(\log^{-3/4} n)
		\]
		with high probability for all $\ell \leq k$.  Proposition \ref{aprop:Spacing for L+D matrix} applies to $D+A$ which by the rotational invariance of the GOE also applies to $Z$ (followed by Weyl's inequality due to the slight perturbation from the factor $\sqrt{\frac{n}{n-1}}$). Therefore, by \eqref{keq:firstspacing}, on the event $\mathcal{E} \cap \mathcal{E}'_2$,
		\[
		|\lambda_2(R^\mathrm{T} D R + A') -  \lambda_2(\tilde{R}^\mathrm{T} D \tilde{R} + A)| = O(\log \log n/\log n). 
		\]
		Iterating this argument a constant number of times shows that 
		\[
		|\lambda_\ell(R^\mathrm{T} D R + A') -  \lambda_\ell(\tilde{R}^\mathrm{T} D \tilde{R} + A)| = O(\log \log n/\log n) 
		\]
		for all $\ell \leq k$, which completes the proof.

	\end{proof}
	
	We now summarize the reductions that culminate in the statement of Theorem \ref{thm:main}.
	
	\begin{proof}[Proof of Theorem \ref{thm:main}]
		We again use the notation introduced in Lemma \ref{klem:firstreduction} and Proposition \ref{kprop:finalreduction}, where all the random elements are placed on the same probability space.  
		We have the decomposition 
		\begin{align*}
			 a_{n} (\lambda_1(A' + R^\mathrm{T} \tilde{D}R + gI) -b_{n}) &= a_{n} (\lambda_1(\tilde{R}^\mathrm{T} D \tilde{R} + A) - b_{n} ) \\
			 &\qquad  + a_n (\lambda_1(R^\mathrm{T} D R + A') - \lambda_1(\tilde{R}^\mathrm{T} D \tilde{R} + A)  ) \\
			 &\qquad  + a_n ( \lambda_1(A' + R^\mathrm{T} \tilde{D}R + gI) - \lambda_1(R^\mathrm{T} D R + A')). 
		\end{align*}
	By Theorem \ref{thm:main_L} and the rotational invariance of $A$, $a_n (\lambda_1(\tilde{R}^\mathrm{T} D\tilde{R} + A) - b_{n} )$ converges in distribution as $n \rightarrow \infty$ to the standard Gumbel distribution.  By Lemma \ref{klem:scalingandg} and Proposition \ref{kprop:finalreduction}, both 
		\[
		a_n (\lambda_1(R^\mathrm{T} D R + A') - \lambda_1(D + A)  ) 
		\]
		and 
		\[
		a_n ( \lambda_1(A' + R^\mathrm{T} \tilde{D}R + gI) - \lambda_1(R^\mathrm{T} D R + A'))
		\]
		converge to zero in probability.  Therefore, by Slutsky's theorem, 
		\[
		a_{n} (\lambda_1(A' + R^\mathrm{T} \tilde{D}R + gI) -b_{n})
		\]
		converges in distribution to the standard Gumbel distribution.  
		
		We let $\mathcal{E}$ be the event that $\lambda_1(\mathcal{L}_A) > 0$ and $\mathcal{E'}$ the event that $\lambda_1(A' + R^\mathrm{T} \tilde{D}R + gI) > 0$.
		Now, by Lemma \ref{klem:firstreduction},
		\[
		\lambda_1(\mathcal{L}_A) \oindicator{\mathcal{E}}
		\]
		and
		\[
		\lambda_1(A' + R^\mathrm{T} \tilde{D}R + gI) \oindicator{\mathcal{E}'}
		\]
		are equal in distribution.  In addition, both 
		\[
		\lambda_1(\mathcal{L}_A) \oindicator{\mathcal{E}^c}
		\]
		and
		\[
		\lambda_1(A' + R^\mathrm{T} \tilde{D}R + gI) \oindicator{\mathcal{E}'^c}
		\]
		converge to zero in probability 
	since $\P(\mathcal{E}^c) + \P(\mathcal{E}'^c)  = o(1)$ by the remarks preceding Lemma \ref{klem:firstreduction}. 
		Therefore, we conclude that $a_n(\lambda_1(\mathcal{L}_A) - b_n)$ converges in distribution to the standard Gumbel distribution as well.
	\end{proof}

	\subsection{Overview}
	The rest of the paper is devoted to the proof of Theorem \ref{thm:main_L}.  Our proof is based on the resolvent approach, which compares the eigenvalues of $L$ with the eigenvalues of $D$.  To this end, we define the resolvent matrices
	\[ G(z) := (L - z)^{-1} \qquad \text{and} \qquad Q(z) := (D - z)^{-1} \]
	for $z \in \mathbb{C}_+ := \{ z \in \mathbb{C} : \Im(z) > 0 \}$.  Here, we use the convention that $(L-z)^{-1}$ (alternatively, $(D-z)^{-1}$) denotes the matrix $(L-zI)^{-1}$ (alternatively, $(D-zI)^{-1}$), where $I$ is the identity matrix.  Often, we will simply write $G$ and $Q$ for the matrices $G(z)$ and $Q(z)$, respectively.  We will define the Stieltjes transforms 
	\begin{equation} \label{def:mnsn}
		m_n(z) := \frac{1}{n} \tr G(z) \qquad \text{and} \qquad s_n(z) := \frac{1}{n} \tr Q(z). 
	\end{equation}
	The limiting Stieltjes transform of $s_n$ is given by 
	\begin{equation} \label{def:s}
		s(z) := \int_{-\infty}^{\infty} \frac{\phi(x)}{x - z} \d x 
	\end{equation} 
	for $z \in \mathbb{C}_+$, where 
	\begin{equation} \label{def:phi}
		\phi(x) := \frac{1}{\sqrt{2 \pi}} e^{-x^2/2} 
	\end{equation} 
	is the density of the standard normal distribution.  The limiting Stieltjes transform $m$ (which is the free convolution of the semicircle law with the standard normal distribution) of $m_n$ is uniquely defined as the solution of 
	\begin{equation} \label{def:fc}
		m(z) = \int_{-\infty}^\infty \frac{\phi(x)}{x - z - m(z)} \d x, \qquad z \in \mathbb{C}_+, 
	\end{equation} 
	where $z m(z) \to -1$ as $|z| \to \infty$ in the upper-half plane.  
	
	For fixed (small) $\delta > 0$, we define the spectral domains 
	\[ S_{\delta} := \{z \in \mathbb{C}_+ : \sqrt{ (2 - \delta) \log n} \leq \Re(z) \leq \sqrt{3 \log n}, n^{-1/4} \leq \Im(z) \leq 1\}, \]
	\[ \tilde{S}_{\delta} := \{ z \in \mathbb{C}_+ : \sqrt{ (2 - \delta) \log n} \leq \Re(z) \leq \sqrt{3 \log n},  \Im(z) = n^{-1/4} \} \]
	and
	\[ \hat{S}_{\delta} := \{ z \in \mathbb{C}_+ : \sqrt{ (2 - \delta) \log n} \leq \Re(z) \leq \sqrt{3 \log n},  \Im(z) = \sqrt{2} n^{-1/4} \}. \]
	Our method requires us to mostly work on $\tilde{S}_{\delta}$ and $\hat{S}_{\delta}$ for some fixed $\delta > 0$, but it will sometimes be more convenient to state results for the larger domain $S_{\delta}$.  
	As is common in the literature, we will often take $E := \Re(z)$ and $\eta := \Im(z)$.  
	
	For any fixed $z \in \mathbb{C}_+$, $m_n(z)$ is random and $\E_A m_n(z)$ will denote its expectation with respect to the GOE random matrix $A$.

	Our key technical result is the following.
	
	\begin{theorem} \label{thm:stieltjes_transforms}
		There exists $\delta > 0$ so that 
		\[ \sup_{z = E + i \eta \in \tilde S_{\delta} \cup \hat{S}_{\delta}} n \eta \left| m_n(z) - s_n(z + \E_A m_n(z)) \right| = o(1) \]
		with overwhelming probability\footnote{An event $E$ holds with overwhelming probability if, for every $p > 0$,  $\Prob(E) \geq 1 - O_p(n^{-p})$; see Definition \ref{def:events} for details.}. 
	\end{theorem}
	
	Theorem \ref{thm:stieltjes_transforms} will allow us to compare the largest eigenvalue of $L$ to the largest eigenvalue of $D$, up to a small shift, which we will need to track carefully.   Theorem \ref{thm:stieltjes_transforms} should be compared to other local laws in the random matrix theory literature such as \cite{MR2784665,MR2669449,MR4134946,MR3183577,MR3770875,MR3622895,MR3602820,MR4078529,MR2481753,MR2871147,MR3098073,MR2964770,MR3068390,MR3699468,MR2567175,MR3962004,MR3800840,MR4019881,MR4091114,MR3208886} and references therein; however, the reader should be aware that this list is very incomplete and represents only a small fraction of the known local law results.

	\subsection{Outline of the remainder of the article}
	In the next section, we gather some technical tools and their proofs that will be of use in the rest of the argument.  In Section \ref{sec:stability}, we prove a quantitative stability theorem for approximate solutions of \eqref{def:fc}.  Section \ref{sec:concentrationofstieltjes} is devoted to the concentration of $m_n(z)$ near its expectation $\E_A m_n(z)$.  Section \ref{sec:proofofmaintechnical} contains the proof of our main technical result, Theorem \ref{thm:stieltjes_transforms}.  Finally, we combine all the results in Section \ref{sec:proofofmain} to prove Theorem \ref{thm:main_L}.
	
	\section{Tools}
	
	This section introduces the tools we will need in the proof of Theorem \ref{thm:main_L}.  We begin with a definition describing high probability events.  
	
	\begin{definition}[High probability events] \label{def:events}
		Let $E$ be an event that depends on $n$.
		\begin{itemize}
			\item $E$ holds \emph{asymptotically almost surely} if $\Prob(E) = 1 - o(1)$.
			\item $E$ holds \emph{with high probability} if $\Prob(E) = 1 - O(n^{-c})$ for some constant $c > 0$.
			\item $E$ holds \emph{with overwhelming probability} if, for every $p > 0$,  $\Prob(E) \geq 1 - O_p(n^{-p})$.  
		\end{itemize}
	\end{definition}
	
	For $z = E + i \eta \in \mathbb{C}_+$, the \emph{Ward identity} states that
	\begin{equation} \label{eq:ward}
		\sum_{j = 1}^n \left| G_{ij}(z) \right|^2 = \frac{1}{ \eta} \Im G_{ii}(z). 
	\end{equation} 
	
	If $A$ and $B$ are invertible matrices, the \emph{resolvent identity} states that
	\begin{equation} \label{eq:resolvent}
		A^{-1} - B^{-1} = A^{-1} (B - A) B^{-1} = B^{-1} (B - A) A^{-1}. 
	\end{equation} 
	
	If $\xi$ is a Gaussian random variable with mean zero and variance $\sigma^2$ and $f: \mathbb{R} \to \mathbb{C}$ is continuously differentiable, the \emph{Gaussian integration by parts formula} states that
	\begin{equation} \label{eq:ibp}
		\E[ \xi f(\xi)] = \sigma^2 \E[ f'(\xi) ], 
	\end{equation}
	provided the expectations are finite.  
	
	Our next result bounds (via Chernoff's inequality) the number of large entries in the diagonal matrix $D$.  
	
	\begin{proposition} \label{prop:count}
		Let $D$ be the $n \times n$ diagonal matrix whose entries are iid standard normal random variables.  Then, for any $\eps \in (0, 1/2)$, there exists $\delta > 0$ so that
		\[ \left| \{ 1 \leq i \leq n : D_{ii} \geq \sqrt{(2- \delta) \log n} \} \right| = O_{\eps}(n^{\eps}) \]
		with overwhelming probability. Similarly, for any $\varepsilon' \in (0,1/2)$, there exists a $\delta'>0$ so that
		\[
		\left| \{ 1 \leq i \leq n : D_{ii} \geq \sqrt{(2- \delta') \log n} \} \right| = \Omega_{\eps'}(n^{\eps'})
		\] 
		with overwhelming probability.
	\end{proposition}

	\begin{proof}
		Fix $\eps \in (0,1/2)$, and let $2\eps<\delta<2$. Let $X_i=\indicator{D_{ii}\geq \sqrt{(2- \delta) \log n}}$ and $S_n=\sum_{i=1}^n X_i$. Chernoff's inequality (see Theorem 2.1.3 in \cite{MR2906465})  gives that for any $\lambda>0$ \begin{equation}\label{eq:Chernoff's for Sn}
			\P\left(|S_n-\E S_n|\geq \lambda\sqrt{\var (S_n)} \right)\leq C\max\left\{\exp(-c\lambda^2),\exp(-c\lambda\sqrt{\var(S_n)}) \right\}
		\end{equation} for absolute constants $C,c>0$, where $\var (S_n)$ is the variance of $S_n$. From the standard bounds \begin{equation*}
		\left(\frac{1}{x}-\frac{1}{x^3} \right)\frac{e^{-x^2/2}}{\sqrt{2\pi}}\leq \P\left(D_{11}\geq x\right)\leq\frac{e^{-x^2/2}}{x\sqrt{2\pi}},\qquad x>0
	\end{equation*} on standard Gaussian random variables it is straightforward to show that \begin{equation*}
			\frac{n^{\delta/2}}{\sqrt{\log n}} \ll \E S_n\leq n^{\delta/2},
		\end{equation*} and \begin{equation}\label{eq:A:Variance of Sn lower bounds}
			\frac{n^{\delta/2}}{2\sqrt{2 \pi (2-\delta)\log n}}\left(1-\frac{1}{(2-\delta)\log n} \right)\leq \var(S_n)\leq n^{\delta/2}.
		\end{equation} Letting $\lambda=\sqrt{\var(S_n)}$ in \eqref{eq:Chernoff's for Sn} gives that \begin{equation*}
			\P\left(S_n> 2n^{\delta/2} \right)\leq C\exp(-c\var(S_n)),
		\end{equation*} and lower bounding $\var(S_n)$ by \eqref{eq:A:Variance of Sn lower bounds} completes the proof of the first statement.
	
		For the last statement, letting $\lambda=\sqrt{\var(S_n)}/\log n$ in \eqref{eq:Chernoff's for Sn} gives that 
		\[
		\P\left( S_n < \E S_n - \var(S_n)/\log^2 n \right) \leq \P\left( S_n < \frac{n^{\delta/2}} {2 \log n} \right) \leq C\exp(-c\var(S_n)/\log n).
		\]
	   Again, lower bounding $\var(S_n)$ by \eqref{eq:A:Variance of Sn lower bounds} completes the proof.
		
	\end{proof}

	We will need the following general concentration result for the Stieltjes transform of random symmetric matrices with independent entries.  
	\begin{proposition}[Naive concentration of the Stieltjes transform] \label{prop:concentration}
		Let $W$ be an $n \times n$ real symmetric random matrix whose entries on and above the diagonal $W_{ij}$, $1 \leq i \leq j \leq n$ are independent random variables.  Then
		\[ \Prob \left( \left| \frac{1}{n} \tr (W - z)^{-1} - \E \frac{1}{n} \tr (W - z)^{-1} \right| \geq \frac{t}{\eta \sqrt{n}} \right) \leq C e^{-ct^2} \]
		for any $t \geq 0$ and any $z = E + i \eta \in \mathbb{C}_+$, where $C, c > 0$ are absolute constants. 
	\end{proposition}
	\begin{proof}
		Fix $z=E+i\eta \in \C_+$ and let $M$ be an $n\times n$ real symmetric matrix. Let $M'$ be an $n\times n$ real symmetric matrix equal to $M$, up to possibly a single row and corresponding column being different. Define $R(z)=(M - z)^{-1}$ and $ R'(z)=(M' - z)^{-1}$. It follows from the resolvent identity \eqref{eq:resolvent} that\begin{equation}
			\rank\left( R(z)- R'(z) \right) \leq 2.
		\end{equation} It then follows that \begin{align*}
			\left|\frac{1}{n} \tr (M - z)^{-1}-\frac{1}{n} \tr (M' - z)^{-1} \right|&=\left|\frac{1}{n} \tr \left((M - z)^{-1}- (M' - z)^{-1}\right) \right| \\
			&\leq \frac{\rank\left(R(z)- R'(z) \right) }{n}\|R(z)- R'(z)\|\\
			&\leq \frac{4}{n\eta}.
		\end{align*} We can then conclude from McDiarmid's inequality (see \cite{MR1678578}) that\begin{equation}
			\Prob \left( \left| \frac{1}{n} \tr (W - z)^{-1} - \E \frac{1}{n} \tr (W - z)^{-1} \right| \geq \frac{t}{\eta \sqrt{n}} \right) \leq C e^{-ct^2},
		\end{equation} for any $t \geq 0$, where $C, c > 0$ are absolute constants. 
	\end{proof}

	\subsection{Basic concentration and linear algebra identities}
	We record several well-known concentration inequalities and algebraic identities that will be of use.  The first proposition is a strong concentration result for the maximum of a sequence of iid Gaussian random variables.  
	\begin{proposition}[Theorem 3 from \cite{MR3389997}] \label{kprop:maxgauss}
		Let $X_1, \dots, X_n$ be iid standard Gaussian random variables.  Define $M_n = \max_{1 \leq i \leq n} X_i$.  Then, for any $t \geq 0$, 
		\[
		\P(|M_n - b_n'|) > t) \leq C \exp(-c t \sqrt{\log n}).
		\] 
		where $b_n'$ is defined in \eqref{eq:def:bn'} and $C, c > 0$ are absolute constants.
	\end{proposition}
	The next lemma is a convenient moment bound for a martingale difference sequence. 
	\begin{lemma} [Lemma 2.12 from \cite{MR2567175}] \label{klem:burkholder2}
		Let $\{X_k\}$ be a complex martingale difference sequence and $\mathcal{F}_k = \sigma(X_1, \dots, X_k)$ be the $\sigma$-algebra generated by $X_1, \dots, X_k$.  Then, for any $p \geq 2$, 
		\[
		\E  \left|\sum_{k=1}^n  X_k \right|^p \leq K_p \left(\E\left( \sum_{k=1}^n \E_{k-1} |X_k|^2 \right)^{p/2} + \E \sum_{k=1}^n |X_k|^p \right). 
		\]
		where $K_{p}$ is a constant that only depends on $p$ and $\E_{k-1}[\cdot] := \E[\cdot | \mathcal{F}_{k-1}]$.  
	\end{lemma}
	The next concentration lemma is helpful in controlling the deviation of a quadratic form from its expectation. 
	\begin{lemma} [Equation (3) from \cite{MR4164840}] \label{klem:quadraticform}
		Let $X$ be an $n$-vector containing iid standard Gaussian random variables, $A$ a deterministic $n \times n$ matrix and $\ell \geq 1$ an integer.  Then
		\[
		\E[X^* A X - \tr A|^{2 \ell} \leq K_{\ell} (\tr A A^* )^\ell
		\]
		where $K_{\ell}$ is a constant that only depends on $\ell$.  
	\end{lemma}
	Finally, we will require the following algebraic identity in Section \ref{sec:concentrationofstieltjes}.
	\begin{lemma} [Theorem A.5 from \cite{MR2567175}] \label{klem:tracedifference}
		Let $A$ be an $n \times n$ symmetric matrix and $A_k$ be the $k$-th major submatrix of size $(n-1) \times (n-1)$.  If $A$ and $A_k$ are both invertible, then
		\[
		\tr( A^{-1}) - \tr(A_k^{-1}) = \frac{1+ \alpha_k^* A_k^{-2} \alpha_k}{A_{kk} - \alpha_k^* A_k^{-1} \alpha_k}
		\] 
		where $\alpha_k$ is obtained from the $k$-th column of $A$ by deleting the $k$-th entry.  
	\end{lemma}
	
		\subsection{Basic order statistics estimates}
	Let $X_1, \dots, X_n$ be iid standard Gaussian random variables.  If we order these random variables in order of magnitude, we denote this by
	\[
	X_{(n)} \leq \dots \leq X_{(2)} \leq X_{(1)}.
	\]
	We define
	\[
	\phi(x) := \frac{1}{\sqrt{2 \pi}} e^{-x^2/2},
	\]
	and
	\[
	\Phi(x) = \int_{-\infty}^x \phi(t) \, dt. 
	\]
	We recall the well-known tail bounds
	\begin{equation} \label{keq:gaussiantailbounds}
		\frac{e^{-x^2/2}}{\sqrt{2 \pi}} \left(\frac{1}{x} - \frac{1}{x^3} \right) \leq \int_x^\infty \frac{1}{\sqrt{2 \pi}} e^{-t^2/2} \, dt \leq \frac{e^{-x^2/2}}{x \sqrt{2 \pi}}
	\end{equation}
	which hold for $x > 0$.  
	Our first simple lemma captures the joint distribution between two order statistics.  
	\begin{lemma} [Chapter 2.2 of \cite{MR1994955}] \label{klem:jointdistr}
		For $1 \leq r < s \leq n$.  
		Let $f_{(r)(s)}$ be the joint density of $X_{(r)}$ and $X_{(s)}$. We have that for $x < y$,
		\[
		f_{(r)(s)}(x, y)= \frac{n! \Phi^{n-s}(x) \phi(x) [\Phi(y) - \Phi(x)]^{s-r-1} \phi(y) [1-\Phi(y)]^{r-1}}{r! (s-r-1)! (n-s)!} .
		\]
	\end{lemma}
	
	We use the previous lemma to obtain some rough bounds on the gaps between the extreme order statistics.  
	%A more precise estimate must exist in the literature, but we were unable to locate one so we prove the following which suffices for our purposes.  
	
	\begin{proposition} \label{kprop:spacings}
		Consider a fixed $k \in \mathbb{N}$.  Then for $\delta >0$,
		\[
		\P(X_{(k)} - X_{(k+1)} < \log^{-1/2-\delta} n ) = O_{k, \delta}(\log^{-\delta} n). 
		\]
	\end{proposition}
	\begin{proof}
		Applying Lemma \ref{klem:jointdistr}, we can find the joint distribution of $X_{(k)}$ and $X_{(k+1)}$.  For $x < y$,
		\[
		f_{(k)(k+1)}(x, y) = \frac{n! \Phi^{n-k-1}(x) \phi(x) \phi(y) [1-\Phi(y)]^{k-1}}{k!  (n-k-1)!} .
		\]
		We let $t = \log^{-1/2-\delta} n$. Therefore, we have that
		
		\begin{align*}
			\P(X_{(k)} &- X_{(k+1)} < t) = \int_{-\infty}^{\infty} \int_{-\infty}^{\infty} \oindicator{y - x < t} f_{(k)(k+1)}(x, y) \oindicator{y > x} \,dx\, dy \\
			&=  \frac{n!}{k!  (n-k-1)!}\int_{-\infty}^{\infty} \int_{-\infty}^{\infty} \oindicator{y - x < t}  \Phi^{n-k-1}(x) \phi(x) \phi(y) [1-\Phi(y)]^{k-1} \oindicator{y > x} \,dx\, dy \\
			&= \frac{n!}{k!  (n-k-1)!}\int_{-\infty}^{\infty} \Phi^{n-k-1}(x) \phi(x) \int_{x}^{x+t}    \phi(y) [1-\Phi(y)]^{k-1}  \,dy \, dx \\ 
			&= \frac{n!}{k!  (n-k-1)!}\int_{-\infty}^{\infty} \Phi^{n-k-1}(x) \phi(x) \left[\frac{(1 - \Phi(x))^k}{k} - \frac{(1 - \Phi(x+t))^k}{k} \right] \, dx \\ 
			&=  \frac{n!}{k!  (n-k-1)!} \int_{-\infty}^{\sqrt{\log n}} \Phi^{n-k-1}(x) \phi(x) \left[\frac{(1 - \Phi(x))^k}{k} - \frac{(1 - \Phi(x+t))^k}{k} \right] \, dx \\
			&\qquad + \frac{n!}{k!  (n-k-1)!} \int_{\sqrt{\log n}}^{\infty} \Phi^{n-k-1}(x) \phi(x) \left[\frac{(1 - \Phi(x))^k}{k} - \frac{(1 - \Phi(x+t))^k}{k} \right] \, dx \\ 
			&= I_1 + I_2. 
		\end{align*}
		We control $I_1$ first: 
		\begin{align*}
			I_1 &\leq n^{k+1} \int_{-\infty}^{\sqrt{\log n}}\Phi^{n-k-1}(x) \phi(x) \, dx \\
			&\leq n^{k+1}\int_{0}^{1-n^{-2/3}} u^{n-k-1}  \, du \\
			&\leq \exp(-n^{-2/3} (n-k-1) + k \log n) \\
			&\leq \exp(-n^{1/3}/2) .
		\end{align*}
		Now, we consider $I_2$.  Note that in this integral, $x \geq \sqrt{\log n}$ so we make use of \eqref{keq:gaussiantailbounds} frequently.  Indeed, we obtain 
		\begin{align*}
			I_2 &= \frac{n!}{k!  (n-k-1)!} \int_{\sqrt{\log n}}^{\infty} \Phi^{n-k-1}(x) \phi(x) \left[\frac{(1 - \Phi(x))^k}{k} - \frac{(1 - \Phi(x+t))^k}{k} \right] \, dx \\  
			&\leq \frac{n!}{k!  (n-k-1)!} \int_{\sqrt{\log n}}^{\infty} \Phi^{n-k-1}(x) \phi(x) \left[\frac{\phi(x)^k}{kx^k} - \frac{\phi(x+t)^k (1- (x+t)^{-2})^k}{k (x+t)^k} \right] \, dx \\
			&\leq \frac{n!}{k!  (n-k-1)! k} \int_{\sqrt{\log n}}^{\infty} \Phi^{n-k-1}(x) \phi(x) \left(\frac{\phi(x)}{ x} \right)^k\left[1 - \frac{ e^{-kxt} e^{-kt^2/2} (1- (x+t)^{-2})^k}{ (1+t/x)^k} \right] \, dx \\  
			&\leq \frac{n!}{k!  (n-k-1)! k} \int_{\sqrt{\log n}}^{\infty} \Phi^{n-k-1}(x) \phi(x) \frac{(1-\Phi(x))^k}{(1 - 1/x^2)^k} \left[1 - \frac{ e^{-kxt} e^{-kt^2/2} (1- (x+t)^{-2})^k}{ (1+t/x)^k} \right] \, dx \\   
			&\leq  \left[1 - e^{-k \sqrt{\log n} t/2} \right] \frac{n!}{k!  (n-k-1)! k} \int_{\sqrt{\log n}}^{\infty} \Phi^{n-k-1}(x) \phi(x) (1-\Phi(x))^k  \, dx \\  
			&\leq    \left[1 - e^{-k \sqrt{\log n} t/2} \right] \frac{n!}{k!  (n-k-1)! k} \int_{0}^{1} u^{n-k-1}  (1-u)^k  \, du \\
			&\leq 1 - e^{-k \sqrt{\log n} t/2} \\
			&= O(\sqrt{\log n} t), 
		\end{align*}	
		as desired. 
	\end{proof} 
	
	The following proposition was used in the proof of Proposition \ref{kprop:finalreduction}.  
	
	\begin{proposition}\label{aprop:Spacing for L+D matrix}
		Let $L$ be the matrix defined in \eqref{def:L}. Consider a fixed $k \in \mathbb{N}$.  Then for $\delta >0$,
		\[
		\P(\lambda_{k}(L) - \lambda_{k+1}(L) < \log^{-1/2-\delta} n ) = o_\delta(1). 
		\]
	\end{proposition}
	The proof of Proposition \ref{aprop:Spacing for L+D matrix} is a step in the proof of Theorem \ref{thm:main_L} given in Section \ref{sec:proofofmain}.  We emphasis that the only place where Proposition \ref{aprop:Spacing for L+D matrix} is used is in the proof of Proposition \ref{kprop:finalreduction}.  
	
	We were not able to find the following lemma in the literature, however it follows implicitly from well known results on point processes and order statistics. 
	
	\begin{lemma}\label{lemma:order and poisson equivalence}
		Let $X_1, X_2, \dots$ be a sequence of iid real random variables, $a_n>0$ and $b_n$ be sequences, and $G$ a differentiable cumulative distribution function of some continuous random variable. Let $X_1^{(n)}\geq X_2^{(n)}\geq\cdots\geq X_n^{(n)}$ be the order statistics of $X_1,\dots,X_n$. The following are equivalent: \begin{enumerate}
			\item\label{order conv} $a_n(X_1^{(n)}-b_n)$ converges in distribution as $n \to \infty$ to some non-degenerate limit with cumulative distribution function $G$.
			
			\item\label{joint order} There exists functions $\{G_k:\R^k\rightarrow\R\}_{k=1}^\infty$ such that for any $k\in\N$ and any $x_1\geq\cdots\geq x_k\in\R$ \begin{equation*}
				\lim\limits_{n\rightarrow\infty}\P\left(a_n(X_1^{(n)} - b_n) \leq x_1,\dots, a_n (X_k^{(n)} - b_n) \leq x_k \right)=G_k(G(x_1),\dots,G(x_k)).
			\end{equation*}
			
			\item\label{point proc conv} The point process $\sum_{i=1}^{n}\delta_{a_n(X_i-b_n)}$ converges in distribution as $n \to \infty$ to a Poisson point process with intensity measure $\mu$ with density $f(x)=\frac{G'(x)}{G(x)}$ where $G(x)>0$ and $\mu([x,\infty))=\infty$ for any $x\in\R$ such that $G(x)=0$.
		\end{enumerate}
	\end{lemma} 
	\begin{proof}  See \cite{MR691492} Theorems 2.3.1 and 2.3.2 for the equivalence of (\ref{joint order}) and (\ref{order conv}).
		
		Assume (\ref{order conv}), and fix $x\in\R$ such that $G(x)>0$. It is straightforward to check that \begin{equation*}
			n\log\left(1-\P(a_n(X_1-b_n)\geq x) \right)=\log G(x)+o(1),
		\end{equation*} and thus \begin{equation*}
			\lim\limits_{n\rightarrow\infty}n\P(a_n(X_1-b_n)\geq x)=-\log G(x).
		\end{equation*} It then follows from \cite{MR900810} Proposition 3.21 that $\sum_{i=1}^{n}\delta_{a_n(X_i-b_n)}$ converges to a Poisson point process with an intensity measure $\mu$ such that $\mu([x,\infty))=-\log G(x)$. 
		
		Assume \eqref{point proc conv}, and fix $x\in\R$. By \cite{MR900810} Proposition 3.21, \begin{equation*}
			\lim\limits_{n\rightarrow\infty}n\P(a_n(X_1-b_n)\geq x)=-\log G(x).
		\end{equation*} One can then check that \begin{equation*}
			\lim\limits_{n\rightarrow\infty}\P(a_n(X_1^{(n)}-b_n)\leq x)=G(x).
		\end{equation*} This completes the proof.
	\end{proof}

	\subsection{Basic Stieltjes transform and free probability estimates}
	%%%%%%%%%%%%%%%%%%%% Results on the Stieltjes Transform of a Gaussian%%%%%%%%%%%%%%%%%%%%%%%%%%%%%%%%%%%%%%%%%%%%%%%	
	This section isolates some useful estimates on $s(z)$ and $m(z)$, defined by \eqref{def:s} and \eqref{def:fc} respectively. Several of the proofs use a contour integral arguement, which is given in detail in Lemma \ref{lemma:Free convolution is sub-Gaussian} below. We begin by stating a result of Biane \cite{MR1488333} on free convolutions of semicircle distributions and an arbitrary distribution. For convenience, we specialize this result to the Gaussian distribution.
	\begin{lemma}[Corollaries 3 and 4 from \cite{MR1488333}]\label{lemma:A:Biane result}
		Define the function $v:\R\rightarrow[0,\infty)$ by\begin{equation}
			v(u)=\inf\left\{v\geq0\,\Bigg|\,\frac{1}{\sqrt{2\pi}}\int_\R\frac{e^{-x^2/2}dx}{(u-x)^2+v^2}\leq 1 \right\}
		\end{equation} and the function $\psi:\R\rightarrow\R$ by \begin{equation}
			\psi(u)=u+\frac{1}{\sqrt{2\pi}}\int_\R\frac{(u-x)e^{-x^2/2}dx}{(u-x)^2+v(u)^2}.
		\end{equation} Then $\psi$ is an increasing homeomorphism from $\R$ to $\R$, and the free additive convolution of the semicircle distribution and the standard Gaussian distribution has density $p: \mathbb{R} \to [0, \infty)$ with \begin{equation}
			p(\psi(u))=\frac{v(u)}{\pi}.
		\end{equation} Moreover $p$ is analytic where it is positive, and hence must be bounded. 
	\end{lemma}
	
	The first lemma below is an important technical bound in our argument and is of interest in its own right.  It demonstrates that the free convolution of the semicircle distribution with the standard Gaussian distribution is sub-Gaussian.  
	\begin{lemma}\label{lemma:Free convolution is sub-Gaussian}
		Let $p: \mathbb{R} \to [0, \infty)$ be the density of the free additive convolution of the semicircle distribution and a standard Gaussian distribution. Then there exists some constant $C>0$ such that $p(x)\leq Ce^{-x^2/2}$ for all $x\in\R$.
	\end{lemma}
	
	\begin{proof}
		Let $p,\ v,$ and $\psi$ be as in Lemma \ref{lemma:A:Biane result}.  It is straight forward to see from the definitions that $v$ is an even function of $u$, and $\psi$ (and hence $\psi^{-1}$) is odd. From now on we will assume $u>0$ and in this proof we will use asymptotic notation under $u\rightarrow\infty$.  Consider $R>0$ and curves $\gamma_1,\gamma_2,\gamma_3$ in $\C$ where $\gamma_1$ is the straight line from $(0,0)$ to $(R,0)$, $\gamma_2$ is the counterclockwise circular arc from $(R,0)$ to $(2R/\sqrt{5},R/\sqrt{5})$, and $\gamma_3$ the straight line from  $(2R/\sqrt{5},R/\sqrt{5})$ to $(0,0)$. Let $\gamma=\gamma_1\cup\gamma_2\cup\gamma_3$. The residue theorem gives for $R$ sufficiently large \begin{align*}
			\frac{1}{\sqrt{2\pi}}\oint_\gamma \frac{e^{-z^2/2}dz}{(u-z)^2+v(u)^2}&=\frac{1}{\sqrt{2\pi}}\oint_\gamma \frac{e^{-z^2/2}dz}{(z-(u+iv(u)))(z-(u-iv(u)))}\\
			&=\frac{2\pi i}{\sqrt{2\pi}}\frac{e^{-(u+iv(u))^2/2}}{2iv(u)}\\
			&=\sqrt{\frac{\pi}{2}}e^{v(u)^2/2}e^{-iuv(u)}\frac{e^{-u^2/2}}{v(u)}.
		\end{align*} Taking $R\rightarrow\infty$, it is straight forward to show\begin{equation}\label{eq:ResidueEquality}
			\sqrt{\frac{\pi}{2}}e^{v(u)^2/2}e^{-iuv(u)}\frac{e^{-u^2/2}}{v(u)}=\frac{1}{\sqrt{2\pi}}\int_0^\infty\frac{e^{-x^2/2}dx}{(u-x)^2+v(u)^2}+\frac{1}{\sqrt{2\pi}}\int_{\gamma_3'} \frac{e^{-z^2/2}dz}{(u-z)^2+v(u)^2}
		\end{equation} where $\gamma_3'=\left\{z=x+iy\in\C\,\Bigg|x\geq0,y=x/2 \right\}$ with orientation such that $\Re(z)$ is decreasing. Note an equivalent definition of $v$ is that $v(u)$ is the unique solution to \begin{equation}
		\frac{1}{\sqrt{2\pi}}\int_\R\frac{e^{-x^2/2}dx}{(u-x)^2+v(u)^2}= 1,
		\end{equation} for $u\in\R$. For the first integral on the right-hand side of \eqref{eq:ResidueEquality}, note that \begin{align}\label{eq:A:real integral bound}
			\frac{1}{\sqrt{2\pi}}\int_0^\infty\frac{e^{-x^2/2}dx}{(u-x)^2+v(u)^2}&=1-\frac{1}{\sqrt{2\pi}}\int_{-\infty}^0\frac{e^{-x^2/2}dx}{(u-x)^2+v(u)^2}, \nonumber\\
			&\geq 1-\frac{1}{2u^2+2v(u)^2}.
		\end{align} The second integral on the right-hand side of \eqref{eq:ResidueEquality} can also be bounded in modulus by $C/u^2$ for some absolute constant $C>0$. Consider this bound and \eqref{eq:A:real integral bound} in \eqref{eq:ResidueEquality} yields that there exists some bounded function $f:[0,\infty)\rightarrow\C$ such that \begin{equation*}
			\sqrt{\frac{\pi}{2}}e^{v(u)^2/2}e^{-iuv(u)}\frac{e^{-u^2/2}}{v(u)}=1-\frac{f(u)}{u^2},
		\end{equation*} and \begin{equation}\label{eq:vequality}
			v(u)=\left(1-\frac{f(u)}{u^2}\right)^{-1}\sqrt{\frac{\pi}{2}}e^{v(u)^2/2}e^{-iuv(u)}e^{-u^2/2}.
		\end{equation} It follows from Lemma \ref{lemma:A:Biane result} that $v$ is bounded, and thus we get from \eqref{eq:vequality} that $v(u)\rightarrow0$ as $u\rightarrow\infty$ and there exists some constant $C>0$ such that \begin{equation}\label{eq:v bounds}
		v(u)\leq Ce^{-u^2/2}
	\end{equation}for $u\in\R$ and absolute constant $C>0$. 
		
		We now turn our attention to $\psi(u)$, in particular $\psi(u)-u$. Using the contour $\gamma$, we get \begin{equation*}
			\frac{1}{\sqrt{2\pi}}\oint_\gamma \frac{(u-z)e^{-z^2/2}dz}{(u-z)^2+v(u)^2}=-i\sqrt{\frac{\pi}{2}}e^{v(u)^2/2}e^{-iuv(u)}e^{-u^2/2}.
		\end{equation*} Again taking $R\rightarrow\infty$ gives\begin{align*}
			\frac{1}{\sqrt{2\pi}}\int_0^\infty&\frac{(u-x)e^{-x^2/2}dx}{(u-x)^2+v(u)^2} \\
			&=-\frac{1}{\sqrt{2\pi}}\int_{\gamma_3'} \frac{(u-z)e^{-z^2/2}dz}{(u-z)^2+v(u)^2}-i\sqrt{\frac{\pi}{2}}e^{v(u)^2/2}e^{-iuv(u)}e^{-u^2/2},
		\end{align*} where the integral on the right-hand side is $O\left(\frac{1}{u}\right)$. It is also straightforward to show \begin{equation*}
			\frac{1}{\sqrt{2\pi}}\int_{-\infty}^0\frac{(u-x)e^{-x^2/2}dx}{(u-x)^2+v(u)^2}=O\left(\frac{1}{u}\right).
		\end{equation*} Thus we get $\psi(u)=u+O\left(\frac{1}{u}\right)$. Thus there exists some bounded continuous function $g$ such that \begin{equation*}
			\psi(u)=u+\frac{g(u)}{u},
		\end{equation*} and \begin{equation*}
			u=\psi\left(\psi^{-1}(u)\right)=\psi^{-1}(u)+\frac{g(\psi^{-1}(u))}{\psi^{-1}(u)}.
		\end{equation*} Solving for $\psi^{-1}(u)$, for $u$ large enough, gives\begin{equation}\label{eq:psi bounds}
			\psi^{-1}(u)=\frac{1}{2}\left(u+u\sqrt{1-\frac{4g(\psi^{-1}(u))}{u^2}}\right)=u+O\left(\frac{1}{u}\right).
		\end{equation} Combining \eqref{eq:v bounds}, \eqref{eq:psi bounds}, and Lemma \ref{lemma:A:Biane result} we see there exists some absolute constant $C'>0$ such that $p(u)\leq C'e^{-u^2/2}$ for $u\in\R$.
	\end{proof}

	Our next lemma confirms that $s(z)$ is bounded and Lipshitz continuous.  
	\begin{lemma}\label{lemma:Gaussian Stieltjes is Bounded}
		There exists constants $C,C'>0$ such that \begin{equation*}
			|s(z)|\leq C,
		\end{equation*} for all $z\in\C_+$ and $s$ is $C'$-Lipschitz continuous on $\C_+$.
	\end{lemma}
	
	\begin{proof}
		It is clear that $s(z)$ and $s'(z)$ are both uniformly bounded by $1$ for $\Im(z)\geq 1$. Let $\gamma:\R\rightarrow\C_+$ be the curve $\gamma(t)=t+2i$. Then for any $z$ such that $0<\Im(z)\leq 1$ Cauchy's integral formula (after passing from finite contours to the image of $\gamma$, similar to what was done in the proof of Lemma \ref{lemma:Free convolution is sub-Gaussian}) gives that \begin{equation*}
			s(z)=i\sqrt{2\pi}e^{-z^2/2}+\int_\gamma\frac{1}{w-z}e^{-w^2/2}dw,
		\end{equation*} and \begin{equation*}
			s'(z)=-i\sqrt{2\pi}ze^{-z^2/2}+\int_\gamma\frac{1}{(w-z)^2}e^{-w^2/2}dw.
		\end{equation*} Both of which are uniformly bounded for $0<\Im(z)\leq 1$.
	\end{proof}

	The next result establishes some simple asymptotics for $s(z)$.  
	\begin{lemma}\label{Alemma:Gaussian Asymptotic Expansion}
		The function $z \mapsto z^2\left(s(z)+\frac{1}{z} \right)$ is uniformly bounded on the strip $\{z \in \mathbb{C}:0< \Im(z)\leq 1 \}\subseteq\C_+$. 
	\end{lemma}
	
	\begin{proof}
		Let $z\in \mathbb{C}$ with $0< \Im(z)\leq 1$, and note \begin{equation*}
			z^2\left(s(z)+\frac{1}{z} \right)=\frac{1}{\sqrt{2\pi}}\int_\R\frac{zx}{x-z}e^{-x^2/2}dx.
		\end{equation*} Let $\gamma:\R\rightarrow\C_+$ be the curve which is given piecewise by \begin{equation*}
			\gamma(t)=\begin{cases}
				t-it/2,\ t\leq -4\\
				t+2i,\ -4<t<4\\
				t+it/2,\ t\geq 4
			\end{cases},
		\end{equation*}  with left to right orientation. Similar to what was done in the proof of Lemma \ref{lemma:Free convolution is sub-Gaussian} we can approximate $\R \cup \gamma$ with finite contours moving from $(-R,0)$ to $(R,0)$, then counterclockwise from $(R,0)$ to $\gamma(2R/\sqrt{5})$, then along $\gamma$ to $\gamma(-2R/\sqrt{5})$ and finally counterclockwise from $\gamma(-2R/\sqrt{5})$ to $(-R,0)$. Applying Cauchy's integral formula we see that \begin{equation*}
			\frac{1}{\sqrt{2\pi}}\int_\R\frac{zx}{x-z}e^{-x^2/2}dx=\sqrt{2\pi}iz^2e^{-z^2/2}+\frac{1}{\sqrt{2\pi}}\int_\gamma\frac{zw}{w-z}e^{-w^2/2}dw
		\end{equation*} where the right-hand side is easily seen to be uniformly bounded for $z$ with $0< \Im(z)\leq 1$.
	\end{proof}
	
	%%%%%%%%%%%%%%%%%%%% Free convolution density is sub-Gaussian %%%%%%%%%%%%%%%%%%%%%%%%%%%%%%%%%%%%%%%%%%%%%%%

	Biane \cite[Lemma 3]{MR1488333} provides a region where $s$ is Lipschitz with Lipschitz constant strictly less than $1$. The next lemma gives a description of this region. This is a key component of the stability result of Section \ref{sec:stability}.
	\begin{lemma}\label{lemma:LessthanoneLipschitz}
		Fix $t\geq1$, and define $v_{t}:\R\rightarrow[0,\infty)$ by \begin{equation*}
			v_{t}(u)=\inf\left\{v\geq0\,\Bigg|\,\frac{1}{\sqrt{2\pi}}\int_\R\frac{e^{-x^2/2}dx}{(u-x)^2+v^2}\leq \frac{1}{t} \right\}.
		\end{equation*} Let $\Omega_t=\{x+iy\in\C_+:y\geq v_{t}(x)\}$ be the region defined in \cite[Lemma 3]{MR1488333}. Then there exists a constant $C_t>0$ such that $v_{t}(x)\leq C_t e^{-x^2/2}$ for all $x\in\R$ and $s$ is Lipschitz on $\Omega_t$ with Lipschitz constant at most $\frac{1}{t}$.
	\end{lemma}
	
	\begin{proof}
		The sub-Gaussian bound on $v$ follows from a straightforward adaptation of the proof of Lemma \ref{lemma:Free convolution is sub-Gaussian}. The Lipschitz statement of $s$ is the result of Biane \cite[Lemma 3]{MR1488333}.
	\end{proof}
	
	%%%%%%%%%%%%%%%%%%%% Results on the Stieltjes Transform of Free Convolution %%%%%%%%%%%%%%%%%%%%%%%%%%%%%%%%%%%%%%%%%%%%%%%
	The next lemma on the behavior of $m(z)$ can be deduced from the proof of Lemma 3.4 in \cite{MR4058984}.
	\begin{lemma}\label{lemma:Bounds on m}
		There exists a constant $C>0$ so that $m$ is $C$-Lipschitz continuous on $\C_+$ and \begin{equation*}
			|m(z)|\leq 1
		\end{equation*} for all $z\in\C_+$. 
	\end{lemma}
	The following lemma exhibits the asymptotic behavior of $m(z)$ near infinity.  
	\begin{lemma}\label{Alemma:FC Asymptotic Expansion}
		On $\C_+$, $m$ has the asymptotic expansion at infinity given by\begin{equation*}
			m(z)=-\frac{1}{z}+O\left(\frac{1}{z^2}\right).
		\end{equation*}
	\end{lemma}
	
	\begin{proof}
		For $z$ bounded away from the real line the result follows from a straightforward application of the dominated convergence theorem. Applying Lemma \ref{Alemma:Gaussian Asymptotic Expansion} on the strip $\{z \in \mathbb{C}:0< \Im(z)\leq 1 \}$ and using that $m$ is bounded, we obtain \begin{align*}
			m(z)&=s(z+m(z))\\
			&=-\frac{1}{z+m(z)}+O\left(\frac{1}{(z+m(z))^2}\right)\\
			&=-\frac{1}{z}+O\left(\frac{1}{z^2}\right), 
		\end{align*}
		which completes the proof.  
	\end{proof}
	
	For brevity we omit the proof of the following lemma as it follows nearly identically to Lemmas \ref{Alemma:Gaussian Asymptotic Expansion} and \ref{Alemma:FC Asymptotic Expansion}.
	
	\begin{lemma}\label{Alemma:FC der Asymptotic Expansion}
		On $\C_+$, $m'$ has the asymptotic expansion at infinity given by\begin{equation*}
			m'(z)=\frac{1}{z^2}+O\left(\frac{1}{z^3}\right).
		\end{equation*}
	\end{lemma}
	
	The lemma that follows controls the imaginary part of $m(z)$ and is crucial in the iteration argument in the proof of Theorem \ref{thm:main_L}.
	\begin{lemma}\label{lem:Imm_fc is small} 
		Fix $0<\delta<2/9$. Then,\begin{equation}
			\sup_{z=E+i\eta \in \tilde S_{\delta} \cup \hat{S}_\delta}\Im m(z)=o(n^{-1/4}).
		\end{equation}
	\end{lemma}
	
	\begin{proof} Let $p$ be the density of the free additive convolution of the semicircle distribution and a standard Gaussian distribution.	Applying Lemma \ref{lemma:Free convolution is sub-Gaussian}, we have 
		\begin{align*}
			\Im m(z)&=\int_\R\frac{\eta}{(u-E)^2+\eta^2}p(u)du\\
			&\leq C\int_\R \frac{\eta}{(u-E)^2+\eta^2}e^{-u^2/2}du\\
			&\leq C\int_{|u-E|\leq E/4} \frac{\eta}{(u-E)^2+\eta^2}e^{-u^2/2}du \\
			&\qquad \qquad +C\int_{|u-E|\geq E/4} \frac{\eta}{(u-E)^2+\eta^2}e^{-u^2/2}du\\
			&\leq C'\left(\frac{E}{\eta}e^{-(3E/4)^2/2}+\frac{\eta}{E^2} \right)\\
			&\leq C''\left(\sqrt{3\log n}n^{1/4}n^{-9(1-\frac{\delta}{2})/16}+\frac{n^{-1/4}}{(2-\delta)\log n} \right)\\
			&=o(n^{-1/4}),  
		\end{align*}
		for absolute constants $C,C',C''>0$, completing the proof.  
	\end{proof}
	The next results is also utilized in the proof of Theorem \ref{thm:main_L}.
	\begin{lemma}\label{Alemma:Difference of real FC} Fix $0<\delta<\frac{1}{2}$. Then \begin{equation}
			\sup_{E+i\eta \in \tilde S_{\delta}}\left|\Re m(E+i\eta)-\Re m(E+i\sqrt{2}\eta)\right|=o\left( \frac{1}{n^{1/2}}\right).
		\end{equation}
	\end{lemma}
	
	\begin{proof}
		Let $p$ be the density of the free additive convolution of the standard semicircle measure and the standard Gaussian measure. Then\begin{equation}\label{eq:A:difference integral}
			\Re m(E+i\eta)-\Re m(E+i\sqrt{2}\eta)=\eta^2\int_\R\frac{(u-E)p(u)}{\left((u-E)^2+\eta^2 \right)\left((u-E)^2+2\eta^2 \right)}du.
		\end{equation} Let $0<c_2<c_1<1$  be functions of $E$ and $\eta$ such that $c_1E\rightarrow\infty$, $\frac{\eta}{c_2E}\rightarrow0$. We will break the right-hand side of \eqref{eq:A:difference integral} up into the integral over four subsets of $\R$. Define first the set  $I_1=\{u: |u-E|\geq c_1E \}$, where $|E-u|$ is large. Then define $I_2=\{u: c_1E\geq|u-E|\geq c_2E \}$, where $|E-u|$ is not too large, but much larger then $\eta$. Next define the set $I_3=\{u: c_2E\geq|u-E|\geq \eta \}$, where $|E-u|$ is roughly on the order of $\eta$. Finally define $I_4=\{u: |u-E|\leq \eta \}$, where $|E-u|$ is smaller than $\eta$.  Let \begin{equation*}
			f_{E,\eta}(u)=\frac{(u-E)p(u)}{\left((u-E)^2+\eta^2 \right)\left((u-E)^2+2\eta^2 \right)}.
		\end{equation*} We complete the proof by showing $\left|\int_{I_k}f_{E,\eta}(u)du \right|=o(1)$ uniformly on $\tilde S_\delta$ for $k=1,2,3,4$ for $c_1=1/\sqrt{E}$ and $c_2=\sqrt{\eta}$. Note that, by Lemma \ref{lemma:Free convolution is sub-Gaussian}, there exists $C>0$ such that $p(u)\leq Ce^{-u^2/2}$ for all $u\in\R$.
		
		$\mathbf{I_1}$: On $I_1$ we have $|f_{E,\eta}(u)|\leq\frac{p(u)}{c_1^3E^3}$. Recalling that $p$ is a probability density the conclusion is clear.
		
		$\mathbf{I_2}$: On $I_2$ we have that \begin{equation*}
			\left|\int_{I_2}f_{E,\eta}(u)du \right| \ll \frac{c_1E^2(c_1-c_2)}{c_2^4E^4}\exp\left(-(E-c_1E)^2/2\right)=o(1)
		\end{equation*} uniformly on $\tilde S_\delta$ for $\delta<1/2$.
		
		$\mathbf{I_3}$: On $I_3$ we have that \begin{equation*}
			\left|\int_{I_3}f_{E,\eta}(u)du \right|\ll \frac{c_2E(c_2E-\eta)}{\eta^4}\exp\left(-(E-c_2E)^2/2\right)=o(1)
		\end{equation*} uniformly on $\tilde S_\delta$ for $\delta<1/2$.
		
		$\mathbf{I_4}$: Finally on $I_4$ we have that \begin{equation*}
			\left|\int_{I_4}f_{E,\eta}(u)du \right|\ll \frac{1}{\eta^2}\exp\left(-(E-\eta)^2/2\right)=o(1)
		\end{equation*} uniformly on $\tilde S_\delta$ for $\delta<1$.
	\end{proof}
	
	In the final lemma of this section, we determine the existence and behavior of a solution to an equation that appears in our analysis.   
	\begin{lemma}\label{lemma:A:Existence of E}
		Let $X_n >0$ and $\eta_n>0$ be such that $X_n \rightarrow\infty$ and $\eta_n\rightarrow 0$ as $n \rightarrow \infty$. Then for any fixed $n$ there exists a solution $E_n$ to the equation\begin{equation}\label{eq:A:Zero real part}
			X_n-E_n-\Re m(E_n+i\eta_n)=0,
		\end{equation} and $E_n=X_n+\frac{1}{X_n}+O\left(\frac{1}{X_n^2}\right)$. Moreover if $\sqrt{3\log n} \geq X_n\geq\sqrt{(2-\delta)\log n}$ for some $0<\delta<1/2$, $\eta_n= n^{-1/4}$, and $E_n$ is the solution to \eqref{eq:A:Zero real part}, then\begin{equation*}
			X_n-E_n-\Re m(E_n+i\sqrt{2}\eta_n)=o(\eta_n^2).
		\end{equation*} 
	\end{lemma}
	\begin{proof}
		Existence of a solution to \eqref{eq:A:Zero real part} follows from the intermediate value theorem. The asymptotic statement is then an immediate application of Lemma \ref{Alemma:FC Asymptotic Expansion}. The final statement follows from Lemma \ref{Alemma:Difference of real FC} after noting $E_n+i\eta_n \in \tilde S_\delta$.
	\end{proof}

	\section{Stability of the fixed point equation} \label{sec:stability}
	
	In this section, we establish a stability property for approximate solutions of the fixed-point equation \eqref{def:fc} on $\tilde S_{\delta}$. Stability of similar equations was considered by the third author and Vu in \cite{MR3208886}, though the techniques used there are not applied here.
	
	\begin{theorem}[Stability] \label{thm:stability}
		For any $\delta \in (0,1)$, there exists $C > 0$ so that the following holds.  Let $\{\eps_n\}$ be a sequence of complex numbers so that $|\eps_n| \leq 1$ for all $n$ and $\eps_n = o(1)$.  Assume $\tilde m_n$ is a complex number with $\Im \tilde m_n \geq 0$ which satisfies 
		\begin{equation}\label{eq:Approximate fixed point}
			\tilde m_n=s(z+\tilde m_n)+\eps_n
		\end{equation} 
		for some $z \in \tilde S_\delta$.  Then  $\left|m(z)-\tilde m_n \right|\leq C\eps_n$ for all $n > C$.  
	\end{theorem}
	\begin{proof}
		It follows from Lemma \ref{lemma:Gaussian Stieltjes is Bounded} that $\tilde m_n$ is bounded and from Lemma \ref{lemma:Bounds on m} that $|m(z)|\leq 1$.  Let $\Omega_2$ and $v_2$ be defined as in Lemma \ref{lemma:LessthanoneLipschitz}. There exists a constant $N_\delta\in \N$ depending only on $\delta$ such that for $n\geq N_\delta$\begin{align*}
			v_{2}(\Re(z+m(z)))&\leq Ce^{-(\Re(z+m(z)))^2/2}\\
			&\leq Ce^{-(E-1)^2/2}\\
			&=Ce^{-E^2/2}e^{E}e^{-1/2}\\
			&\leq Ce^{-1/2}n^{-(1-\delta/2)}e^{\sqrt{3\log(n)}}\\
			&\leq n^{-1/4}\\
			&\leq \eta,
		\end{align*} for some absolute constant $C>0$. Thus is follows that $z+m(z)\in\Omega_2$. A similar argument shows $z+\tilde m_n \in \Omega_2$. It then follows from Lemma \ref{lemma:LessthanoneLipschitz} that \begin{align*}
			|m(z)-\tilde m_n|&\leq|s(z+m(z))-s(z+\tilde m_n)|+\eps_n\\
			&\leq\frac{1}{2}|m(z)-\tilde m_n|+\eps_n.
		\end{align*} A rearrangement completes the proof. 
	\end{proof}
	
	\section{Concentration of the Stieltjes transform} \label{sec:concentrationofstieltjes}

	This section is devoted to the following bound.  
	
	\begin{theorem}[Concentration] \label{thm:concentration}
		There exists $\delta > 0$ so that
		\[ \sup_{z = E + i \eta \in \tilde{S}_{\delta} \cup \hat{S}_{\delta}} n \eta \left| m_n(z) - \E_A m_n(z) \right| = o(1) \]
		with overwhelming probability.  
	\end{theorem}
	
	\begin{proof}
		By Proposition \ref{prop:count}, there exists a $\delta' > 0$ such that 
		\[
		\mathcal{E} = \left\{\left| \{1 \leq i \leq n: D_{ii} \geq \sqrt{(2 - \delta') \log n}\} \right| \leq  n^{1/8} \right\}
		\]
		holds with overwhelming probability.    Thus, it suffices to prove the theorem conditioned on the occurrence of $\mathcal{E}$. As $A$ is independent of $\mathcal{E}$, for the sake of brevity, we omit this conditioning from our notation. Furthermore, in the remainder of the proof, we will use $\E$ to mean $\E_A$.  
		
		We let $\delta = \delta'/2$ and begin by considering a fixed $z \in \tilde{S}_{\delta} \cup \hat{S}_\delta$.  
		Let $\E_k$ denote the conditional expectation with respect to the $\sigma$-field generated by $A_{ij}$ with $i,j > k$, so that $\E_n m_n(z) = \E m_n(z)$ and $\E_0 m_n(z) = m_n(z)$.  Thus, $m_n(z) - \E m_n(z)$ can be written as the following telescopic sum
		\[
		m_n(z) - \E m_n(z) = \sum_{k=1}^n \left( \E_{k-1} m_n(z) - \E_k m_n(z) \right) := \sum_{k=1}^n \gamma_k. 
		\]  
		The following martingale argument is inspired by a similar calculation in \cite[Chapter 6]{MR2567175}.  We let $G_k$ denote the resolvent of $L$ after removing the $k$-th row and column and $a_k$ is obtained from the $k$-th column of $L$ by removing the $k$-th entry.   
		We then have that
		\begin{align*}
			\gamma_k &= \frac{1}{n}(\E_{k-1} \tr (L-z)^{-1} - \E_k \tr (L-z)^{-1} )  \\
			&= \frac{1}{n} \Big(\E_{k-1} \big[ \tr (L-z)^{-1} - (L_k - z)^{-1} \big] - \E_k \big[ \tr (L-z)^{-1} - \tr (L_k - z)^{-1} \big] \Big) \\
			&= \frac{1}{n} (\E_{k-1} - \E_k) \Bigg(\frac{a_k^* G_k^{2} a_k - \E_{a_k} a_k^* G_k^{2} a_k}{L_{kk} - z - a_k^* G_k a_k} + \frac{1 + \E_{a_k} a_k^* G_k^{2} a_k }{L_{kk} - z - a_k^* G_k a_k} \\
			&\hspace{7cm}- \frac{1 + \E_{a_k} a_k^* G_k^{2} a_k}{L_{kk} - z - \E_{a_k} a_k^* G_k a_k} \Bigg) \\
			&= \frac{1}{n} (\E_{k-1} - \E_k) \Bigg(\frac{a_k^* G_k^{2} a_k - \E_{a_k} a_k^* G_k^{2} a_k}{L_{kk} - z - a_k^* G_k a_k} \\
			&\hspace{4cm} - \frac{(1 + \E_{a_k}a_k^* G_k^{2} a_k) (a_k^* G_k a_k - \E_{a_k} a_k^* G_k a_k)}{(L_{kk} - z -  a_k^* G_k a_k)(L_{kk} - z - \E_{a_k} a_k^* G_k a_k)}\Bigg) \\
			&= \frac{1}{n} (\E_{k-1} - \E_k) \Bigg(\frac{a_k^* G_k^{2} a_k - \frac{1}{n} \tr G_k^{2} }{L_{kk} - z - a_k^* G_k a_k} \\
			&\hspace{4cm} - \frac{(1 + \frac{1}{n} \tr G_k^{2} ) (a_k^* G_k a_k - \frac{1}{n}\tr G_k )}{(L_{kk} - z -  a_k^* G_k a_k)(L_{kk} - z - \frac{1}{n} \tr G_k )}\Bigg), 
		\end{align*}
		where the third equality follows from Lemma \ref{klem:tracedifference} and $\E_{a_k}$ indicates expectation over the random variables in $a_k$.  
		
		We define the following quantities,
		\[
		\alpha_k = a_k^* G_k^{2} a_k - \frac{1}{n} \tr G_k^{2}, 
		\]
		\[
		\beta_k = \frac{1}{ L_{kk} - z -  a_k^* G_k a_k}, \quad \bar{\beta}_k = \frac{1}{L_{kk} - z - \frac{1}{n} \tr G_k}, 
		\]
		\[
		b_k = \frac{1}{L_{kk} - z- \frac{1}{n} \E\tr G_k}
		\]
		\[
		\delta_k = a_k^* G_k a_k - \frac{1}{n}\tr G_k, \quad \hat{\delta}_k = a_k^* G_k a_k - \frac{1}{n}\E\tr G_k
		\]
		\[
		\epsilon_k = 1 + \frac{1}{n} \tr G_k^{2} ,
		\]
		so that
		\begin{align*}
			m_n(z) - \E m_n(z) &= \frac{1}{n} \sum_{k=1}^n (\E_{k-1} - \E_k) \alpha_k \beta_k  - \frac{1}{n} \sum_{k=1}^n (\E_{k-1} - \E_k) \epsilon_k \delta_k \beta_k \bar{\beta}_k \\
			&:= S_1 - S_2.
		\end{align*}
		
		We will show that $n \eta |S_1| = o(1)$ and $n \eta |S_2| = o(1)$  with overwhelming probability uniformly for all $z \in \tilde{S}_\delta \cup \hat{S}_\delta$.  This will be done via the method of moments.  We begin with $S_1$.  By Markov's inequality, it suffices to bound $\E| \eta n S_1|^{2 \ell} = \E |\eta \sum_{k =1}^n (\E_{k-1} - \E_k) \alpha_k \beta_k|^{2 \ell}$ for $\ell \in \N$.  
		
		By Lemma \ref{klem:burkholder2}, for any $\ell \geq 1$,  
		\[
		\E \left|\eta \sum_{k =1}^n (\E_{k-1} - \E_k) \alpha_k \beta_k \right|^{2 \ell} \leq K_{\ell} \left( \E \left(\sum_{k=1}^n \E_{k} |\eta \alpha_k \beta_k|^2 \right)^{\ell} + \sum_{k=1}^n \E  |\eta \alpha_k \beta_k|^{2 \ell} \right).
		\]
		We use $K_\ell$ to indicate a constant that only depends on $\ell$, but may change from line to line.  Since $\Im a^*_k G_k a_k > 0$,
		\[
		|\beta_k| \leq \eta^{-1}.
		\]  
		Therefore,
		\begin{equation} \label{keq:S1}
			\E \left|\eta \sum_{k=1}^n (\E_{k-1} - \E_k) \alpha_k \beta_k \right|^{2 \ell} \leq K_{\ell} \left( \E \left(\sum_{k=1}^n \E_{k} |\eta \alpha_k \beta_k|^2 \right)^{\ell} + \sum_{k=1}^n \E  |\alpha_k|^{2 \ell} \right).
		\end{equation}
		By Lemma \ref{klem:quadraticform}, 
		\[
		\E|\alpha_k|^{2 \ell} \leq K_{\ell} n^{-2 \ell}  \E|\tr G_k^{2} G_k^{*2}|^\ell.
		\]
		Let $\mathcal{E}'$ denote the event that $\|A\| \leq 10$.  We have already seen that $\mathcal{E}'$ occurs with overwhelming probability.     
		Note that on the event $\mathcal{E} \cap \mathcal{E}'$, due to Cauchy's interlacing theorem and Weyl's inequality, with overwhelming probability, the eigenvalues of $L$ after removing the $k$-th row and $k$-th column will have at most $n^{1/4}$ eigenvalues larger than $\sqrt{(2- 2 \delta) \log n}$ (see Proposition \ref{prop:count}) for $n$ sufficiently large.  Thus, we have 
		\begin{align} \label{keq:indicator}
			 \tr G_k^2 G_k^{*2} &=  \left(\sum_{i=1}^n \frac{1}{((\lambda_i - E)^2 + \eta^2)^2} \right) \oindicator{\mathcal{E}'} +  \left(\sum_{i=1}^n \frac{1}{((\lambda_i - E)^2 + \eta^2)^2} \right) \oindicator{\bar{\mathcal{E}'}} \nonumber\\
			&\leq n^{1/4} \eta^{-4} + O_\delta(n \log^{-2} n) + n \eta^{-4}  \oindicator{\mathcal{E}'^c} \nonumber\\
			&\leq 2 n^{5/4} + n^{2}  \oindicator{\mathcal{E}'^c}.
		\end{align}
		The same argument establishes that with overwhelming probability, 
		\begin{equation} \label{keq:gkgk bound}
		 \tr G_k G_k^* \leq n + n^2 \oindicator{\mathcal{E}'^c}.
		\end{equation}
		
		We now have that
		\begin{align*}
			\E|\alpha_k|^{2 \ell} &\leq K_{\ell} n^{-2 \ell} \E | 2 n^{5/4} + n^{2}  \oindicator{\mathcal{E}'^c}|^\ell \\
			&\leq K_{\ell} n^{-2 \ell} (n^{5 \ell/4} + n^{2 \ell} \E \oindicator{\mathcal{E}'^c}) \\
			&= K_{\ell} n^{-2 \ell} (n^{5 \ell/4} + n^{2 \ell} \P(\mathcal{E}'^c) ) \\
			&\leq K_{\ell} n^{-3 \ell/4}
		\end{align*}
		where the last line follows from the observation that since $\mathcal{E}'$ occurs with overwhelming probability, $\P(\mathcal{E}'^c) = O_{\ell}(n^{-100 \ell})$, say.
		Therefore, by equation \eqref{keq:S1},
		\begin{equation*}
			\E \left|\eta \sum_{k=1}^n (\E_{k-1} - \E_k) \alpha_k \beta_k \right|^{2 \ell} \leq  K_{\ell} \left( \E \left(\sum_{k=1}^n \E_{k} |\eta \alpha_k \beta_k|^2 \right)^{\ell} + n^{-3\ell/4 + 1} \right). 
		\end{equation*}
		
		We now direct our attention to the remaining sum on the right-hand side.  Observe that
		\begin{align*}
			|\alpha_k \beta_k| &\leq \left| \frac{a_k^* G_k^{2} a_k}{ L_{kk} - z -  a_k^* G_k a_k} \right|+ \left| \frac{\frac{1}{n} \tr G_k^{2}}{ L_{kk} - z -  a_k^* G_k a_k} \right| \\
			&\leq  \frac{|a_k^* G_k^{2} a_k|}{ |\Im (L_{kk} - z -  a_k^* G_k a_k)|} +  \left|\frac{\frac{1}{n} \tr G_k^{2}}{ L_{kk} - z -  a_k^* G_k a_k} \right| \\
			&\leq \eta^{-1} + \eta^{-3}, 
		\end{align*}
		where the last inequality follows from the observation that
		\[
		|a_k^* G_k^{2} a_k| \leq 1 + a_k^* (L_k - z)^{-1} (L_k - \bar{z})^{-1} a_k = -\eta^{-1} \Im(L_{kk} - z - a^*_k G_k a_k) 
		\]
		where $L_k$ denotes the matrix $L$ with the $k$-th row and column removed.
		Thus, for a fixed constant $K_0 > 0$,
		\[
		|\alpha_k \beta_k|^2 \leq 4 K_0^2 \alpha_k^2 + 4 \eta^{-6} \oindicator{|\beta_k| \geq 2 K_0}. 
		\]
		We again have by Lemma \ref{klem:quadraticform} and \eqref{keq:indicator} that for some constant $K > 0$, 
		\[
		\E_k |\alpha_k|^2 \leq \frac{K}{n^2} \E_k \tr G_k^2 G_k^{*2} \leq K (n^{-3/4} + \E_k \oindicator{\mathcal{E}'^c}) .
		\]
		
		We now return to bounding $S_1$. Let $I$ denote the indices such that $D_{ii} \geq \sqrt{(2 - \delta') \log n}$.  On the event $\mathcal{E}$, $|I| \leq n^{1/8}$, so we have that   
		\begin{align} \label{keq:S1bound}
			\P(n \eta |S_1| \geq \varepsilon) &\leq K_{\ell} \left( \E \left(\sum_{k=1}^n \E_{k} |\eta \alpha_k \beta_k|^2 \right)^{\ell} +  n^{-3\ell/4 + 1}\right) \nonumber \\
			&\leq K_{\ell} \left( \E \left(\sum_{k \in I} \E_{k} |\eta \alpha_k \beta_k|^2 \right)^{\ell} + \left(\sum_{k \notin I} \E_{k} |\eta \alpha_k \beta_k|^2 \right)^{\ell} +n^{-3\ell/4 + 1}\right) \nonumber \\
			&\leq K_{\ell} \left( \E \left(\sum_{k \notin I} \E_{k} (4 K_0^2 \alpha_k^2 \eta^2 + 4 \eta^{-4} \oindicator{|\beta_k| \geq 2 K_0}) \right)^{\ell} +    n^{-5\ell/8}\right) \nonumber\\
			&\leq  K_{\ell} \left( \eta^{2 \ell} n^{\ell/4} + \eta^{-4 \ell} n^{\ell-1} \sum_{k \notin I} \P(|\beta_k| \geq 2 K_0) +  n^{-5\ell/8}\right) \nonumber\\
			&\leq  K_{\ell} \left(  \eta^{-4 \ell} n^{\ell-1} \sum_{k \notin I} \P(|\beta_k| \geq 2 K_0) +   n^{-\ell/4}\right),
		\end{align}
		where in the third inequality we have utilized the calculation
		\begin{align} \label{keq:indicatorcalculation}
		\E \left(\sum_{k \in I} \E_{k} |\eta \alpha_k \beta_k|^2 \right)^{\ell} &\leq \E \left(\sum_{k \in I} \E_{k} |\alpha_k|^2 \right)^{\ell} \nonumber \\
		&\leq \E \left(\sum_{k \in I} K (n^{-3/4} + \E_k \oindicator{\mathcal{E}'^c}) \right)^{\ell} \nonumber  \\
		&\leq K_{\ell} \left(n^{-5 \ell/8} + \E \left(\sum_{k \in I} \E_k \oindicator{\mathcal{E}'^c} \right)^\ell \right) \nonumber \\
		&\leq  K_{\ell} \left(n^{-5 \ell/8} + n^{(\ell-1)/8} \E \left(\sum_{k \in I} \E_k \oindicator{\mathcal{E}'^c} \right) \right) \nonumber \\
		&\leq K_{\ell} \left(n^{-5 \ell/8} + n^{(\ell-1)/8} \sum_{k \in I} \P(\mathcal{E}'^c) \right) \nonumber \\
		&\leq K_{\ell} n^{-5 \ell/8}.
		\end{align}
		A nearly identical calculation justifies the fourth inequality in \eqref{keq:S1bound}.

		To control $\P(|\beta_k| \geq 2 K_0)$ in \eqref{keq:S1bound}, we make the following observation.  For $k \notin I$, 
		\[
		|b_k| \leq \frac{1}{|\Re(L_{kk} - z- \frac{1}{n} \E\tr G_k)|} \leq \frac{1}{(\delta/10)\sqrt{ \log n} - |\frac{1}{n} \E\tr G_k|} \leq K_0.
		\]  
		Therefore, if $|\beta_k| \geq 2 K_0$ then,
		\begin{align*}
			|L_{kk} - z - a_k^* G_k a_k| &= |L_{kk} -z - \frac{1}{n} \E\tr G_k + \frac{1}{n} \E\tr G_k - a_k^* G_k a_k| \\
			&= |b_k^{-1} - \hat{\delta}_k| \\
			&\leq \frac{1}{2 K_0}, 
		\end{align*}
		which implies that $\hat{\delta}_k \geq \frac{1}{2 K_0}$.  
		
		Thus, continuing from \eqref{keq:S1bound},
		\[
		\P(n \eta |S_1| \geq \varepsilon) \leq  K_{\ell} \left( \eta^{-4 \ell} n^{\ell-1} \sum_{k \notin I} \P(|\hat{\delta}_k| \geq 1/2 K_0)  + n^{-\ell/4} \right).
		\]
		
		A straight-forward calculation shows that 
		\begin{equation} 
			\E|\delta_k - \hat{\delta}_k|^{8 \ell} \leq K_{\ell} n^{-4 \ell}.
		\end{equation}
		As the proof is similar to the many calculations we have done above, it is omitted.  One can find the analogous argument in Section 6.2.3 of \cite{MR2567175}.
		
		By Markov's inequality and equation \eqref{keq:gkgk bound}, 
		\begin{align*}
		\P\left(|\hat{\delta}_k| \geq \frac{1}{2 K_0}\right) &\leq \frac{\E |\hat{\delta}_k|^{8\ell}}{(2 K_0)^{8 \ell}} \\
		&\leq K_{\ell} (\E |\hat{\delta}_k - \delta_k|^{8\ell} + \E |\delta_k|^{8 \ell}) \\
		&\leq K_{\ell} (n^{-4 \ell} + \E |\delta_k|^{8 \ell}) \\
		&\leq \E |a_k^* G_k a_k - \frac{1}{n} \E \tr G_k|^{8 \ell} \\
		&\leq K_{\ell} n^{-4 \ell}.
		\end{align*}
		We then have that 
		\[
		\P(n \eta |S_1| \geq \varepsilon) \leq  K_{\ell} \left( \eta^{-4 \ell} n^{-3\ell}  + n^{-\ell/4} \right) \leq K_{\ell} n^{-\ell/4}.
		\]
		As $\ell$ can be arbitrarily large, we have established that $n \eta |S_1| = o(1)$ uniformly in $z$, with overwhelming probability.

		We now show that $n \eta |S_2| = o(1)$ with overwhelming probability. Again, the argument is quite similar to the above, so we will be more brief.  By Markov's inequality,
		\[
		\P(n \eta |S_2| \geq \varepsilon) \leq \frac{\E \Bigg| \eta \sum_{k=1}^n (\E_{k-1} - \E_k) \epsilon_k \delta_k \beta_k \bar{\beta}_k \Bigg|^{2 \ell} }{\varepsilon^{2 \ell}}.
		\]
		By the same argument in \eqref{keq:indicator}, we have that $|\epsilon_k| \leq 2 + n^2 \oindicator{\mathcal{E}'^c}$.  As $\mathcal{E}'$ occurs with overwhelming probability, we omit the $\oindicator{\mathcal{E}'^c}$ term in the remainder of this calculation as it is negligible as seen from \eqref{keq:indicatorcalculation}.
		
		We continue with 
		\begin{align*}
			\E \Bigg| \eta \sum_{k=1}^n &(\E_{k-1} - \E_k) \epsilon_k \delta_k \beta_k \bar{\beta}_k \Bigg|^{2 \ell} \\
			&\leq K_{\ell}   \left( \E \left( \sum_{k=1}^n \E_k |\eta \epsilon_k \delta_k \beta_k \bar{\beta}_k|^2 \right)^\ell + \sum_{k=1}^n \E |\eta \epsilon_k \delta_k \beta_k \bar{\beta}_k|^{2 \ell} \right)  \\
			&\leq  K_{\ell}   \left( \E \left( \sum_{k=1}^n \E_k |\eta  \delta_k \beta_k \bar{\beta}_k|^2 \right)^\ell + \eta^{-2 \ell} n^{-\ell + 1}  \right)  \\
			&\leq K_{\ell}   \left( \ \E \left(\sum_{k \notin I} \E_k |\eta  \delta_k \beta_k \bar{\beta}_k|^2 \right)^\ell+ \eta^{-2 \ell} n^{-7\ell/8 }  \right) \\
			&\leq K_{\ell}   \left( \ \E \left(\sum_{k \notin I} \E_k |\eta  \delta_k \beta_k \bar{\beta}_k|^2 \right)^\ell+  n^{-3\ell/8 }  \right). 
		\end{align*}
		The summands on the right-hand side can be bounded as follows: 
		\begin{align*}
			|\delta_k \beta_k \bar{\beta}_k|^2 &\leq (2 K_0)^4 |\delta_k|^2 + \eta^{-4} |\delta_k^2| \oindicator{|\beta_k \bar{\beta}_k| \geq (2 K_0)^2} \\
			&\leq  (2 K_0)^4 |\delta_k|^2 + \eta^{-4} |\delta_k^2| \oindicator{|\delta_k| \geq 1/(4 K_0) \text{ or } |\hat{\delta}_k \geq 1/(4 K_0)}. 
		\end{align*}
		The final inequality is due to the observation that on the event that
		\[
		|L_{kk} - z - a_k^* G_k^{-1} a_k| |L_{kk} - z - \frac{1}{n} \tr G_k^{-1} | \leq (2 K_0)^2,
		\]
		either
		\[
		|L_{kk} - z - a_k^* G_k^{-1} a_k|  \leq 2 K_0
		\]
		or 
		\[
		|L_{kk} - z - \frac{1}{n} \tr G_k^{-1} |  \leq 2 K_0.
		\]
		If $|L_{kk} - z - a_k^* G_k^{-1} a_k|  \leq 2 K_0$, then on the event that $|b_k| \leq K_0$ we have that $|\hat{\delta}_k| \geq 1 /(2 K_0)$ as before.  On the other hand, if $|L_{kk} - z - \frac{1}{n} \tr G_k^{-1} |  \leq 2 K_0$ then
		\[
		|b_k^{-1} - \delta_k| \leq \frac{1}{2 K_0}
		\]
		which implies that $|\delta_k| \geq 1/(2 K_0)$.  Returning to the moment calculation, we have
		\begin{align*}
			\E &\left| \eta \sum_{k=1}^n (\E_{k-1} - \E_k) \epsilon_k \delta_k \beta_k \bar{\beta}_k \right|^{2 \ell}  \\
			&\leq K_{\ell}   \left(  \E \left(\sum_{k \notin I} \E_k  \eta^2 |\delta_k|^2 + \E_k \eta^{-2} |\delta_k^2| \oindicator{|\delta_k| \text{ or } |\hat{\delta}_k| \geq 1/(2 K_0)}\right)^\ell+ n^{-3\ell/8 }   \right)  \\
			&\leq K_{\ell}   \left(  \eta^{2 \ell}  + \E \left(\sum_{k \notin I} \E_k \eta^{-2} |\delta_k^2| \oindicator{|\delta_k|  \text{ or } |\hat{\delta}_k |\geq 1/(2 K_0)}\right)^\ell+ n^{-3\ell/8 }   \right) \\
			&\leq K_{\ell}   \left(   \eta^{-2 \ell} n^{\ell -1}  \sum_{k \notin I} \E  |\delta_k|^{2 \ell} \oindicator{|\delta_k|  \text{ or } |\hat{\delta}_k| \geq 1/(2 K_0)} +n^{-3\ell/8 }  \right) \\
			&\leq K_{\ell}   \Bigg(   \eta^{-2 \ell} n^{\ell -1}  \sum_{k \notin I} (\E  |\delta_k|^{4 \ell})^{1/2} (\P(|\delta_k| \geq 1/(2 K_0)) + \P( |\hat{\delta}_k| \geq 1/(2 K_0))^{1/2} \\
			&\qquad \qquad   +n^{-3\ell/8 }  \Bigg) \\
			&\leq K_{\ell}   \Bigg(   \eta^{-2 \ell} n^{\ell -1}  \sum_{k \notin I} \sqrt{\E  |\delta_k|^{4 \ell}} \sqrt{\E|\delta_k|^{2 \ell} +\E|\hat{\delta}_k|^{2 \ell} } +n^{-3\ell/8 } \Bigg) \\
			&\leq K_{\ell}   \left(  \eta^{-2 \ell} n^{-\ell} +  n^{-3\ell/8 } \right) \\
			&\leq K_{\ell} n^{-3\ell/8 }. 
		\end{align*}
		Thus, we have shown that for a fixed $z \in \tilde{S}_\delta \cup \hat{S}_\delta$, $n \eta | m_n(z) - \E m_n(z)| = o(1)$ with overwhelming probability where the probability and the $o(1)$ error are uniform in $z$.  To extend this result to all $z \in \tilde{S}_\delta \cup \hat S_{\delta}$, we observe that $m_n(z)$ is $n^2$-Lipschitz so it suffices to establish the result for an $n^{-3}$-net of $\tilde{S}_\delta \cup \hat{S}_\delta$.  Such a net will be of size at most $O(n^4)$ and since an event that holds with overwhelming probability can tolerate a polynomial-sized union bound, the result is proved.  
	\end{proof}

	\section{Proof of Theorem \ref{thm:stieltjes_transforms}} \label{sec:proofofmaintechnical}
	
	We now turn to the proof of Theorem \ref{thm:stieltjes_transforms}.  We begin with some preliminary results and notation we will need for the proof. For convenience, we establish the result for $\tilde{S}_\delta$ as an identical argument applies to $\hat{S}_\delta$. 
	
	Fix $\eps \in (0, 1/100)$.  Since $A$ is a GOE matrix, it follows from standard norm bounds (see, for example, \cite{MR1863696,MR2963170} and references therein) that $\|A\| \leq 3$ with overwhelming probability.  
	Thus, by Proposition \ref{prop:count} and Weyl's perturbation theorem (see, for instance, Corollary III.2.6 in \cite{MR1477662}), there exists $\delta > 0$ so that the event
	\begin{equation} \label{eq:s:Leig}
		\left\{ | \{ 1 \leq i \leq n : \lambda_i(L) \geq \sqrt{(2 - 2\delta) \log n} \}| = O(n^{\eps}) \right\} 
	\end{equation}
	holds with overwhelming probability.  Moreover, by taking $\delta$ smaller if needed, we can also ensure the same value of $\delta$ applies to the bounds in Theorem \ref{thm:concentration} and Proposition \ref{prop:count}.  For convenience, we define the event
	\begin{align*}
		\mathcal{E} := &\left\{ | \{ 1 \leq i \leq n : \lambda_i(L) \geq \sqrt{(2 - 2\delta) \log n} \}| = O(n^{\eps}) \right\} \\
		&\qquad \bigcap \left\{ \sup_{z = E + i \eta \in \tilde S_{\delta}} n \eta \left| m_n(z) - \E_A m_n(z) \right| = o(1) \right\} \\
		&\qquad \bigcap \left\{\left| \{ 1 \leq i \leq n : D_{ii} \geq \sqrt{(2- 2\delta) \log n} \} \right| = O_{\eps}(n^{\eps}) \right\}
	\end{align*}
	to be the intersection of the events from \eqref{eq:s:Leig},  Theorem \ref{thm:concentration}, and Proposition \ref{prop:count}.  It follows that $\mathcal{E}$ holds with overwhelming probability, and we will work on this event throughout the proof.  
	For the remainder of the proof, we consider $\eps$ and $\delta$ fixed.  We will allow implicit constants in our asymptotic notation to depend on these parameters without denoting this dependence.  
	
	Since the event $\mathcal{E}$ holds with overwhelming probability, we will often be able to insert or remove the indicator function $\oindicator{\mathcal{E}}$ into the expected value with only negligible error.  For example, using the naive bound, $\sup_{i,j} |G_{ij}(z)| \leq \|G(z)\| \leq \frac{1}{\eta}$, we have
	\begin{align*}
		\E \sum_{i,j} G_{ij}(z) = \E \sum_{i,j} G_{ij}(z) \oindicator{\mathcal{E}} + \E \sum_{i,j} G_{ij}(z) \oindicator{\mathcal{E}^c}, 
	\end{align*}
	where
	\[ \left| \E \sum_{i,j} G_{ij}(z) \oindicator{\mathcal{E}^c} \right| \leq \frac{n^2}{\eta} \Prob(\mathcal{E}^c) = O_p \left( \frac{1}{n^p \eta} \right) \]
	for any $p > 0$.  Here, we use the convention that all indices in the sums are over $[n]$ unless otherwise noted.  In particular, taking $p$ sufficiently large shows that
	\[ \E \sum_{i,j} G_{ij}(z) = \E \sum_{i,j} G_{ij}(z) \oindicator{\mathcal{E}} + o \left( \frac{1}{n \eta} \right) \]
	for any $z$ in the spectral domain $S_\delta$.  
	In a similar way, one can apply the same procedure to the conditional expectation $\E_A \sum_{i,j} G_{ij}(z)$, where the error term
	\[ \E_A \sum_{i,j} G_{ij}(z) \oindicator{\mathcal{E}^c} \]
	can be bounded with overwhelming probability using an $L^1$-norm bound as above and Markov's inequality.  
	We will often insert and remove indicator functions of events that hold with overwhelming probability in this way.  As the arguments are all of a similar format, we will not always show all of the details, and we refer to this as procedure as ``inserting the indicator function using naive bounds.''
	
	Theorem \ref{thm:stieltjes_transforms} focuses on the spectral domain $\tilde S_{\delta} \cup \hat S_{\delta}$.  However, it will sometimes be convenient to work on the larger spectral domain $S_{\delta}$.  We begin with some initial bounds for $m_n$ in the spectral domain $S_\delta$.  
	
	\begin{lemma} \label{lemma:s:mn}
		With overwhelming probability, 
		\[ \sup_{z \in S_{\delta}} \E |m_n(z)| + \sup_{z \in S_{\delta}} \E_A |m_n(z)| + \sup_{z \in S_{\delta}} |m_n(z)| = O\left( \frac{1}{ \sqrt{\log n}} \right). \]
	\end{lemma}
	\begin{proof}
		We start by bounding 
		\[ \sup_{z \in S_{\delta}} \E |m_n(z)|. \]
		By inserting the indicator function of $\mathcal{E}$ using naive bounds, it suffices to bound
		\[ \sup_{z \in S_{\delta}} \E |m_n(z)\oindicator{\mathcal{E}}|. \]
		For $z = E + i \eta \in S_\delta$, we have the uniform bound
		\begin{align*}
			\E \left| m_n(z)\oindicator{\mathcal{E}}\right| &\leq \E \frac{1}{n} \sum_{j=1}^n \frac{1}{ \sqrt{ ( \lambda_j(L) - E)^2 + \eta^2 }} \oindicator{\mathcal{E}}\\
			&\leq \E \frac{1}{n} \sum_{j : \lambda_j(L) \geq \sqrt{(2 - 2\delta)\log n}} \frac{1}{\eta}\oindicator{\mathcal{E}}  + \E \frac{1}{n} \sum_{j : \lambda_j(L) < \sqrt{(2 - 2\delta)\log n}} \frac{1}{ |\lambda_j(L) - E| }\oindicator{\mathcal{E}} \\
			&\ll \frac{n^\eps}{n \eta} + \frac{n}{n \sqrt{\log n}}, 
		\end{align*}
		where in the first sum we used that, on the event $\mathcal{E}$, there are only $O(n^{\eps})$ terms in the sum and in the second sum we bounded the number of summands by $n$ and used that $\lambda_j(L) < \sqrt{(2 - 2\delta)\log n}$ while $E \geq \sqrt{(2 - \delta) \log n}$.  Thus, since $\eta \geq n^{-1/4}$ and $\eps < 1/100$, we conclude that 
		\[ \sup_{z \in S_{\delta}} \E |m_n(z)| = O \left( \frac{1}{\sqrt{\log n}} \right). \]
		
		The same method also bounds $\sup_{z \in S_{\delta}} \E_A |m_n(z)|$, where we can again use naive bounds to insert the the indicator function of $\mathcal{E}$ into the conditional expectation.  In this case, one also needs to use a net argument and continuity to establish the result for all $z \in S_{\delta}$; we omit the details.  The bound for $\sup_{z \in S_{\delta}} |m_n(z)|$ is also similar.  However, in this case, we do not have any expectation and can simply repeat the arguments above by working on the event $\mathcal{E}$, which holds with overwhelming probability; we omit the details.  
	\end{proof}
	
	With Theorem \ref{thm:concentration} and Lemma \ref{lemma:s:mn} in hand, we can now complete the proof of Theorem \ref{thm:stieltjes_transforms}.  
	
	\begin{proof}[Proof of Theorem \ref{thm:stieltjes_transforms}]
		For $z \in S_{\delta}$, we define 
		\[ G := G(z) \]
		and 
		\[ Q := Q(z + \E_A m_n(z)), \]
		where $G(z) := (L-z)^{-1}$ is the resolvent of $L$ and $Q(z) := (D - z)^{-1}$ is the resolvent of $D$.  In particular, $s_n(z + \E_A m_n(z)) = \frac{1}{n} \tr Q = \E_A \frac{1}{n} \tr Q$ since $Q$ does not depend on $A$.  By the resolvent identity \eqref{eq:resolvent}, we have
		\begin{align} \label{eq:s:res_start}
			\E_A m_n(z) - s_n(z + \E_A m_n(z)) &= \E_A \frac{1}{n} \tr G - \frac{1}{n} \tr Q \\
			&=  \E_A \frac{1}{n} \tr (GAQ) - \E_A m_n(z) \E_A \frac{1}{n} \tr (GQ). \nonumber
		\end{align}
		We will apply the Gaussian integration by parts formula \eqref{eq:ibp} to 
		\begin{equation} \label{eq:s:ibp}
			\E_A \frac{1}{n} \tr (GAQ) = \frac{1}{n} \sum_{i,j} Q_{ii} \E_A [ G_{ij} A_{ji} ], 
		\end{equation} 
		where we used the fact that $Q$ is a diagonal matrix and does not depend on $A$ (so we can pull it out of the conditional expectation).  A simple computation involving the resolvent identity \eqref{eq:resolvent} shows that 
		\[ \frac{ \partial G_{kl} } { \partial A_{ij} } = \begin{cases} 
			G_{ki} G_{jl} + G_{kj} G_{il}, & \text{ if } i \neq j, \\
			G_{ki} G_{jl}, & \text{ if } i = j. 
		\end{cases}
		\] 
		Thus, returning to \eqref{eq:s:ibp} and applying \eqref{eq:ibp}, we obtain 
		\[ \E_A \frac{1}{n} \tr (GAQ) = \frac{1}{n^2} \E_A \sum_{i,j} Q_{ii} G_{ij}^2 + \frac{1}{n} \E_A m_n(z) \tr (QG),  \]
		which when combined with \eqref{eq:s:res_start} yields
		\begin{align} \label{eq:s:res_mid}
			\E_A m_n(z) - s_n(z + \E_A m_n(z)) &=  \frac{1}{n^2} \E_A \sum_{i,j} Q_{ii} G_{ij}^2 + \frac{1}{n} \E_A m_n(z) \tr (QG) \\
			&\qquad\qquad - \E_A m_n(z) \E_A \frac{1}{n} \tr (GQ). \nonumber
		\end{align}
		We aim to bound the terms on the right-hand side uniformly for $z \in \tilde{S}_\delta$.  
		
		For the first term, we apply the Ward identity \eqref{eq:ward} to get
		\begin{align*}
			\left| \frac{1}{n^2} \E_A \sum_{i,j} Q_{ii} G_{ij}^2 \right| &\leq \E_A \frac{1}{n^2} \sum_{i} | Q_{ii} | \sum_{j} |G_{ij}|^2 \\
			&\leq \E_A \frac{1}{n^2 \eta} \sum_{i} |Q_{ii}| \Im G_{ii} \\
			&\leq \E_A \frac{1}{n^2 \eta} \sum_{i} \frac{1}{ | D_{ii} - z - \E_A m(z) | } \Im G_{ii}. 
		\end{align*}
		Define the event 
		\begin{equation} \label{eq:s:F}
			\mathcal{F} := \left\{ \sup_{z \in \tilde S_{\delta}} \E_A m_n(z) = O \left( \frac{1}{\sqrt{ \log n}} \right) \right\}, 
		\end{equation}
		and note that, by Lemma \ref{lemma:s:mn}, $\mathcal{F}$ holds with overwhelming probability.  By using naive bounds, we can insert the indicator function of the event $\mathcal{E} \cap \mathcal{F}$ into the above equations to obtain
		\begin{align*}
			&\E_A \frac{1}{n^2 \eta} \sum_{i} \frac{1}{ | D_{ii} - z - \E_A m(z) | } \Im G_{ii} \\
			&= \E_A \frac{1}{n^2 \eta} \sum_{i} \frac{1}{ | D_{ii} - z - \E_A m(z) | } \Im G_{ii} \oindicator{\mathcal{E} \cap \mathcal{F}} + o \left( \frac{1}{n \eta} \right) \\
			&= \E_A \frac{1}{n^2 \eta}  \sum_{i: D_{ii} \geq \sqrt{ (2 - 2\delta) \log n}} \frac{1}{ | D_{ii} - z - \E_A m(z) | } \Im G_{ii} \oindicator{\mathcal{E} \cap \mathcal{F}} \\
			&\qquad +  \E_A \frac{1}{n^2 \eta}\sum_{i: D_{ii} < \sqrt{(2 - 2 \delta) \log n}} \frac{1}{ | D_{ii} - z - \E_A m(z) | } \Im G_{ii} \oindicator{\mathcal{E} \cap \mathcal{F}} + o \left( \frac{1}{n \eta} \right)
		\end{align*}
		with overwhelming probability.  
		
		Observe that since $\Re(z) \geq \sqrt{(2 - \delta) \log n}$, one has 
		\[ | D_{ii} - z - \E_A m(z) | \gg \sqrt{\log n} \]
		on the event $\mathcal{F}$ whenever $D_{ii} < \sqrt{(2 - 2 \delta) \log n}$.  Thus, we have   
		\begin{align*}
			\E_A \frac{1}{n^2 \eta} &\sum_{i: D_{ii} < \sqrt{(2 - 2 \delta) \log n}} \frac{1}{ | D_{ii} - z - \E_A m(z) | } \Im G_{ii} \oindicator{\mathcal{E} \cap \mathcal{F}} \\
			&\ll \frac{1}{n \eta \sqrt{\log n}} \E_A \Im m_n(z) \oindicator{\mathcal{E} \cap \mathcal{F}} \\
			&= o \left( \frac{1}{n \eta} \right).  
		\end{align*}
		For the other term, we apply the naive bounds $\|G\| \leq \frac{1}{ \eta}$ and $ \frac{1}{ | D_{ii} - z - \E_A m(z) | } \leq \frac{1}{\eta}$ to obtain
		\begin{align*}
			\E_A \frac{1}{n^2 \eta}  \sum_{i: D_{ii} \geq \sqrt{ (2 - 2\delta) \log n}} \frac{1}{ | D_{ii} - z - \E_A m(z) | } \Im G_{ii} \oindicator{\mathcal{E} \cap \mathcal{F}} \ll \frac{n^\eps}{n^2 \eta^3} = o \left( \frac{1}{n \eta} \right)  
		\end{align*}
		uniformly for $z \in \tilde S_\delta$.  
		
		Combining the terms above, we conclude that, with overwhelming probability, 
		\[ \frac{1}{n^2} \E_A \sum_{i,j} Q_{ii} G_{ij}^2 = o \left( \frac{1}{n \eta} \right) \]
		uniformly for $z \in \tilde S_\delta$.  In view of \eqref{eq:s:res_mid}, it remains to show
		\begin{equation} \label{eq:s:res_end}
			\frac{1}{n} \E_A m_n(z) \tr (QG) - \E_A m_n(z) \E_A \frac{1}{n} \tr (GQ) = o \left( \frac{1}{n \eta} \right)
		\end{equation} 
		with overwhelming probability, uniformly for $z \in \tilde S_{\delta}$.  
		We will use Theorem \ref{thm:concentration} to establish \eqref{eq:s:res_end}.  Indeed, by using naive bounds, we can insert the indicator function of the event $\mathcal{E} \cap \mathcal{F}$, and it suffices to bound
		\[ \sup_{z \in \tilde{S}_\delta} \E_A \left| \frac{1}{n} m_n(z) \tr (QG) - \E_A \left[ m_n(z)\right] \frac{1}{n} \tr (GQ) \right|\oindicator{\mathcal{E} \cap \mathcal{F}}. \]
		In order to bound this term, we will need the following result.
		
		\begin{lemma} \label{lem:s:GQbnd}
			One has
			\[ \sup_{z \in \tilde S_{\delta}} \left| \frac{1}{n} \tr (GQ) \right|\oindicator{\mathcal{E} \cap \mathcal{F}} = O\left(\frac{1}{\log n} \right) \]
			with probability $1$. 
		\end{lemma}
		
		We prove Lemma \ref{lem:s:GQbnd} below, but let us first complete the proof of Theorem \ref{thm:stieltjes_transforms}.  Indeed, applying Lemma \ref{lem:s:GQbnd}, we obtain
		\begin{align*}
			\E_A &\left[ \left| \frac{1}{n} m_n(z) \tr (QG) - \E_A \left[ m_n(z)\right] \frac{1}{n} \tr (GQ) \right|\oindicator{\mathcal{E} \cap \mathcal{F}} \right] \\
			&\qquad\leq \E_A \left[ \left| m_n(z) - \E_A \left[ m_n(z)\right] \right| \left| \frac{1}{n} \tr (GQ) \right|\oindicator{\mathcal{E} \cap \mathcal{F}} \right] \\
			&\qquad\ll \frac{1}{\log n} \E_A \left[ \left| m_n(z) - \E_A \left[ m_n(z)\right] \right| \oindicator{\mathcal{E} \cap \mathcal{F}} \right] 
		\end{align*}
		uniformly for $z \in \tilde S_{\delta}$.  Applying Theorem \ref{thm:concentration} (which we included in the event $\mathcal{E}$ for just this purpose) establishes \eqref{eq:s:res_end}.  In view of \eqref{eq:s:res_mid}, this completes the proof of Theorem \ref{thm:stieltjes_transforms}.  
	\end{proof}
	
	We now give the proof of Lemma \ref{lem:s:GQbnd}.
	
	\begin{proof}[Proof of Lemma \ref{lem:s:GQbnd}]
		By the Cauchy--Schwarz inequality, it suffices to show
		\begin{equation} \label{eq:s:GGast}
			\sup_{z \in \tilde S_{\delta}} \frac{1}{n} \tr (G G^\ast)\oindicator{\mathcal{E} \cap \mathcal{F}} = O\left(\frac{1}{\log n} \right) 
		\end{equation}
		and
		\begin{equation} \label{eq:s:QQast}
			\sup_{z \in \tilde S_{\delta}} \frac{1}{n} \tr (Q Q^\ast)\oindicator{\mathcal{E} \cap \mathcal{F}} = O\left(\frac{1}{\log n} \right) 
		\end{equation} 
		with probability $1$.  Fix $z = E + i \eta \in \tilde S_{\delta}$; all of our bounds will be uniform in $z$ and hold with probability $1$.  
		
		For the term on the left-hand side of \eqref{eq:s:GGast}, we apply the spectral theorem to obtain
		\begin{align*}
			\frac{1}{n} \tr (G G^\ast)\oindicator{\mathcal{E} \cap \mathcal{F}} &= \frac{1}{n} \sum_{j} \frac{1}{ (\lambda_j(L) - E)^2 + \eta^2 }\oindicator{\mathcal{E} \cap \mathcal{F}} \\
			&= \frac{1}{n} \sum_{j: \lambda_j(L) \geq \sqrt{(2-2\delta) \log n}} \frac{1}{ (\lambda_j(L) - E)^2 + \eta^2 }\oindicator{\mathcal{E} \cap \mathcal{F}} \\
			&\qquad\qquad+ \frac{1}{n} \sum_{j: \lambda_j(L) < \sqrt{(2-2\delta) \log n}} \frac{1}{ (\lambda_j(L) - E)^2 + \eta^2 }\oindicator{\mathcal{E} \cap \mathcal{F}}. 
		\end{align*}
		Observe that, on the event $\mathcal{E}$, the first sum only contains $O(n^\eps)$ terms, and so using a naive bound we obtain 
		\[ \frac{1}{n} \sum_{j: \lambda_j(L) \geq \sqrt{(2-2\delta) \log n}} \frac{1}{ (\lambda_j(L) - E)^2 + \eta^2 }\oindicator{\mathcal{E} \cap \mathcal{F}} \ll \frac{n^{\eps}}{n \eta^2} = O\left(\frac{1}{\log n} \right) \]
		with probability $1$.  For the second sum, we have
		\[  \frac{1}{ (\lambda_j(L) - E)^2 + \eta^2 } \gg \frac{1}{\log n} \]
		uniformly for $z \in \tilde S_{\delta}$ since $\lambda_j(L) < \sqrt{(2 - 2\delta) \log n}$ and $E \geq \sqrt{(2 - \delta) \log n}$.  Thus, we find
		\[ \frac{1}{n} \sum_{j: \lambda_j(L) < \sqrt{(2-2\delta) \log n}} \frac{1}{ (\lambda_j(L) - E)^2 + \eta^2 } \ll \frac{n}{n \log n} = O\left(\frac{1}{\log n} \right) \]
		with probability $1$, where we bounded the total number of terms in this sum by $n$.  This completes the proof of \eqref{eq:s:GGast}.  
		
		The proof of \eqref{eq:s:QQast} is similar.  Since $Q$ is a diagonal matrix, we have
		\begin{align*}
			\frac{1}{n} \tr (Q Q^\ast)\oindicator{\mathcal{E} \cap \mathcal{F}} &= \frac{1}{n} \sum_{j} \frac{ 1} { |D_{jj} - z - \E_A m_n(z)|^2} \oindicator{\mathcal{E} \cap \mathcal{F}} \\
			&= \frac{1}{n} \sum_{j: D_{jj} \geq \sqrt{(2 - 2\delta) \log n}} \frac{ 1} { |D_{jj} - z - \E_A m_n(z)|^2} \oindicator{\mathcal{E} \cap \mathcal{F}} \\
			&\qquad \qquad + \frac{1}{n} \sum_{j: D_{jj} < \sqrt{(2 - 2\delta) \log n}} \frac{ 1} { |D_{jj} - z - \E_A m_n(z)|^2} \oindicator{\mathcal{E} \cap \mathcal{F}}.
		\end{align*}
		On the event $\mathcal{E}$, the first sum contains $O(n^{\eps})$ terms, and so using a naive bound we obtain
		\[ \frac{1}{n} \sum_{j: D_{jj} \geq \sqrt{(2 - 2\delta) \log n}} \frac{ 1} { |D_{jj} - z - \E_A m_n(z)|^2} \oindicator{\mathcal{E} \cap \mathcal{F}} \leq \frac{n^{\eps}}{n \eta^2} = O\left(\frac{1}{\log n} \right) \]
		with probability $1$.  We now turn our attention to the second sum.  On the event $\mathcal{F}$, $\E_A m_n(z)$ is uniformly bounded for $z \in \tilde S_{\delta}$.  Thus, on the event $\mathcal{E} \cap \mathcal{F}$, whenever $D_{jj} < \sqrt{(2 - 2 \delta) \log n}$, we have 
		\[ \frac{ 1} { |D_{jj} - z - \E_A m_n(z)|^2} \gg \frac{1}{\log n} \]
		uniformly for $z \in \tilde S_{\delta}$ since $E \geq \sqrt{(2-\delta) \log n}$.  Therefore, we conclude that 
		\[ \frac{1}{n} \sum_{j: D_{jj} < \sqrt{(2 - 2\delta) \log n}} \frac{ 1} { |D_{jj} - z - \E_A m_n(z)|^2} \oindicator{\mathcal{E} \cap \mathcal{F}} \ll \frac{n}{n \log n } = O\left(\frac{1}{\log n} \right) \]
		with probability $1$, and the proof is complete.  
	\end{proof}
	
	We include the following extensions of Theorem \ref{thm:stieltjes_transforms}, which we will need in the next section.  
	
	\begin{theorem} \label{thm:s:expected_stieltjes_transforms}
		There exists $\delta > 0$ so that 
		\begin{equation} \label{eq:s:explimit}
			\sup_{z \in \tilde S_{\delta} \cup \hat{S}_\delta} \sqrt{n} \eta \left| \E m_n(z) - m(z) \right| = o (1).
		\end{equation} 
	\end{theorem}
	\begin{remark}
		It is likely that the error term in \eqref{eq:s:explimit} can be improved.  However, we will not need a sharp bound to prove our main results.  In addition, the proof reveals that \eqref{eq:s:explimit} can be extended to hold for all $z \in S_{\delta}$, but we will not need a larger spectral domain in this work.  
	\end{remark}
	\begin{proof}[Proof of Theorem \ref{thm:s:expected_stieltjes_transforms}]
		The proof is similar to the proof of Theorem \ref{thm:stieltjes_transforms}.  Again, for notational convenience, we only prove this for $\tilde{S}_\delta$.  We outline the main ideas of the proof here.  
		
		Fix $\eps \in (0, 1/100)$.  Since $A$ is a GOE matrix, it follows from standard norm bounds (see, for example, \cite{MR1863696,MR2963170} and references therein) that $\|A\| \leq 3$ with overwhelming probability.  
		Thus, by Proposition \ref{prop:count} and Weyl's perturbation theorem (see, for instance, Corollary III.2.6 in \cite{MR1477662}), there exists $\delta > 0$ so that the event given in \eqref{eq:s:Leig} 
		holds with overwhelming probability.  Moreover, by taking $\delta$ smaller if needed, we can also ensure the same value of $\delta$ applies to the bounds in Proposition \ref{prop:count}.  For convenience, we define the event
		\begin{align*}
			\Omega := &\left\{ | \{ 1 \leq i \leq n : \lambda_i(L) \geq \sqrt{(2 - 2\delta) \log n} \}| = O(n^{\eps}) \right\} \\
			&\qquad \bigcap \left\{ \sup_{z = E + i \eta \in \tilde S_{\delta}} \sqrt{n} \eta (\log n)^{3/4} \left| m_n(z) - \E m_n(z) \right| = o(1) \right\} \\
			&\qquad \bigcap \left\{\left| \{ 1 \leq i \leq n : D_{ii} \geq \sqrt{(2- 2\delta) \log n} \} \right| = O_{\eps}(n^{\eps}) \right\}
		\end{align*}
		to be the intersection of the events from \eqref{eq:s:Leig},  Proposition \ref{prop:count}, and Proposition \ref{prop:concentration}.  Here, we applied Proposition \ref{prop:concentration} with $t = (\log n)^{3/4}$ and then used continuity and a net argument to extend the bound to all $z \in \tilde S_{\delta}$.  It follows that $\Omega$ holds with overwhelming probability, and we will work on this event throughout the proof.  
		
		For $z \in S_{\delta}$, we define 
		\[ G := G(z) \]
		and 
		\[ Q := Q(z + \E m_n(z)), \]
		where $G(z) := (L-z)^{-1}$ is again the resolvent of $L$ and $Q(z) := (D - z)^{-1}$ is the resolvent of $D$.  
		
		Similar to the proof of Theorem \ref{thm:stieltjes_transforms}, we can use the resolvent identity \eqref{eq:resolvent} to write 
		\begin{align*} 
			\E m_n(z) - \E s_n(z + \E m_n(z)) &= \E \frac{1}{n} \tr G - \E \frac{1}{n} \tr Q \\
			&=  \E \frac{1}{n} \tr (GAQ) - \E m_n(z) \E \frac{1}{n} \tr (GQ). 
		\end{align*}
		Applying the same Gaussian integration by parts argument from the proof of Theorem \ref{thm:stieltjes_transforms} and noting that
		\[ \E s_n(z + \E m_n(z)) = s(z + \E m_n(z)), \]
		we obtain
		\begin{align} \label{eq:s:res_mid2}
			\E m_n(z) - s(z + \E m_n(z)) &=  \frac{1}{n^2} \E \sum_{i,j} Q_{ii} G_{ij}^2 + \frac{1}{n} \E m_n(z) \tr (QG) \\
			&\qquad\qquad - \E m_n(z) \E \frac{1}{n} \tr (GQ). \nonumber
		\end{align}
		Our goal is to bound the error terms on the right-hand side.  The first term 
		\[ \frac{1}{n^2} \E \sum_{i,j} Q_{ii} G_{ij}^2 \]
		is handled using the Ward identity, exactly in the same way as it was bounded in the proof of Theorem \ref{thm:stieltjes_transforms}, and we omit the details.   
		
		We now turn to bounding the error term
		\begin{align*}
			\left| \frac{1}{n} \E m_n(z) \tr (QG) - \E m_n(z) \E \frac{1}{n} \tr (GQ) \right| &\leq \E \left[ \left| m_n(z) - \E \left[ m_n(z) \right] \right| \left| \frac{1}{n} \tr (QG) \right| \right]. 
		\end{align*}
		Using naive bounds, we can insert the indicator function of $\Omega$, and it suffices to bound 
		\[ \E \left[ \left| m_n(z) - \E \left[ m_n(z) \right] \right| \left| \frac{1}{n} \tr (QG) \right| \oindicator{\Omega} \right]. \]
		We will need the following analogue of Lemma \ref{lem:s:GQbnd}.  
		
		\begin{lemma} \label{lem:s:GQbnd2}
			One has
			\[ \sup_{z \in \tilde S_{\delta} \cup \hat{S}_\delta} \left| \frac{1}{n} \tr (GQ) \right|\oindicator{\Omega} = O\left(\frac{1}{\log n} \right) \]
			with probability $1$. 
		\end{lemma}
		We provide the proof of Lemma \ref{lem:s:GQbnd2} below, but first we complete the proof of Theorem \ref{thm:s:expected_stieltjes_transforms}.  Indeed, we find 
		\begin{align*}
			\E \left[ \left| m_n(z) - \E \left[ m_n(z) \right] \right| \left| \frac{1}{n} \tr (QG) \right|\oindicator{\Omega}  \right] &\ll \frac{1}{\log n}  \E \left[ \left| m_n(z) - \E \left[ m_n(z) \right] \right| \oindicator{\Omega} \right] \\
			&\ll  \frac{1}{\sqrt{n} (\log n)^{1/4} \eta} \\
			&= o \left( \frac{1}{\sqrt{n} \eta} \right)
		\end{align*}
		uniformly for $z \in \tilde S_{\delta}$.  Here, we used Lemma \ref{lem:s:GQbnd2} in the first bound and 
		\[ \sup_{z = E + i \eta \in \tilde S_{\delta}} \sqrt{n} \eta (\log n)^{3/4} \left| m_n(z) - \E m_n(z) \right| = o(1), \]
		which is part of the event $\Omega$, in the second bound.  
		
		Combining the bounds above with \eqref{eq:s:res_mid2}, we conclude that
		\[ \sup_{z \in \tilde S_{\delta} } \sqrt{n} \eta \left| \E m_n(z) - s(z + \E m_n(z)) \right| = o(1). \]
		An application of Theorem \ref{thm:stability} now completes the proof.  
	\end{proof}
	
	We now outline the proof of Lemma \ref{lem:s:GQbnd2}. 
	\begin{proof}[Proof of Lemma \ref{lem:s:GQbnd2}]
		The proof of Lemma \ref{lem:s:GQbnd2} follows the proof of Lemma \ref{lem:s:GQbnd} nearly exactly.  Only the following changes need to be made:
		\begin{itemize}
			\item One needs to replace the indicator function $\oindicator{\mathcal{E} \cap \mathcal{F}}$ with $\oindicator{\Omega}$ and any use of the event $\mathcal{E}$ by $\Omega$. 
			\item Occurrences of $\E_A m_n(z)$ need to be replaced by $\E m_n(z)$.  
			\item One does not need the event $\mathcal{F}$ to control $\E m_n(z)$ (in fact, $\E m_n(z)$ is deterministic).  Instead, one can use that $\E m_n(z)$ is bounded uniformly for $z \in \tilde S_{\delta}$ by Lemma \ref{lemma:s:mn}.  
		\end{itemize}
	\end{proof}
	
	To conclude this section, we present the following concentration bound, which we will also need in the next section.  
	
	\begin{lemma} \label{lem:expected_stieltjes_transforms2} 
		For any fixed $\delta > 0$, asymptotically almost surely, one has 
		\begin{equation} \label{eq:s:conc}
			\sup_{z \in \tilde S_{\delta} \cup \hat{S}_\delta} \sqrt{n} \eta \left| \E m_n(z) - \E_A m_n(z) \right| = o (1). 
		\end{equation} 
	\end{lemma}
	\begin{proof}
		We will establish \eqref{eq:s:conc} by applying the Gaussian Poincar\'{e} inequality.  Since $\E_A m_n(z)$ only depends on the randomness from $D$, we will apply the tensorization property of the Poincar\'{e} inequality to the $n$ iid standard normal entries of $D$.  We refer the reader to Section 4.4 of \cite{MR2760897} for further details concerning the Poincar\'{e} inequality and its uses in random matrix theory.  
		
		For any $z := E + i \eta \in \tilde S_{\delta} \cup \hat{S}_\delta$, we begin with
		\begin{align*}
			\E \left| \E_A m_n(z) - \E m_n(z) \right|^2 &\leq \frac{1}{n^2} \E \sum_i \left| \sum_j \E_A G_{ji} G_{ij} \right|^2 \\
			&\leq \frac{1}{n^2} \E \sum_i \left| \sum_j G_{ji} G_{ij} \right|^2, 
		\end{align*}
		where the first inequality follows from the Gaussian Poincar\'{e} inequality and the second from Jensen's inequality.  Here, we also used the fact that
		\[ \frac{ \partial \E_A G_{jj}} {\partial D_{ii}} = - \E_A \left[ G_{ji} G_{ij} \right], \]
		which can be deduced from the resolvent identity \eqref{eq:resolvent}.  Applying the triangle inequality and the Ward identity \eqref{eq:ward}, we obtain
		\begin{align*}
			\E \left| \E_A m_n(z) - \E m_n(z) \right|^2 &\leq \frac{1}{n^2} \E \sum_i \left( \sum_j |G_{ij}|^2 \right)^2 \\
			&\leq \frac{1}{n^2 \eta^2} \E \sum_i \left( \Im G_{ii} \right)^2 \\
			&\leq \frac{1}{n^2 \eta^3} \E \Im m_n(z) \\
			&\ll \frac{1}{n^2 \eta^3 \sqrt{\log n}},
		\end{align*}
		where we used Lemma \ref{lemma:s:mn} in the last step.  Thus, by Markov's inequality, we conclude that
		\begin{equation} \label{eq:s:markov}
			\sup_{z \in \tilde S_{\delta} \cup \hat{S}_\delta} \Prob \left( \left| \E_A m_n(z) - \E m_n(z) \right| \geq \frac{1}{\sqrt{n} \eta (\log n)^{1/4}} \right) \ll \frac{1}{n \eta} = \frac{1}{n^{3/4}}. 
		\end{equation}
		
		We now extend the bound in \eqref{eq:s:markov} to hold simultaneously for all $z \in \tilde S_{\delta}$ by a net argument.  We note that all of our bounds below will hold uniformly for $z \in \tilde S_{\delta} \cup \hat{S}_\delta$.  Let $\mathcal{N}$ be a $\frac{1}{\sqrt{n} (\log n)^{1/4}}$-net of $\tilde S_{\delta} \cup \hat{S}_\delta$.  A counting argument shows that $\mathcal{N}$ can be chosen to have cardinality 
		\[ |\mathcal{N}| = O_{\delta}(\sqrt{n} \log n ). \]
		Therefore, \eqref{eq:s:markov} together with the union bound implies that the event
		\[ \left\{ \sup_{z \in \mathcal{N}} \left| \E_A m_n(z) - \E m_n(z) \right| \leq \frac{1}{\sqrt{n} \eta (\log n)^{1/4}} \right\} \]
		holds with probability $1 - o(1)$.  Let $\mathcal{F}'$ be the intersection of the event above with the event $\mathcal{F}$, defined in \eqref{eq:s:F}.  It follows that $\mathcal{F}'$ holds with probability $1 - o(1)$, and we will prove that \eqref{eq:s:conc} holds on $\mathcal{F}'$.  Indeed, let $z \in \tilde S_{\delta} \cup \hat{S}_\delta$ be arbitrary.  Then there exists $z' \in \mathcal{N}$ so that
		\begin{equation} \label{eq:s:net}
			|z - z'| \leq \frac{1}{\sqrt{n} (\log n)^{1/4}}. 
		\end{equation} 
		Thus, on the event $\mathcal{F}'$, we have
		\[ \left| \E_A m_n(z') - \E m_n(z') \right| \ll \frac{1}{\sqrt{n} \eta (\log n)^{1/4}}. \]
		Moreover, by the resolvent identity \eqref{eq:resolvent}, we find
		\begin{align*}
			\left| \E m_n(z) - \E m_n(z') \right| &= \frac{1}{n} |z - z'| \left| \E \tr (G(z) G(z')) \right| \\
			&\leq |z - z'| \frac{1}{n} \E \sum_{i,j} \left[|G_{ij}(z)|^2 + |G_{ij}(z')|^2 \right] \\
			&= \frac{ |z - z'|}{\eta} \left[ \E \Im m_n(z) + \E \Im m_n(z') \right] \\
			&\ll \frac{1}{\sqrt{n} \eta (\log n)^{3/4}},
		\end{align*}
		where we used the Ward identity \eqref{eq:ward}, Lemma \ref{lemma:s:mn}, and \eqref{eq:s:net}.  Similarly, we have
		\begin{align*}
			\left| \E_A m_n(z) - \E_A m_n(z') \right| &\leq |z - z'| \frac{1}{n} \E_A \sum_{i,j} \left[|G_{ij}(z)|^2 + |G_{ij}(z')|^2 \right] \\
			&= \frac{ |z - z'|}{\eta} \left[ \E_A \Im m_n(z) + \E_A \Im m_n(z') \right] \\
			&\ll \frac{1}{\sqrt{n} \eta (\log n)^{3/4}},
		\end{align*}
		where the final inequality holds on the event $\mathcal{F} \supset \mathcal{F}'$.  Combining the bounds above, we conclude that
		\[ \left| \E_A m_n(z) - \E m_n(z) \right| = o \left( \frac{1}{\sqrt{n} \eta} \right). \]
		Since this bound holds uniformly for $z \in \tilde S_{\delta} \cup \hat{S}_\delta$, the proof is complete.  
	\end{proof}
	
	\section{Proof of Theorem \ref{thm:main_L}} \label{sec:proofofmain}

This section is devoted to the proof of Theorem \ref{thm:main_L}.  In view of Theorem \ref{thm:stieltjes_transforms} (noting that $\tilde{S}_{\delta_1}\subseteq\tilde{S}_{\delta_2}$ for any $0<\delta_1<\delta_2$), there exists $\tilde\delta > 0$ so that for any $0<\delta<\tilde\delta$ \begin{equation}\label{eq:A:Recursive Estimate}
	\sup_{z = E + i \eta \in \tilde S_{\delta} \cup \hat{S}_\delta} n \eta \left| m_n(z) - s_n(z + \E_A m_n(z)) \right| = o(1)
\end{equation}
with overwhelming probability. Throughout this section we will make repeated use of Proposition \ref{prop:count}, so we fix $0 < \delta < \tilde \delta$ such that 
\[ \left| \{ 1 \leq i \leq n : D_{ii} \geq \sqrt{(2- \delta) \log n} \} \right| = O(n^{1/100}) \]
with overwhelming probability. Considering imaginary parts, we conclude that, on the same event as \eqref{eq:A:Recursive Estimate},  \begin{align}\label{eq:A:eigenvalue counting function}
	\sup_{z = E + i \eta \in \tilde S_{\delta} \cup \hat{S}_\delta}\Bigg|&\sum_{j=1}^{n}\frac{\eta^2}{(\lambda_j(L)-E)^2+\eta^2} \nonumber\\  &\qquad -\sum_{j=1}^{n}\frac{\eta^2+\eta\Im\E_A m_n(z)}{(D_{jj}-E-\Re\E_Am_n(z))^2+(\eta+\Im\E_A m_n(z))^2} \Bigg|
\end{align} tends to $0$. 

If $E$ is chosen such that $(D_{jj}-E-\Re\E_Am_n(z))^2=0$ for some $j$, then the second sum in \eqref{eq:A:eigenvalue counting function} is bounded below by something asymptotically close to $1$. Hence, the first sum must also be at least $1$. One way in which the first sum can be close to $1$ is if $E$ is close to an eigenvalue of $L$ and the other terms are negligible. Bai and Silverstein (see \cite[Chapter 6]{MR2567175}) used this observation to show the spectra of certain random matrix models separate.  
%There is also the possibility that no eigenvalue of $L$ is very close to $E$, but that a large number are reasonably close and thus the first sum in \eqref{eq:A:eigenvalue counting function} is still close to $1$. Bai and Silverstein avoid this possibility by iterating \eqref{eq:A:eigenvalue counting function} to get a higher power of $\eta$ in the numerator of the terms for both sums. After iteration, the first sum can only be bounded away from $0$ if there are eigenvalues very close to $E$. 
In this section, we use a similar method (and an iteration argument) to precisely locate $\lambda_1(L)$. We then extend this argument to $\lambda_2(L),\dots,\lambda_k(L)$ and complete the proof of Theorem \ref{thm:main_L}. 
%The steps are technical, but the end result is a numerator of the form $\eta^4(1+o(1))$, where $o(1)$ is uniform in the index.

For the remainder of this section we fix $\eta=n^{-1/4}$. For any $z_1=E+i\eta \in \tilde S_\delta$, let $z_2=E+i\sqrt{2}\eta$ and consider the differences \begin{equation*}
	I_1=\sum_{j=1}^{n}\frac{\eta^2}{(\lambda_j(L)-E)^2+\eta^2}-\sum_{j=1}^{n}\frac{\eta^2}{(\lambda_j(L)-E)^2+2\eta^2},
\end{equation*} and \begin{align*}
	I_2&=\sum_{j=1}^{n}\frac{\eta^2+\eta\Im\E_A m_n(z_1)}{(D_{jj}-E-\Re\E_Am_n(z_1))^2+(\eta+\Im\E_A m_n(z_1))^2}\\
	&\quad-\sum_{j=1}^{n}\frac{\eta^2+\frac{1}{\sqrt{2}}\eta\Im\E_A m_n(z_2)}{(D_{jj}-E-\Re\E_Am_n(z_2))^2+(\sqrt{2}\eta+\Im\E_A m_n(z_2))^2}.
\end{align*} From \eqref{eq:A:Recursive Estimate} we have that \begin{equation}\label{eq:A:Difference of sums}
	\sup_{z_1 = E + i \eta \in \tilde S_{\delta}}|I_1-I_2|=o(1),
\end{equation} with overwhelming probability. Note that \begin{equation*}
	I_1=\sum_{j=1}^{n}\frac{\eta^4}{((\lambda_j(L)-E)^2+\eta^2)((\lambda_j(L)-E)^2+2\eta^2)},
\end{equation*} and \begin{equation*}
	I_2=\sum_{j=1}^{n}\frac{N_{j,n}}{B_{j,n}},
\end{equation*} where\begin{align*}
	B_{j,n}&=((D_{jj}-E-\Re\E_Am_n(z_1))^2+(\eta+\Im\E_A m_n(z_1))^2) \\
	&\qquad \times ((D_{jj}-E-\Re\E_Am_n(z_2))^2+(\sqrt{2}\eta+\Im\E_A m_n(z_2))^2),
\end{align*} and \begin{align*}
	N_{j,n}&=\eta^2\left[\left(\sqrt{2}\eta+\Im\E_A m_n(z_2) \right)^2-\left(\eta+\Im\E_A m_n(z_1) \right)^2 \right]\\
	&\quad+\eta\Im\E_A m_n(z_1)\left(\sqrt{2}\eta+\Im\E_A m_n(z_2)\right)^2 \\
	&\quad -\frac{1}{\sqrt{2}}\eta\Im\E_A m_n(z_2)\left(\eta+\Im\E_A m_n(z_1)\right)^2\\
	&\quad+\eta^2\left[\left(D_{jj}-E-\Re\E_Am_n(z_2) \right)^2-\left(D_{jj}-E-\Re\E_Am_n(z_1) \right)^2 \right]\\
	&\quad+\eta\Im\E_A m_n(z_1)\left(D_{jj}-E-\Re\E_Am_n(z_2)\right)^2 \\
	&\quad-\frac{1}{\sqrt{2}}\eta\Im\E_A m_n(z_2)\left(D_{jj}-E-\Re\E_Am_n(z_1)\right)^2. 
\end{align*} 

In several steps we now show $N_{j,n}=\eta^4(1+o(1))$ uniformly on $\tilde{S}_\delta$. The implied constants in our asymptotic notation in this section are uniform over $j\in[n]$. Applying Lemma \ref{lem:Imm_fc is small} with Theorem \ref{thm:s:expected_stieltjes_transforms} and Lemma \ref{lem:expected_stieltjes_transforms2}, we deduce that \begin{equation*}
	\sup_{z = E + i \eta \in \tilde{S}_{\delta} \cup \hat{S}_{\delta}}	\Im\E_Am(z)=o(\eta)
\end{equation*} with overwhelming probability. From this we see that \begin{equation}\label{eq:A:first N term small}
	\sup_{z_1 = E + i \eta \in \tilde{S}_{\delta}}\left|\eta^2\left[\left(\sqrt{2}\eta+\Im\E_A m_n(z_2) \right)^2-\left(\eta+\Im\E_A m_n(z_1) \right)^2 \right]\right|=\eta^4(1+o(1)),
\end{equation}
and %\todo{Some of these equations could use some better formatting.  In particular, the second part of the equation (the part that starts with the minus sign) should be indented further to make it clear that it belongs with the line above.  Is it necessary to put this equation on so many lines?}
\begin{align}\label{eq:A:second N term small}
	\!	\eta\Im\E_A m_n(z_1)\left(\! \sqrt{2}\eta+\Im\E_A m_n(z_2)\!\right)^2\!\! -\frac{\eta\Im\E_A m_n(z_2)}{\sqrt{2}}\left( \eta+\Im\E_A m_n(z_1)\right)^2\! = o(\eta^4),
\end{align} uniformly for $z_1\in\tilde{S}_\delta$ with overwhelming probability. The following lemma allows us to control the terms of $N_{j,n}$ which involve $\Re \E_A m_n$.

\begin{lemma}\label{lemma:A:expected real parts are close}
	There exists $\delta>0$ such that \begin{equation*}
		\sup_{z_1 = E + i \eta \in \tilde S_{\delta}}\left|	\Re\E_A m_n(z_1)-\Re\E_A m_n(z_2)\right|=o \left(n^{-1/2} \right),
	\end{equation*} with overwhelming probability. 
\end{lemma} Lemma \ref{lemma:A:expected real parts are close} is very similar to Lemma \ref{Alemma:Difference of real FC}, however neither follows from the other as we do not have a result comparing $\E_A m_n(z)$ to $m(z)$ within the appropriate error $o(n^{-1/2})$. The proof of Lemma \ref{lemma:A:expected real parts are close} follows in a nearly identical way to the proof of Lemma \ref{Alemma:Difference of real FC}, with summation replacing integration.
\begin{proof}[Proof of Lemma \ref{lemma:A:expected real parts are close}]
	Fix $E+in^{-1/4}\in \tilde{S}_\delta$. We begin by considering  
	\begin{equation}\label{eq:Difference of transforms}
		\Re m_n(z_1)-\Re m_n(z_2)=\frac{1}{n}\sum_{j=1}^n\frac{\eta^2(\lambda_j-E)}{\left((\lambda_j-E)^2+\eta^2 \right)\left((\lambda_j-E)^2+2\eta^2 \right)},
	\end{equation} 
	where we let $\lambda_n \leq \cdots \leq \lambda_1$ denote the eigenvalues of $L$.  
	We divide $\{1,\dots, n\}$ into three subsets $J_1=\{j:|\lambda_j-E|\geq \sqrt{E} \}$, $J_2=\{j:n^{-1/8}E \leq |\lambda_j-E|< \sqrt{E} \}$, and $J_3=\{j:|\lambda_j-E|< n^{-1/8}E \}$. Note for any $E\in[\sqrt{(2-\delta)\log n},\sqrt{3\log n}]$, $J_2,J_3\subseteq\{j:\lambda_j\geq \sqrt{(2-\delta)\log n}-\left[(2-\delta)\log n\right]^{1/4} \}$. From standard bounds on $\|A\|$ (see, for example, \cite{MR1863696,MR2963170} and references therein) and taking $\delta$ sufficiently small in Proposition \ref{prop:count}, we have \begin{equation}\label{eq:A:J sizes}
		|J_2|,|J_3|\leq\left|\{j:\lambda_j\geq \sqrt{(2-\delta)\log n}-\left[(2-\delta)\log n\right]^{1/4} \} \right|=O(n^{1/100})
	\end{equation}  with overwhelming probability. Note that\begin{equation}\label{eq:A:J_1 estimates}
		\left|\frac{1}{n}\sum_{j\in J_1}\frac{\eta^2(\lambda_j-E)}{\left((\lambda_j-E)^2+\eta^2 \right)\left((\lambda_j-E)^2+2\eta^2 \right)}\right|\leq \frac{\eta^2}{(2-\delta)^{3/4}\log^{3/4}n},
	\end{equation} and so we focus on the sets where $|\lambda_j-E|$ is small. For $J_2$ applying \eqref{eq:A:J sizes} \begin{align}\label{eq:A:J_2 estimates}
		\left|\frac{1}{n}\sum_{j\in J_2}\frac{\eta^2(\lambda_j-E)}{\left((\lambda_j-E)^2+\eta^2 \right)\left((\lambda_j-E)^2+2\eta^2 \right)}\right|&\leq\frac{1}{n}\sum_{j\in J_2}\frac{\eta^2}{|\lambda_j-E|^3}\nonumber\\
		& =\eta^2O\left(\frac{n^{1/100}}{n\eta^2E^3} \right)\\
		&=o(\eta^2)\nonumber
	\end{align} with overwhelming probability and implicit constants depending only on $\delta$. In a similar fashion applying \eqref{eq:A:J sizes}, \begin{align}\label{eq:A:J_3 estimates}
		\left|\frac{1}{n}\sum_{j\in J_3}\frac{\eta^2(\lambda_j-E)}{\left((\lambda_j-E)^2+\eta^2 \right)\left((\lambda_j-E)^2+2\eta^2 \right)}\right|&\leq\frac{1}{n}\sum_{j\in J_3}\frac{\eta^2|\lambda_j-E|}{\eta^4}\nonumber \\
		&\leq O\left(\eta^{5/2}\sqrt{3\log n}n^{1/100}\right)\\
		&= o(\eta^2)\nonumber
	\end{align} with overwhelming probability and implicit constants depending only on $\delta$. Combining \eqref{eq:Difference of transforms}, \eqref{eq:A:J_1 estimates}, \eqref{eq:A:J_2 estimates}, and \eqref{eq:A:J_3 estimates} we can conclude that\begin{equation*}
		\Re m_n(z_1)-\Re m_n(z_2)=o(\eta^2)
	\end{equation*} with overwhelming probability. Applying Theorem \ref{thm:concentration} we see that
	\begin{align*}
		\left|\Re\E_A m_n(z_1)-\Re\E_A m_n(z_2)\right|&\leq\left|\Re\E_A m_n(z_1)-\Re m_n(z_1)\right|\\
		&\quad+\left|\Re\E_A m_n(z_2)-\Re m_n(z_2)\right|\\
		&\quad  +\left|\Re m_n(z_1)-\Re m_n(z_2)\right|\\
		&=o(\eta^3)+o(\eta^2)
	\end{align*} with overwhelming probability and implicit constants depending only on $\delta$. To extend to the supremum over $\tilde{S}_\delta$ note that $\E_A m_n$ (or any Stieltjes transform) is Lipschitz on $\tilde{S}_\delta$ with Lipschitz constant at most $\sqrt{n}$. Taking a $\frac{1}{n^4}$-net and using the union bound completes the proof.
\end{proof}

Thus we have from  Lemma \ref{lemma:Bounds on m}, Theorem \ref{thm:s:expected_stieltjes_transforms}, Lemma \ref{lem:expected_stieltjes_transforms2}, and Lemma \ref{lemma:A:expected real parts are close} that \begin{align}\label{eq:A:third N piece is small}
	&\eta^2\left[\left(D_{jj}-E-\Re\E_Am_n(z_2) \right)^2-\left(D_{jj}-E-\Re\E_Am_n(z_1) \right)^2 \right]\nonumber \\
	&=\eta^2\left[\left(2D_{jj}-2E-\Re\E_Am_n(z_2)-\Re\E_Am_n(z_1) \right)\left(\Re\E_Am_n(z_1)-\Re\E_Am_n(z_2) \right) \right]\nonumber\\
	&=o(\eta^4|D_{jj}-E-\Re\E_Am_n(z_1)|),
\end{align} with overwhelming probability.

To deal with the last term in $N_{j,n}$ note some algebra leads to\begin{equation}\label{eq:A:recursive for sums}
	\eta\Im m_n(z_1)-\frac{1}{\sqrt{2}}\eta\Im m_n(z_2)=\frac{1}{n}I_1,
\end{equation} and applying Theorem \ref{thm:concentration} to \eqref{eq:A:recursive for sums} gives	\begin{align}\label{eq:A:fourth N piece is small}
	&\eta\Im\E_A m_n(z_1)\left(D_{jj}-E-\Re\E_Am_n(z_2)\right)^2 \\
	&-\frac{1}{\sqrt{2}}\eta\Im\E_A m_n(z_2)\left(D_{jj}-E-\Re\E_Am_n(z_1)\right)^2\\
	&=\frac{\left(D_{jj}-E-\Re\E_Am_n(z_1) \right)^2}{n}I_1+o(\eta^4|D_{jj}-E-\Re\E_Am_n(z_1)|)
\end{align} uniformly for $E\in[\sqrt{(2-\delta)\log n},\sqrt{3\log n}]$ and $j\in[n]$, with overwhelming probability.  Thus we conclude from \eqref{eq:A:first N term small}, \eqref{eq:A:second N term small}, \eqref{eq:A:third N piece is small}, and \eqref{eq:A:fourth N piece is small} that \begin{equation}\label{eq:A:simplified I_2}
	I_2=\sum_{j=1}^{n}\frac{\eta^4(1+o(1))+n^{-1}\left(D_{jj}-E-\Re\E_Am_n(z_1) \right)^2I_1+o(\eta^4|D_{jj}-E-\Re\E_Am_n(z_1)|)}{B_{j,n}},
\end{equation} with overwhelming probability. A straightforward application of Proposition \ref{prop:count} with $\eps=1/100$ (similar to \eqref{eq:A:J_2 estimates} and \eqref{eq:A:J_3 estimates} in the proof of Lemma \ref{lemma:A:expected real parts are close})  yields\begin{equation}\label{eq:A:controlling I_2 piece}
	\sup_{z_1 = E + i \eta \in \tilde{S}_{\delta}}\left|\sum_{j=1}^{n}\frac{\eta^4\left(D_{jj}-E-\Re\E_Am_n(z_1) \right)^2}{B_{j,n}}\right|=o(1)
\end{equation} and \begin{equation}\label{eq:A:I_2 reduction}
	\sup_{z_1 = E + i \eta \in \tilde{S}_{\delta}}\left|\sum_{j=1}^{n}\frac{o(\eta^4|D_{jj}-E-\Re\E_Am_n(z_1)|)}{B_{j,n}}\right|=o(1),
\end{equation} with overwhelming probability. Combining \eqref{eq:A:Difference of sums}, \eqref{eq:A:simplified I_2}, \eqref{eq:A:controlling I_2 piece}, and \eqref{eq:A:I_2 reduction}  we arrive at the conclusion \begin{equation}\label{eq:A:simple description of differences}
	\sup_{z_1 = E + i \eta \in \tilde{S}_{\delta}}\left|(1+o(1))I_1- (1+o(1))\sum_{j=1}^{n}\frac{\eta^4}{B_{j,n}}\right|=o(1),
\end{equation} with overwhelming probability, where we have pulled out $(1+o(1))$ terms which are uniform in $j$. At this point, we order the diagonal entries of $D$ as $D_{(1)}> D_{(2)}>\dots> D_{(n)}$. With the technical estimates complete we will now locate $\lambda_1(L)$ with respect to $D_{(1)}$. Let $E_{(1)}$ be a solution to the equation \begin{equation*}
	D_{(1)}-E-\Re m(E+i\eta)=0,
\end{equation*} whose existence is guaranteed by Lemma \ref{lemma:A:Existence of E}. From Lemma \ref{lemma:A:Existence of E} we see that asymptotically almost surely $E_{(1)}=D_{(1)}+\frac{1}{D_{(1)}}+O\left(\frac{1}{D_{(1)}^2}\right)$, and after noting $D_{(1)}=a_n+o(1)$ asymptotically almost surely \begin{equation*}
	\sup_{E\in\left[E_{(1)}-\frac{1}{a_n^{3/2}},E_{(1)}+\frac{1}{a_n^{3/2}} \right] } (1+o(1))\sum_{j=1}^{n}\frac{\eta^4}{B_{j,n}}\geq \frac{1}{2}.
\end{equation*} Thus, from \eqref{eq:A:simple description of differences} \begin{equation}\label{eq:A:I_1 is big at location.}
	\sup_{E\in\left[E_{(1)}-\frac{1}{a_n^{3/2}},E_{(1)}+\frac{1}{a_n^{3/2}} \right] }I_1\geq 1/2,
\end{equation} asymptotically almost surely. 

Assume, for the sake of contradiction, that there is not at least one eigenvalue of $L$ within $1/a_n^{3/2}$ of $E_{(1)}$. For any $\delta'>0$, $E_{(1)}\geq \sqrt{(2-\delta')\log n}$ asymptotically almost surely. By taking $\delta'$ sufficiently small and applying Proposition \ref{prop:count} there are at most $O(n^{1/100})$ eigenvalues of $L$ greater than $\sqrt{(2-2\delta')\log n}$. If $J=\{j: \lambda_j(L)\geq\sqrt{(2-2\delta')\log n} \}$, then asymptotically almost surely \begin{equation}\label{eq:A:eigenvalue too far}
	I_1\leq o(1)\sum_{j\notin J}\eta^4+\sum_{j\in J}\frac{\eta^4}{1/a_n^6}=o(1)+O\left(\frac{a_n^6}{n^{99/100}} \right)=o(1),
\end{equation} a contradiction of \eqref{eq:A:simple description of differences} and \eqref{eq:A:I_1 is big at location.}. Thus there must be at least one eigenvalue of $L$ within $1/a_n^{3/2}$ of $E_{(1)}$. Let $\delta''>0$ be sufficiently small such that from Proposition \ref{prop:count} $\tilde J=\{j: D_{jj}\geq\sqrt{(2-\delta'')\log n}\}$ has $O(n^{1/100})$ elements. Then,\begin{align}\label{eq:A:no big eignevalues first bound}
	\sup_{E\in\left[E_{(1)}+\frac{1}{a_n^{3/2}},\sqrt{3\log n} \right] }\left|\sum_{j=1}^{n}\frac{\eta^4}{B_{j,n}}\right|&\leq\sup_{E\in\left[E_{(1)}+\frac{1}{a_n^{3/2}},\sqrt{3\log n} \right] }\left|\sum_{j\in \tilde J}\frac{\eta^4}{B_{j,n}}\right| \nonumber\\
	&\qquad+\sup_{E\in\left[E_{(1)}+\frac{1}{a_n^{3/2}},\sqrt{3\log n} \right] }\left|\sum_{j\notin \tilde J}^{n}\frac{\eta^4}{B_{j,n}}\right|\nonumber\\
	&\leq \left|\sum_{j\in \tilde J}a_n^6\eta^4\right|+\left|\sum_{j\notin \tilde J}\frac{\eta^4}{(2-\delta'')^4\log^2 n}\right|\nonumber\\
	&=o(1).
\end{align} Note to bound the denominator for the sum over $\widetilde{J}$ in \eqref{eq:A:no big eignevalues first bound} we make use Lemma \ref{Alemma:FC der Asymptotic Expansion} to deal with the non-linearity of the Stieltjes transform. Since $\|L \| \leq \sqrt{3 \log n}$ asymptotically almost surely, if there exists an eigenvalue of $L$ beyond $\frac{1}{a_n^{3/2}}$ of $E_{(1)}$, it would then follow from the definition of $I_1$ that\begin{equation}
	\sup_{E\in\left[E_{(1)}+\frac{1}{a_n^{3/2}},\sqrt{3\log n} \right] } I_1\geq \frac{1}{2},
\end{equation} a contradiction of \eqref{eq:A:simple description of differences} and \eqref{eq:A:no big eignevalues first bound}. Thus there cannot be an eigenvalue of $L$ beyond $\frac{1}{a_n^{3/2}}$ of $E_{(1)}$ asymptotically almost surely, and hence $|\lambda_1(L)-E_{(1)}|=O\left(\frac{1}{a_n^{3/2}}\right)$ asymptotically almost surely. Using the asymptotic expansion in Lemma \ref{lemma:A:Existence of E} and the fact that $D_{(1)}=\sqrt{2\log n}+o(1)$ asymptotically almost surely we see that with probability tending to one \begin{equation*}
	E_{(1)}=D_{(1)}+\frac{1}{\sqrt{2\log n}}+O\left(\frac{1}{\log n}\right).
\end{equation*} Let $b_n'$ be defined as in \eqref{eq:def:bn'}, we conclude the proof of Theorem \ref{thm:main_L} in the case that $k=1$ by noting with probability tending to one \begin{align*}
	a_n(\lambda_1(L)-b_n)&=a_n(E_{(1)}-b_n)+o(1)\\
	&=a_n(D_{(1)}-b_n')+o(1)
\end{align*} and hence for any $x\in\R$\begin{align*}
	\lim_{n\rightarrow\infty}\P\left(a_n(\lambda_1(L)-b_n)\leq x \right)&=\lim_{n\rightarrow\infty}\P\left(a_n(D_{(1)}-b_n')\leq x+o(1) \right)\\
	&=e^{-e^{-x}},
\end{align*} where the last equality is a well known result for maximum of iid Gaussian random variables (see, for instance,  \cite[Theorem 1.5.3]{MR691492}). %This completes the proof of Theorem \ref{thm:main_L}.

For any $k\in\N$, we now define $E_{(k)}$ to be a solution to
\[D_{(k)}-E-\Re m(E+i\eta)=0.\] Again, using the asymptotic expansion in Lemma \ref{lemma:A:Existence of E} we see that with probability tending to one that \begin{equation}\label{aeq:E control}
	E_{(k)}=D_{(k)}+\frac{1}{\sqrt{2\log n}}+O\left(\frac{1}{\log n}\right).
\end{equation}Above we showed that $|\lambda_1(L)-E_{(1)}|=O\left(\frac{1}{a_n^{3/2}}\right)$ asymptotically almost surely, however this comparison can be improved significantly to show $|\lambda_k(L)-E_{(k)}|=o(\eta)$ for any fixed $k\in\N$. To this end, fix $k\in\N$ and consider the first $k$ eigenvalues $\lambda_1(L)\geq\lambda_2(L)\geq\dots\geq\lambda_k(L)$ of $L$. Using the notation above we have that \begin{equation}\label{eq:A:Eigenvalue location}
	\sup_{E\in\left[E_{(1)}-\frac{1}{a_n^{3/2}},E_{(1)}+\frac{1}{a_n^{3/2}} \right] }	\left|(1+o(1))I_1- \sum_{j=1}^{n}\frac{\eta^4(1+o(1))}{B_{j,n}}\right|=o(1)
\end{equation} asymptotically almost surely. We also have with probability tending to one (see Proposition \ref{kprop:spacings}) that $E_{(2)}\notin \left[E_{(1)}-\frac{1}{a_n^{3/2}},E_{(1)}+\frac{1}{a_n^{3/2}} \right]$ and hence \begin{equation}\label{eq:A:estinamte on second sum}
	\sup_{E\in\left[E_{(1)}-\frac{1}{a_n^{3/2}},E_{(1)}+\frac{1}{a_n^{3/2}} \right] }	\left|\sum_{j=1}^{n}\frac{\eta^4(1+o(1))}{B_{j,n}}\right|=\frac{1}{2}+o(1).
\end{equation} Now let $c_n$ be such that $c_n\eta=E_{(1)}-\lambda_1(L)$. From the above we know $c_n\eta\in\left[-\frac{1}{a_n^{3/2}},\frac{1}{a_n^{3/2}} \right]$ asymptotically almost surely. When $E=\lambda_1(L)$, $I_1\geq 1/2$ and, using Lemma \ref{Alemma:FC der Asymptotic Expansion} to control the non-linearity of $m$, the second sum of \eqref{eq:A:Eigenvalue location} is $\frac{1}{(1+c_n^2)(2+c_n^2)}+o(1)$, a contradiction of \eqref{eq:A:estinamte on second sum} unless $c_n=o(1)$. It then follows $|\lambda_1(L)-E_{(1)}|=o(\eta)$.

 In fact, for any eigenvalue $\lambda_j(L)\in\left[E_{(1)}-\frac{1}{a_n^{3/2}},E_{(1)}+\frac{1}{a_n^{3/2}} \right]$, the preceding argument shows $|\lambda_j(L)-E_{(1)}|=o(\eta)$. If there are $N$ eigenvalues of $L$ in the interval $\left[E_{(1)}-\frac{1}{a_n^{3/2}},E_{(1)}+\frac{1}{a_n^{3/2}} \right]$ and hence within $o(\eta)$ of $E_{(1)}$, then \begin{equation*}
 		\sup_{E\in\left[E_{(1)}-\frac{1}{a_n^{3/2}},E_{(1)}+\frac{1}{a_n^{3/2}} \right] }	\left|I_1\right|\geq\frac{N}{2}+o(1).
 \end{equation*} We then see from \eqref{eq:A:Eigenvalue location} and  \eqref{eq:A:estinamte on second sum} that $N=1$. 

We now outline the argument for locating $\lambda_2(L)$, as it is extremely similar to that of $\lambda_1(L)$. An identical argument to the establishment of  \eqref{eq:A:I_1 is big at location.} and \eqref{eq:A:eigenvalue too far} concludes there must be at least one eigenvalue of $L$ within $\frac{1}{a_n^{3/2}}$ of $E_{(2)}$ asymptotically almost surely. After noting 
\[\sup_{E\in\left[E_{(2)}+\frac{1}{a_n^{3/2}},E_{(1)}-\frac{1}{a_n^{3/2}} \right] }	\left| \sum_{j=1}^{n}\frac{\eta^4(1+o(1))}{B_{j,n}}\right|=o(1),\] we see that $\lambda_2(L)$ must be within $\frac{1}{a_n^{3/2}}$ of $E_{(2)}$ asymptotically almost surely. It then follows in an identical manner to that of $\lambda_1(L)$ that $|\lambda_2(L)-E_{(2)}|=o(\eta)$ asymptotically almost surely. Union bounding we see \begin{equation}
	\P\left( \|(\lambda_1,\lambda_2)-(E_{(1)},E_{(2)})\|\geq\eta\right)=o(1).
\end{equation} Iterating for all $2<j\leq k$ allows us to conclude \begin{equation}\label{eq:A:k eigenvalues comparison}
	\P\left( \|(\lambda_1,\lambda_2,\dots,\lambda_{k})-(E_{(1)},E_{(2)},\dots,E_{(k)})\|\geq\eta\right)=o(1).
\end{equation} At this point it is worth noting Proposition \ref{aprop:Spacing for L+D matrix} follows from Proposition \ref{kprop:spacings}, \eqref{aeq:E control}, and \eqref{eq:A:k eigenvalues comparison}. From \eqref{eq:A:k eigenvalues comparison} we have that for any Borel set $B\subset\R^k$ \begin{equation}\label{eq:Joint dist comparison}
	\P\left(a_n\left((\lambda_1,\lambda_2,\dots,\lambda_{k})-b_n\right)\in B\right)=\P\left(a_n\left((D_{(1)},D_{(2)},\dots,D_{(k)})-b_n'\right)\in B\right)+o(1).
\end{equation} Theorem \ref{thm:main_L} then follows from the definition of $F_k$, \eqref{eq:normal_limit}, and \eqref{eq:Joint dist comparison} with the choice $B=\prod_{j=1}^{k}(-\infty,x_j]$.

	%To complete the proof of Theorem \ref{thm:main_L} (2) fix $k\in\N$ and consider the first largest $k$ eigenvalues $\lambda_1\geq\lambda_2\geq\dots\geq\lambda_k$. Using the notation above we have that \begin{equation}\label{eq:A:Eigenvalue location}
		%	\sup_{E\in\left[E_{(1)}-\frac{1}{a_n^{3/2}},E_{(1)}+\frac{1}{a_n^{3/2}} \right] }	\left|(1+o(1))I_1- \sum_{j=1}^{n}\frac{\eta^4(1+o(1))}{B_{j,n}}\right|=o(1)
		%\end{equation} asymptotically almost surely. We also have with probability tending to one that $E_{(2)}\notin \left[E_{(1)}-\frac{1}{a_n^{3/2}},E_{(1)}+\frac{1}{a_n^{3/2}} \right]$ and hence the second sum in \eqref{eq:A:Eigenvalue location} is $\frac{1}{2}+o(1)$. Now let $c_n$ be such that $c_n\eta=E_{(1)}-\lambda_1$. From the proof of Theorem \ref{thm:main_L} (1) above we know $c_n\eta\in\left[E_{(1)}-\frac{1}{a_n^{3/2}},E_{(1)}+\frac{1}{a_n^{3/2}} \right]$ asymptotically almost surely. At $\lambda_1$, $I_2\geq 1/2$ and the second sum of \eqref{eq:A:Eigenvalue location} is $\frac{1}{(1+c_n^2)(2+c_n^2)}+o(1)$, a contradiction unless $c_n=o(1)$. It then follows $|\lambda_1-E_{(1)}|=o(\eta)$ and $\lambda_2\notin\left[E_{(1)}-\frac{1}{a_n^{3/2}},E_{(1)}+\frac{1}{a_n^{3/2}} \right]$. A nearly identical argument will show that asymptotically almost surely $|\lambda_j-E_{(j)}|=o(\eta)$ for all $1\leq j\leq k$.

		\appendix
		\section{Limiting law for largest diagonal entry (by Santiago Arenas-Velilla and Victor P\'erez-Abreu)} \label{sec:appendix}
		This appendix is devoted to the following result, which describes the fluctuations of the largest diagonal entry of $\mathcal{L}_A$ when $A$ is drawn from the GOE.  The arguments presented here originally appeared in the preprint \cite{arenas2021extremal}.  
		\begin{theorem} 
			Let $A$ be drawn from the GOE.  Then the centered and rescaled largest diagonal entry of $\mathcal{L}_A$ 
			\[ a_n \left(\max_{1 \leq i \leq n} (\mathcal{L}_A)_{ii} - b_n' \right) \]
			converges in distribution as $n \rightarrow \infty$ to the standard Gumbel distribution,
			where $a_n$ is defined in \eqref{def:anbn} and $b_n'$ is defined in \eqref{eq:def:bn'}.  
		\end{theorem}
		
		\begin{proof}
			Let 
			\[
			v = (v_i)_{i=1}^n= \Big( (\mathcal{L}_A)_{11},  (\mathcal{L}_A)_{22}, \cdots, (\mathcal{L}_A)_{nn} \Big)^\mathrm{T},
			\]
			and note that $v$ is a multivariate normal random vector with mean zero.  
			The covariance matrix, $\Sigma$, associated to $v$ is of the form
			\[
			\Sigma_{ij} = \begin{cases}
				\frac{n-1}{n}, & \text{ if } i = j, \\
				\frac{1}{n}, & \text{ otherwise. }
			\end{cases}
			\] 
			The eigenvalues of $\Sigma$ are $\frac{n-2}{n}$ (with multiplicity $n-1$) and $\frac{2n-2}{n}$ (with multiplicity $1$).  Let 
			\[
			\Sigma = O \Lambda O^\mathrm{T}
			\] 
			be the spectral decomposition of $\Sigma$, where $O$ is an orthogonal matrix whose last column is the vector $\mathbf{e}$, defined in \eqref{eq:def:e}.   We observe that
			\[
			g = (g_i)_{i=1}^n := \Sigma^{-1/2} v 
			\]
			is a vector of iid standard Gaussian random variables.  Since
			\[
			\Sigma^{1/2}_{ij} = \sum_{k=1}^n O_{ik} \Lambda_{kk}^{1/2} O^\mathrm{T}_{kj} = \left( \sqrt{\frac{n-2}{n}} \sum_{k=1}^{n} O_{ik} O^\mathrm{T}_{kj} \right)  + \left(\sqrt{\frac{2n-2}{n}} - \sqrt{\frac{n-2}{n}} \right) O_{in} O^\mathrm{T}_{nj}.
			\]  
			we find that
			\[
			\Sigma^{1/2}_{ij} = \begin{cases}
				\sqrt{\frac{n-2}{n}} + \frac{1}{n} \left(\sqrt{\frac{2n-2}{n}} - \sqrt{\frac{n-2}{n}} \right),   & \text{if } i =j, \\
				\frac{1}{n} \left(\sqrt{\frac{2n-2}{n}} - \sqrt{\frac{n-2}{n}} \right),  & \text{otherwise}.
			\end{cases}
			\]
			Therefore, we have 
			\begin{align*}
				v_i &= \left(\Sigma^{1/2} g\right)_i \\
				&= \sqrt{\frac{n-2}{n}} g_i + \frac{1}{n} \left(\sqrt{\frac{2n-2}{n}} - \sqrt{\frac{n-2}{n}} \right) \sum_{k =1}^n g_k. 
			\end{align*}
			Taking a maximum over $i \in \{1, 2, \ldots, n\}$ yields
			\[
			\max_i v_i = \sqrt{\frac{n-2}{n}} \max_i g_i + \frac{1}{n} \left(\sqrt{\frac{2n-2}{n}} - \sqrt{\frac{n-2}{n}} \right) \sum_{k =1}^n g_k.
			\]
			After centering and scaling, we obtain that
			\begin{align*}
				a_n (\max_i v_i - b'_n) &= a_n(\max_i g_i - b'_n) + a_n\left(\left(\sqrt{\frac{n-2}{n}}-1\right) \max_i g_i \right) \\
				&\qquad + \frac{a_n}{n} \left(\sqrt{\frac{2n-2}{n}} - \sqrt{\frac{n-2}{n}} \right) \sum_{k=1}^n g_k.
			\end{align*}
			Both of the latter two terms on the right-hand side converge to zero in probability.  So by Slutsky's theorem, $a_n(\max_i v_i - b'_n)$ converges to the Gumbel distribution, as it is well-known (see, for instance,  \cite[Theorem 1.5.3]{MR691492}) that this is the limiting distribution of $a_n (\max_i g_i - b'_n)$.   
		\end{proof}

		\bibliography{laplacian11}
		\bibliographystyle{abbrv}

	\end{document}